\documentclass[11pt]{amsart}
\usepackage{tikz}
\usetikzlibrary{decorations.markings}

\usepackage{fullpage}
\usepackage{amssymb}
\usepackage{amsfonts}
\usepackage{amsmath}
\usepackage{mathrsfs} % script fonts
\usepackage{mathtools}
\usepackage{aliascnt}

\usepackage{enumerate}
\usepackage[shortlabels]{enumitem}

\usepackage{hyperref} 
\usepackage[noabbrev,capitalise]{cleveref} % Auto-reference theorems by name
\usepackage{url}

\theoremstyle{plain}
\newtheorem{theorem}{Theorem}[section]

\newaliascnt{corollary}{theorem}
\newtheorem{corollary}[corollary]{Corollary}
\aliascntresetthe{corollary}

\newaliascnt{proposition}{theorem}
\newtheorem{proposition}[proposition]{Proposition}
\aliascntresetthe{proposition}

\newaliascnt{lemma}{theorem}
\newtheorem{lemma}[lemma]{Lemma}
\aliascntresetthe{lemma}

\newaliascnt{conjecture}{theorem}

\aliascntresetthe{conjecture}

\newaliascnt{question}{theorem}

\aliascntresetthe{question}

\theoremstyle{definition}

\newaliascnt{definition}{theorem}
\newtheorem{definition}[definition]{Definition}
\aliascntresetthe{definition}

\newaliascnt{example}{theorem}

\aliascntresetthe{example}

\newaliascnt{remark}{theorem}
\newtheorem{remark}[remark]{Remark}
\aliascntresetthe{remark}

\newaliascnt{fact}{theorem}

\aliascntresetthe{fact}

\newaliascnt{claim}{theorem}

\aliascntresetthe{claim}

\crefname{theorem}{Theorem}{Theorems}
\Crefname{theorem}{Theorem}{Theorems}

\crefname{corollary}{Corollary}{Corollaries}
\Crefname{corollary}{Corollary}{Corollaries}

\crefname{proposition}{Proposition}{Propositions}
\Crefname{proposition}{Proposition}{Propositions}

\crefname{lemma}{Lemma}{Lemmas}
\Crefname{lemma}{Lemma}{Lemmas}

\crefname{conjecture}{Conjecture}{Conjectures}
\Crefname{conjecture}{Conjecture}{Conjectures}

\crefname{question}{Question}{Questions}
\Crefname{question}{Question}{Questions}

\crefname{definition}{Definition}{Definitions}
\Crefname{definition}{Definition}{Definitions}

\crefname{example}{Example}{Examples}
\Crefname{example}{Example}{Examples}

\crefname{remark}{Remark}{Remarks}
\Crefname{remark}{Remark}{Remarks}

\crefname{fact}{Fact}{Facts}
\Crefname{fact}{Fact}{Facts}

\crefname{claim}{Claim}{Claims}
\Crefname{claim}{Claim}{Claims}

\numberwithin{equation}{section}

\newcommand{\declarecommand}[1]{\providecommand{#1}{}\renewcommand{#1}}

\declarecommand{\R}{\mathbb{R}}
\declarecommand{\Q}{\mathbb{Q}}
\declarecommand{\Z}{\mathbb{Z}}
\declarecommand{\N}{\mathbb{N}}
\declarecommand{\C}{\mathbb{C}}
\declarecommand{\F}{\mathbb{F}}
\declarecommand{\E}{\mathbb{E}}
\declarecommand{\emptyset}{\varnothing}

\declarecommand{\Re}{\mathrm{Re}}
\declarecommand{\Im}{\mathrm{Im}}
\DeclarePairedDelimiter\abs{\lvert}{\rvert}

\renewcommand\abs[1]{\left\lvert#1\right\rvert}

\newcommand{\n}{\mathbf{n}}

\renewcommand{\phi}{\varphi}

\declarecommand{\ds}{\displaystyle}
\usepackage{lipsum}
\usepackage{xcolor}
\usepackage{bbm}

\newcommand{\msum}{\mathop{\sum\cdots\sum}}

\newcommand{\1}{\mathbbm{1}}

\renewcommand{\P}{\mathop{\mathbb{P}}}
\renewcommand{\E}{\mathop{\mathbb{E}}}

\makeatletter
\def\@tocline#1#2#3#4#5#6#7{\relax
  \ifnum #1>\c@tocdepth % then omit
  \else
    \par \addpenalty\@secpenalty\addvspace{#2}%
    \begingroup \hyphenpenalty\@M
    \@ifempty{#4}{%
      \@tempdima\csname r@tocindent\number#1\endcsname\relax
    }{%
      \@tempdima#4\relax
    }%
    \parindent\z@ \leftskip#3\relax \advance\leftskip\@tempdima\relax
    \rightskip\@pnumwidth plus4em \parfillskip-\@pnumwidth
    #5\leavevmode\hskip-\@tempdima
      \ifcase #1
       \or\or \hskip 1em \or \hskip 2em \else \hskip 3em \fi%
      #6\nobreak\relax
    \hfill\hbox to\@pnumwidth{\@tocpagenum{#7}}\par% <---- \dotfill -> \hfill
    \nobreak
    \endgroup
  \fi}
\makeatother

\title{Improved bounds for multiplicative functions in almost all short intervals}

\author{Siddarth Menon}
\address{
Siddarth Menon \\
Department of Pure Mathematics and Mathematical Statistics, Centre for Mathematical Sciences \\ 
University of Cambridge \\
Cambridge, CB3 0WB \\
United Kingdom}
\email{sm2884@cam.ac.uk}

\date{\today}

\begin{document}

\begin{abstract}
We refine the Matom\"aki--Radziwi\l{}\l{} method for short averages of multiplicative functions.
For the Liouville function and multiplicative functions supported on smooth numbers, we prove decay bounds that are essentially optimal with the Matom\"aki--Radziwi\l{}\l{} method.
The key new ingredient is a sharper treatment of the sieve error, achieved by introducing a more widely separated final prime range when restricting to integers with typical factorizations.
We additionally give a weaker but still improved bound for arbitrary $1$-bounded multiplicative functions and discuss some of the limitations.
As an application, we give an improved bound for the averaged Chowla conjecture of Matom\"aki--Radziwi\l{}\l{}--Tao that seems essentially best possible via their method. 
\end{abstract}

\maketitle

\tableofcontents % Uncomment if paper exceeds 30 pages

\section{Introduction}\label{sec:1}

In this paper, we study short partial sums of multiplicative functions $f$.
Remarkable progress in proving cancellation for these types of sums for general real-valued multiplicative functions (notably the Liouville function $\lambda$) was made by Matom\"aki and Radziwi\l{}\l{} in \cite{matomaki_radziwill}, and subsequently extended to complex-valued multiplicative functions in \cite{average_chowla}.
Specifically for the Liouville function, it is shown that
\begin{equation}{\label{eq:MR_error}}
    \frac{1}{X}\int_X^{2X} \abs{\frac{1}{h}\sum_{x\leq n\leq x+h}\lambda(n)}^2 dx \ll \frac{(\log\log h)^2}{(\log h)^2} + \frac{1}{(\log X)^c},
\end{equation}
with $c = \frac{1}{50}$. This is $o(1)$ if $h = h(X) \to\infty$ as $X \to \infty$, however slowly. 
The result can moreover be extended to real-valued multiplicative functions $f:\N\to [-1,1]$ via subtracting off the long average $\frac{1}{X}\sum_{X\leq n\leq 2X} f(n)$ inside the square in the integrand. 

The results of Matom\"aki and Radziwi\l{}\l{} had a number of groundbreaking consequences and led to significant improvements in Chowla's conjecture concerning average correlations of the form $\sum_{n\leq x}\lambda(n)\lambda(n+1)$ \cite{average_chowla, chowla_elliot_2, expansion, pilatte}.
There were additionally many further applications of their paper along with the follow-up paper \cite{mr2}. 
These include sign-changes of multiplicative functions, counting almost-primes in short intervals \cite{almost_almost, almost_almost_2, gaussian_almost}, the behavior of the prime divisor function $\omega$ in short intervals \cite{omega}, and Tao's resolution of the Erd\H{o}s discrepancy problem \cite{tao_discrep}.
Their paper additionally prompted generalizations from short intervals to short arithmetic progressions \cite{mult_in_aps}, and from bounded multiplicative functions to divisor bounded multiplicative functions \cite{mangerel_unbounded, zhou_unbounded, sun2024divisorboundedmultiplicativefunctions}. 
Moreover, specifically for the Liouville function, Chinis showed that under the Riemann hypothesis, one can obtain a substantially better error term \cite{chinis}.

Observe that for (relatively) ``larger'' scales of interval lengths, the error term in \eqref{eq:MR_error} is suboptimal; for $h > \exp\left((\log X)^c\right)$ the decay is limited by the second error term.
Here, we modify the arguments of \cite{matomaki_radziwill} and \cite{average_chowla} to prove an error term of the shape $(\log h)^{-2+o(1)}$, which performs better at these interval scales.
In fact, the proof has enough flexibility to twist $\lambda$ by a Dirichlet character $\chi$ with reasonably small moduli.

\begin{theorem}{\label{thm:liouville_bound}}
    Let $B > 0$ be fixed and let $X \geq h \geq 10$. Let $q\leq (\log X)^B$, and let $\chi$ be a Dirichlet character mod $q$.
    Then 
    $$\frac{1}{X} \int_X^{2X} \abs{\frac{1}{h}\sum_{x\leq n\leq x+h} \chi(n)\lambda(n)}^2 dx \ll_B \frac{(\log\log h)^2}{(\log h)^2} + \frac{(\log\log X)^4}{(\log X)^2}.$$
\end{theorem}

Some formulations of problems like Chowla's conjecture are given in terms of the M\"obius function $\mu(n)$ rather than $\lambda(n)$.
The results proven here are completely analogous and carry over with essentially no extra effort.

One can ask for similar bounds in the case of general \emph{non-pretentious} multiplicative functions $f$.
Hal\'asz' theorem says that a multiplicative function should have small mean value, provided that it is not close to any Archimedean character $n^{it}$. This is in the sense that the \emph{pretentious distance}
$$\mathbb{D}(f,n\mapsto n^{it};X)^2 = \sum_{p\leq X} \frac{1-\Re (f(p)p^{-it})}{p}$$
is large for all $t$ in some appropriate range.
In what follows, define 
$$M(f;X) = \min_{|t|\leq X} \mathbb{D}(f, n\mapsto n^{it};X)^2.$$

In \cite{average_chowla} the authors prove a bound of the form 
$$\frac{1}{X}\int_X^{2X}\abs{\frac{1}{h}\sum_{x\leq n\leq x+h} f(n)}^2 dx \ll \exp(-M(f;X)) + \frac{(\log\log h)^2}{(\log h)^2} + \frac{1}{(\log X)^{\frac{1}{50}}}.$$

Here, this third error term limits the decay for larger interval lengths (in the same sense as the previous case) as well as sufficiently non-pretentious functions (when $M(f;X) \geq \frac{1}{49}  \log\log X$, say). 
Our proof technique for \cref{thm:liouville_bound} extends naturally to general complex-valued \emph{non-pretentious} multiplicative functions $f$, provided that $f$ is supported on $X^{1-\delta}$-smooth numbers (i.e. $f(n) = 0$ when $n$ has a prime factor larger than $X^{1-\delta}$), where $\delta > 0$ is an arbitrarily small constant.
In this setting, we can additionally improve the dependence on the pretentious distance.

\begin{theorem}{\label{thm:smooth_f_thm}}
    Let $X \geq h \geq 10$ and fix some $0 < \delta < 1$.
    Let $f:\N\to \C$ be a 1-bounded multiplicative function, supported on $X^{1-\delta}$-smooth numbers.
    Then
    $$\frac{1}{X} \int_X^{2X} \abs{\frac{1}{h}\sum_{x\leq n\leq x+h} f(n)}^2 dx \ll_{\delta} \exp(-(2-o(1))M(f;X)) + \frac{(\log\log h)^2}{(\log h)^2} + \frac{1}{(\log X)^{2-o(1)}}.$$
\end{theorem}
Here, the $o(1)$ in the exponential is as $M(f;X) \to\infty$ with $X$. If $M(f;X)$ were bounded, the prefactor would not necessarily converge to 2, although in this case the entire first term is bounded by some constant anyways. 

Although we cannot prove a bound of the same strength for arbitrary multiplicative functions, we can prove an improved estimate that saves essentially one power of $\log X$ rather than two, and $\exp(-(1-o(1))M(f;X))$ rather than its square in the previous theorem.
\begin{theorem}{\label{thm:general_f}}
    Let $X \geq h \geq 10$.
    Let $f: \N\to \C$ be a 1-bounded multiplicative function.
    Then
    $$\frac{1}{X} \int_X^{2X} \abs{\frac{1}{h}\sum_{x\leq n\leq x+h} f(n)}^2 dx \ll \exp(-(1-o(1))M(f;X)) + \frac{(\log\log h)^2}{(\log h)^2} + \frac{1}{(\log X)^{1-o(1)}}.$$
\end{theorem}

The same remark that follows \cref{thm:smooth_f_thm}, concerning the $o(1)$ in the exponential, holds here.  

As an application of their short interval bounds for $\lambda(n)$, Matom\"aki--Radziwi\l{}\l{}--Tao proved the following results concerning a short interval exponential sum in $\lambda$ as well as an averaged form of Chowla's conjecture:
\begin{equation}{\label{eq:MRT_exp_sum}}
    \sup_{\alpha\in \R} \frac{1}{X}\int_0^X \abs{ \frac{1}{h} \sum_{x \leq n \leq x+h} \lambda(n) e(\alpha n)} dx \ll \frac{\log\log h}{\log h} + \frac{1}{(\log X)^{\frac{1}{700}}},
\end{equation}
and 
\begin{equation}{\label{eq:MRT_avg_chowla}}
    \frac{1}{H^{k-1}} \sum_{1\leq h_2, \dots, h_k \leq H} \abs{ \frac{1}{X}\sum_{1\leq n \leq X} \lambda(n)\lambda(n+h_2)\dots \lambda(n+h_k)} \ll k\left(\frac{\log\log H}{\log H} + \frac{1}{(\log X)^{\frac{1}{3000}}}\right).
\end{equation}

Equation \eqref{eq:MRT_exp_sum} is proved via the circle method, and implies \eqref{eq:MRT_avg_chowla} via a Fourier identity.
The improved bounds for $\lambda$ twisted by Dirichlet characters of moduli $q\leq (\log X)^B$ will enable us to increase the major/minor arc cutoff. 
In practice this enables us to prove the following:

\begin{theorem}{\label{thm:improved_exp_sum}}
    For any $X \geq h \geq 10$, one has 
    $$\sup_{\alpha\in \R} \frac{1}{X}\int_0^X \abs{\frac{1}{h}\sum_{x\leq n\leq x+h} \lambda(n)e(\alpha n)} dx \ll \frac{\log\log h}{\log h} + \frac{(\log\log X)^2}{\log X}.$$
\end{theorem}

\begin{theorem}{\label{thm:improved_avg_chowla}}
    For any natural number $k\geq 2$ and any $X \geq H \geq 10$, we have 
    $$\frac{1}{H^{k-1}}\sum_{1\leq h_2, \dots, h_k \leq H} \abs{\frac{1}{X}\sum_{1\leq n\leq X} \lambda(n)\lambda(n+h_2) \dots \lambda(n+h_k)} \ll k\left(\frac{\log\log H}{\log H} + \frac{(\log\log X)^2}{\log X}\right).$$
\end{theorem}

In \cite{average_chowla}, the authors state that their framework cannot show decay much larger than $\frac{1}{\log H}$, as this involves controlling $\lambda(n)$ on numbers $n$ not divisible by any prime less than $H$, which means that we gain nothing by averaging over shifts. The bound in \cref{thm:improved_avg_chowla} (essentially) attains this threshold value, and so seems to be the best one could hope for via the Matom\"aki--Radziwi\l{}\l{} method. 

\subsection{Outline of proofs}

We follow the general strategy of Matom\"aki--Radziwi{\l}{\l}, reducing short interval estimates to mean-square estimates for Dirichlet polynomials on the line $\operatorname{Re}(s)=1$.
As in \cite{matomaki_radziwill}, we first restrict to integers with a typical factorization, namely integers having prime factors in a sequence of intervals $[P_j,Q_j]$.
The main new point is that we modify the final prime range, which is much more widely separated: roughly, it begins at $X^{o(1)}$ and extends up to $X$.
Standard estimates for smooth numbers then show that the integers with no prime factor in this final range have density $\ll_R (\log X)^{-R+o(1)}$, where we may choose $R > 2$. This is the source of the improved sieve error. We discuss this, and also present some necessary bounds for special Dirichlet polynomials, in \cref{sec:2}.

In \cref{sec:3} we give a proof of \cref{thm:liouville_bound}. After passing to the Dirichlet polynomial side, the contribution from the first set of prime ranges $[P_j,Q_j]$ is treated by the same $L^2$--$L^\infty$ approach as in \cite{matomaki_radziwill}.
On the set of frequencies for which one of the corresponding
prime-supported Dirichlet polynomials is small, this gives the required saving.
It remains to handle an exceptional set of frequencies on which all of these prime polynomials are large. 
Here we use a different and more separated range of primes $[P,Q]$ in the Ramar\'e identity. In the case of
$\chi(n)\lambda(n)$, we apply a Ramar\'e-type decomposition along this final prime range, and then use Heath-Brown's identity to decompose the resulting prime-supported Dirichlet polynomial.

In \cref{sec:4} we study more general $1$-bounded complex-valued multiplicative functions. The same factorization argument is combined with a Hal\'asz--Montgomery estimate for Dirichlet polynomials supported on primes and Hal\'asz' theorem.
In this part of the argument, three inputs differ most clearly from the corresponding steps in \cite{matomaki_radziwill}. 
Firstly, we prove a variant of the Ramar\'e decomposition for Dirichlet polynomials using Perron's formula to remove the cross condition $pm \leq X$, which ensures that the remaining Dirichlet polynomials are genuinely dyadic. 
Secondly, we give a Hal\'asz--Montgomery estimate for long prime-supported Dirichlet polynomials which allows for a much larger discrete set of frequencies.
Lastly, we give a Hal\'asz-type bound for the remainder Dirichlet polynomial arising from Ramar\'e's identity which gives the same main term and tighter error bounds.
When $f$ is supported on $X^{1-\delta}$-smooth numbers, the contribution of integers of the form $n=pm$ with $p$ very large and $m$ very small is absent, and the final prime range gives both the $(\log X)^{-2+o(1)}$ saving and the stronger dependence $\exp(-(2-o(1))M(f;X))$ in \cref{thm:smooth_f_thm}.
For arbitrary multiplicative functions
this large prime obstruction cannot be removed with the same strength; choosing weaker parameters instead gives \cref{thm:general_f}. 

In \cref{sec:5}, we apply \cref{thm:liouville_bound} to the averaged Chowla problem via the
circle method of Matom\"aki--Radziwi{\l}{\l}--Tao \cite{average_chowla}.
The improvement needed there is the uniformity in Dirichlet twists $\chi(n)\lambda(n)$ for moduli $q\leq(\log X)^B$.
This allows the major/minor arc cutoff to be increased to a sufficiently large power of $\log X$, and hence to a power of $\log h$ throughout the full range $10\leq h\leq X$.
The minor arc argument is unchanged from \cite{average_chowla}; the improved short interval estimate is used on the major arcs. This yields \cref{thm:improved_exp_sum} and \cref{thm:improved_avg_chowla}.

\subsection{Notation}
We write $\mathbb{P}$ for the set of primes, and unless otherwise specified the letters $p,p_1,p_2, \dots$ denote primes. We use $\1_A$ for the indicator function of a set or condition $A$.

For $n$ a positive integer, $\tau(n)$ will denote the number of divisors of $n$ and $\tau_{k}(n)$ will denote the $k$-fold divisor function. Also $\Omega(n)$ and $\omega(n)$ will denote the number of prime factors of $n$ counted with and without multiplicity (respectively).
We also write $e(\alpha) = e^{2\pi i \alpha}$. 

We use the normal Vinogradov $\ll, \gg$ notation and normal big $O$ and little $o$ notations, with subscripts indicating
the parameters on which the implicit constant is allowed to depend. 

Additionally, we follow the convention that the length of a Dirichlet polynomial is denoted by the same capital letter as the polynomial itself. For example, we would refer to the Dirichlet polynomial $\sum_{D\leq n\leq 2D} d_n n^{-s}$ by $D(s)$, and refer to its length as $D$. 

For the remainder of the paper, fix once and for all the constants
\begin{equation}{\label{eq:fixed_constants}}
    \kappa = \frac{1}{2^{25}-2} \qquad \nu = 1-\frac{24}{2^{25}-2}.
\end{equation}

\section{Preliminaries and lemmas}\label{sec:2}
As in \cite{matomaki_radziwill}, the strategy will proceed by using a suitable analog of Parseval's identity to go to the Dirichlet polynomial (frequency) side.
We will first collect necessary lemmas concerning $L^2$ estimates for Dirichlet polynomials, then present some necessary pointwise bounds, and finally the decomposition theorems we use to give a proof of the main theorem.

\subsection{Lemmas concerning Dirichlet polynomials}\label{sec:2.1}

We first need continuous and discrete mean value theorems for Dirichlet polynomials. 
\begin{lemma}[Mean value theorem]{\label{lem:mvt}}
  Let $N, T \geq 10$ and let $(a_n)$ be a sequence of complex numbers. Let $A(s) = \sum_{n\leq N} a_n n^{-s}$.
  Then
  $$\int_{-T}^T \abs{A(it)}^2 dt \ll (T+N)\sum_{n\leq N}|a_n|^2.$$
\end{lemma}
\begin{proof}
    See Theorem 9.1 of \cite{iwaniec_kowalski}.
\end{proof}

We will additionally need the following variants of the mean value theorem.
\begin{lemma}[Improved mean value theorem]{\label{lem:IMVT}}
    Let $N, T \geq 10$ and let $(a_n)$ be a sequence of complex numbers. 
     Let $A(s) = \sum_{n\leq N} a_n n^{-s}$.
    Then 
    $$\int_{-T}^T \abs{A(it)}^2 dt \ll T\sum_{n\leq N}|a_n|^2 + T\sum_{1\leq \ell \leq \frac{N}{T}} \sum_{n\leq N}|a_n||a_{n+\ell}|.$$
\end{lemma}
\begin{proof}
    See Lemma 7.1 of \cite{iwaniec_kowalski}.
\end{proof}

\begin{lemma}[Hal\'asz--Montgomery]{\label{lem:hm}}
  Let $N, T \geq 10$ and let $(a_n)$ be a sequence of complex numbers. Let $A(s) = \sum_{n\leq N} a_n n^{-s}$ and $\mathcal{T}\subseteq [-T, T]$ be a sequence of well-spaced points (well-spaced in the sense that for $t, t'\in \mathcal{T}$, $|t-t'| \geq 1)$.
  Then
  $$\sum_{t\in \mathcal{T}}\abs{A(it)}^2 \ll (N+|\mathcal{T}| T^{\frac{1}{2}})\log(T) \sum_{n\leq N}|a_n|^2.$$
\end{lemma}
\begin{proof}
    See Theorem 9.4 of \cite{iwaniec_kowalski}.
\end{proof}

In general, we would like the diagonal contribution to be the dominant term, so $N > \abs{\mathcal{T}}T^{\frac{1}{2}}$. Even for $\abs{\mathcal{T}}$ quite small (suppose $\abs{\mathcal{T}} = o(T^{\epsilon})$ for $\epsilon > 0$ small), we would require $N > T^{\frac{1}{2}+\epsilon}$. To gain a bit of flexibility beyond this barrier, we give a variant of the Hal\'asz--Montgomery inequality using exponent pairs.

\begin{lemma}{\label{lem:exp_pair_hm}}
  Let $N, T \geq 10$ and let $(a_n)$ be a sequence of complex numbers. Let $A(s) = \sum_{N\leq n\leq 2N} a_n n^{-s}$ and $\mathcal{T}\subseteq [-T, T]$ be a sequence of well-spaced points (well-spaced in the sense that for $t, t'\in \mathcal{T}$, $|t-t'| \geq 1)$.
  Then,
  $$\sum_{t\in \mathcal{T}}\abs{A(it)}^2 \ll \left(N\log N + \abs{\mathcal{T}}T^{\frac{1}{6}}N^{\frac{1}{2}}\right)\sum_{N\leq n\leq 2N} |a_n|^2.$$
\end{lemma}
\begin{proof}
    By the duality principle (see Chapter 7, Theorem 6 in \cite{ten_lecs}), it is enough to prove that 
    \begin{equation}{\label{eq:duality_sufficient}}
        \sum_{N\leq n\leq 2N}\abs{\sum_{t\in \mathcal{T}} b_t n^{it}}^2 \ll \left(N\log N + \abs{\mathcal{T}}T^{\frac{1}{6}}N^{\frac{1}{2}}\right)\sum_{t\in\mathcal{T}} |b_t|^2
    \end{equation}
    for every sequence of coefficients $(b_t)_{t\in \mathcal{T}}$. We open up the square on the left-hand side of \eqref{eq:duality_sufficient} and interchange the sums to obtain
    \begin{align*}
        \sum_{N\leq n\leq 2N}\abs{\sum_{t\in \mathcal{T}} b_t n^{it}}^2 & = \sum_{t_1, t_2\in \mathcal{T}}b_{t_1}\overline{b_{t_2}}\sum_{N\leq n\leq 2N} n^{i(t_1-t_2)}\\
        &\ll \underbrace{\max_{t_2\in\mathcal{T}} \sum_{t_1\in \mathcal{T}} \abs{\sum_{N\leq n\leq 2N}n^{i(t_1-t_2)}}}_{\Sigma} \sum_{t\in \mathcal{T}}\abs{b_{t}}^2,
    \end{align*}
    where in the second line we have applied the inequality $|b_{t_1}\overline{b_{t_2}}| \ll |b_{t_1}|^2 + |b_{t_2}|^2$, used the symmetry in $t_1, t_2$ of $n^{i(t_1-t_2)}$, and used the triangle inequality to pull out a maximum over $t_2$. It will be enough to prove that $\Sigma \ll \left(N\log N + \abs{\mathcal{T}}T^{\frac{1}{6}}N^{\frac{1}{2}}\right)$. Now, in $\Sigma$, fix the value of $t_2$ attaining the maximum. 
    
    Observe that the partial sum of zeta can be interpreted as an exponential sum with phase $\psi(n) = -\frac{t_1-t_2}{2\pi}\log(n)$. The interval has length $N$ and the derivative of the phase satisfies 
    $$\abs{\psi'(n)} = \frac{1}{2\pi}\frac{|t_1-t_2|}{n} \asymp \frac{|t_1-t_2|}{N}.$$

    With this derivative estimate, we will split into cases depending on the range of $t_1$ relative to the fixed $t_2$. 
    
    \textbf{Case I: } Suppose $t_1 = t_2$, in which case the contribution is $O(N)$. 

    \textbf{Case II: } Suppose that $1 \leq |t_1 - t_2| \leq cN$, for an absolute constant $c > 0$. Here, $c$ is chosen sufficiently small so that the nearest integer to $\abs{\psi'(n)}$ is 0. Hence, by the Kusmin--Landau theorem (Theorem 2.1 of \cite{gk91}), 
    $$\sum_{N\leq n\leq 2N} e(\psi(n)) \ll \frac{N}{|t_1-t_2|}.$$ 
    
    Note that because the set $\mathcal{T}$ is well-spaced and $t_2$ is fixed, for any integer $1\leq m\leq cN$, there is a bounded number of $t_1$ for which $\lfloor \abs{t_1-t_2}\rfloor = m$. 
    Hence, we can bound the contribution of these $t_1$ to $\Sigma$ by
    $$\ll \sum_{1\leq m\leq cN} \frac{N}{m} \ll N\log N.$$

    \textbf{Case III: } Lastly suppose $cN < |t_1 - t_2| \leq 2T$. Again interpret the sum over $n$ as an exponential sum, and observe that because we have $\abs{\psi'(n)} \gg 1$, we may use the fact that $\left(\frac{1}{6}, \frac{2}{3}\right)$ is an exponent pair (see Chapter 3 of \cite{gk91}), to bound 
    \begin{align*}
        \sum_{N\leq n\leq 2N} e(\psi(n)) &\ll \left(\frac{|t_1-t_2|}{N}\right)^{\frac{1}{6}} N^{\frac{2}{3}} \ll T^{\frac{1}{6}}N^{\frac{1}{2}}.
    \end{align*}
    Crudely bounding the number of such $t_1$ by all of $\mathcal{T}$, we obtain the bound 
    $$\ll \abs{\mathcal{T}} T^{\frac{1}{6}}N^{\frac{1}{2}}.$$
    Hence \eqref{eq:duality_sufficient} holds.
\end{proof}

We will also collect some results concerning Dirichlet polynomials supported on primes, including some $L^2$ bounds and how often such polynomials can assume large values.

\begin{lemma}{\label{lem:prime_hm}}
  Let $P, T \geq 10$ and let $(a_p)$ be a sequence of complex numbers supported on primes. Let $P(s) = \sum_{P\leq p \leq 2P} \frac{a_p}{p^s}$ and let $\mathcal{T}\subseteq [-T, T]$ be a well-spaced set.
  Then 
  $$\sum_{t\in \mathcal{T}}\abs{P(it)}^2 \ll \left(\frac{P}{\log P} + \abs{\mathcal{T}} T^{\frac{1}{2}} P^{\frac{1}{20}} \log T\right)\sum_{P \leq p \leq 2P}|a_p|^2.$$
\end{lemma}

\begin{proof}
  By the duality principle, it suffices to prove that for an arbitrary sequence of complex numbers $(b_t)_{t\in\mathcal{T}}$, we have 
  \begin{equation}{\label{eq:duality_primes}}
      \sum_{P \leq p \leq 2P} \abs{\sum_{t\in \mathcal{T}} b_t p^{it}}^2 \ll \left(\frac{P}{\log P} + \abs{\mathcal{T}} T^{\frac{1}{2}} P^{\frac{1}{20}} \log T\right)\sum_{t\in \mathcal{T}}|b_t|^2.
  \end{equation}
  The first thing we do is to set up our sieve with the specified parameters.
  In the relevant dyadic range $P \leq n\leq 2P$ we may majorize the prime indicator function $\1_{\P}(n) \leq v(n)$, where
  $$v(n) = \sum_{\substack{d|(n, \Pi(Z))\\d\leq D}} \lambda_d^+$$
  with appropriate sieve weights (we refer to the fundamental lemma 6.3 in \cite{iwaniec_kowalski} for a construction of these weights). Here, $\Pi(Z) = \prod_{p\leq Z} p$, and we take $D = P^{\frac{1}{20}}$ and $Z = P^{\frac{1}{40}}$.
  Furthermore, we upper bound the dyadic support condition in \eqref{eq:duality_primes} by $g\left(\frac{p}{P}\right)$, where $g$ is a smooth function, compactly supported in $[0.9, 2.1]$ and $g(x) = 1$ whenever $x\in [1,2]$.
  We apply the sieve majorant to get 
  $$\sum_{P \leq p \leq 2P}\abs{\sum_{t\in \mathcal{T}} b_t p^{it}}^2 \leq \sum_n v(n) g\left(\frac{n}{P}\right) \abs{\sum_{t\in \mathcal{T}} b_t n^{it}}^2$$
  
  We can now open up the square, and exchange sums to bound the above by
  \begin{equation}{\label{eq:t1_t2_sum}}
      \ll \sum_{t_1, t_2\in \mathcal{T}} \abs{b_{t_1}b_{t_2}} \underbrace{\abs{\sum_n g\left(\frac{n}{P}\right) n^{i(t_2-t_1)} \sum_{\substack{d|(n, \Pi(Z))\\d\leq D}} \lambda_d^+}}_{\abs{Q(i(t_2-t_1))}}.
  \end{equation}
    Write $t = t_2-t_1$ and consider the entire polynomial over $n=dk$ as a function of one variable $t$.
  We swap the $n$ and $d$ sums and get that
  \begin{equation}{\label{eq:def_Q}}
      Q(it) = \sum_{d\leq P^{\frac{1}{20}}} \lambda_d^+ d^{-it}\sum_{k} g\left(\frac{k}{(P/d)}\right)k^{-it}
  \end{equation}
  We know that for smooth and compactly supported $g$, this inner sum over $k$ is bounded by estimates for the $\zeta$ function.
  We use Mellin inversion on $g$ to write
  \begin{align*}
      \sum_k g\left(\frac{k}{(P/d)}\right) k^{-it} &= \frac{1}{2\pi i}\int_{\Re(s)=2}\widetilde{g}(s)\left(\frac{P}{d}\right)^s\sum_k k^{-s-it}ds\\
      &= \frac{1}{2\pi i}\int_{\Re(s)=2} \widetilde{g}(s)\left(\frac{P}{d}\right)^s \zeta(s+it)ds.
  \end{align*}
  
    We shift the contour from $\Re(s)=2$ to $\Re(s)=0$ picking up the simple pole of $\zeta$ at $1-it$. We apply the usual convexity bound to the zeta function along the line $\Re(s)=0$ and use rapid decay of $\widetilde{g}(s)$ to compute the integral as 
  $$\sum_{k} g\left(\frac{k}{(P/d)}\right)k^{-it} = \widetilde{g}(1-it) \left(\frac{P}{d}\right)^{1-it}+ O\left( (2+|t|)^{\frac{1}{2}}\log(2+|t|)\right).$$
    Plugging this back into the definition \eqref{eq:def_Q}, we get cancellation of the $d^{-it}$ phases, and get an upper bound on $\abs{Q(it)}$ of the form
  $$P^{1} \abs{\widetilde{g}(1-it) \sum_{d\leq P^{\frac{1}{20}}}\frac{\lambda_d^+}{d}} + P^{\frac{1}{20}} T^{\frac{1}{2}}\log T.$$
  Again by the fundamental lemma of the sieve, we have
  \begin{equation}{\label{eq:linear_sieve}}
      \sum_{d\leq P^{\frac{1}{20}}}\frac{\lambda_d^+}{d} \ll \frac{1}{\log Z} \ll \frac{1}{\log P}.
  \end{equation}
  
  We then plug these bounds into the sum over $t_1, t_2$ in \eqref{eq:t1_t2_sum}. Using the inequality $\abs{b_{t_1}b_{t_2}}\leq \abs{b_{t_1}}^2+\abs{b_{t_2}}^2$, we may trivially bound the second term. Lastly, note that $\sum_{t\in \mathcal{T}}\abs{\widetilde{g}(1-i(t_2-t_1))}$ will be summable off of the diagonal contribution $t_1 = t_2$. We may trivially bound the diagonal contribution (this is admissible), and for the off-diagonal pairs $(t_1, t_2)$, precisely because compact support of $g$ ensures that $\widetilde{g}(1-i(t_2-t_1))$ decays faster than, say, $(1+|t_2-t_1|)^{-100}$, and because $\mathcal{T}$ is well-separated, we recover the desired bound.
\end{proof}
\begin{remark}
    The $P^{\frac{1}{20}}$ parameter is sufficient here, and is made explicit to avoid introducing additional small parameters, but for other applications one should note that this power can be changed at the cost of the implicit constant in \eqref{eq:linear_sieve}. 
\end{remark}

\begin{lemma}{\label{lem:prime_hm_2}}
  Let $P, T \geq 10$ and let $(a_p)$ be a sequence of complex numbers supported on primes. Let $P(s) = \sum_{P\leq p \leq 2P} \frac{a_p}{p^s}$ and let $\mathcal{T}\subseteq [-T, T]$ be a well-spaced set.
  Then for every $\epsilon > 0$,
  $$\sum_{t\in \mathcal{T}}\abs{P(it)}^2 \ll_{\epsilon} \left(P + \abs{\mathcal{T}}P \exp\left(-\frac{\log P}{(\log T)^{\frac{2}{3}+\epsilon}}\right)(\log T)^2\right)\sum_{P\leq p\leq 2P}\frac{\abs{a_p}^2}{\log P}.$$
\end{lemma}
\begin{proof}
    See Lemma 11 of \cite{matomaki_radziwill}.
\end{proof}

The difference between \cref{lem:prime_hm} and \cref{lem:prime_hm_2} is the behavior of the second (off-diagonal) term. The first lemma is limited in that on the 1-line, it only gives savings for $P \gg T^{0.55+\epsilon}$, say, and the second is limited in the sense that the exceptional set $\mathcal{T}$ must be very small in order to give any savings.

\begin{lemma}{\label{lem:large_values}}
    Let $P, T \geq 10$ and let $(a_p)$ be a sequence of complex numbers supported on primes, with $|a_p|\leq 1$. Let $P(s) = \sum_{P\leq p \leq 2P}\frac{a_p}{p^s}$. 
    Let $\mathcal{T}\subseteq [-T,T]$ be a sequence of well-spaced points such that $\abs{P(1+it)} \geq V^{-1}$ for every $t\in \mathcal{T}$.
    Then 
    $$\abs{\mathcal{T}} \ll T^{2\frac{\log V}{\log P}} V^2 \exp\left(2\frac{\log T}{\log P} \log\log T\right).$$
\end{lemma}
\begin{proof}
    See Lemma 8 of \cite{matomaki_radziwill}.
\end{proof}

In \cref{sec:3} we will additionally need $L^{\infty}$ bounds for some special Dirichlet polynomials associated to $\1_{\P}(n)$, $\mu(n)$, and the constant function 1 along the line $s=1+it$. 

\begin{lemma}\label{lem:pnt}
    Let $A, B > 0$ be given and let $X \geq 10$. Let $q\leq (\log X)^B$ and let $\chi$ be a Dirichlet character mod $q$.
    Let $P > 0$ be such that $\exp\left((\log X)^{\frac{2}{3}+\epsilon}\right) \leq P \leq X$, for fixed $\epsilon > 0$.
    Then for any $|t|\leq X$,
    $$\sum_{P\leq p\leq 2P}\frac{\chi(p)}{p^{1+it}} \ll_{A, B} \frac{\log X}{1+|t|} + (\log X)^{-A}.$$
\end{lemma}
We follow the strategy in Lemma 2 of \cite{matomäki2015noteliouvillefunctionshort}. 
\begin{proof}
    Supposing that $|t| \leq 100$, the bound is immediate by the    triangle inequality. Here and later, define $\log L(w, \chi)$ to be the branch of the logarithm that agrees with
    $$\log L(w,\chi) = \sum_p\sum_{k\geq 1}\frac{\chi(p)^k}{k p^{kw}}.$$
    
    If $|t| > 100$, then we employ the same strategy as in Matom\"aki--Radziwi\l{}\l{}: again apply Perron's formula with $\sigma = \frac{1}{\log X}$ to obtain
    $$\sum_{P\leq p\leq 2P} \chi(p) p^{-1-it} = \frac{1}{2\pi i}\int_{\sigma-iT}^{\sigma+iT} \log L(1 + s + it, \chi) \frac{(2P)^s - P^s}{s} ds + O\left(\frac{\log X}{1+|t|}\right),$$
    taking $T = \frac{(|t|+1)}{2} < |t|-1$. Strictly speaking, applying Perron's formula to $\log L$ also incurs an error term of $O\left(\frac{1}{P}\right)$ coming from the contribution of higher prime powers. 
    However, this is certainly absorbed by the $O_A((\log X)^{-A})$ we will shortly incur, since $P\geq \exp\left((\log X\right)^{2/3+\epsilon})$.

    Crucially we know now that the imaginary part of $1+s+it$ is greater than 1 in absolute value, and hence we are far enough away from any possible exceptional zero (and in the case $\chi$ is principal, the pole of $L(s, \chi)$ at $s=1$).
    We pull this contour to the edge of the Vinogradov--Korobov zero free region (see Montgomery \cite{ten_lecs}, Chapter 9), to 
    $$\Re(s) = \sigma_0 = -\frac{c}{\log q + (\log X)^{\frac{2}{3}}(\log\log X)^{\frac{1}{3}}},$$
    where $c > 0$ is some small constant. 
    Note that because $q \leq (\log X)^B$ this zero-free region is essentially as strong as it is for the Riemann zeta function (since $\log q \ll_B \log\log X$).
    For $s$ such that $\Re(s)$ is in the zero-free region and $1\leq \abs{\Im(s) + t} \leq 2X$, we can bound $\abs{\log L(1+s+it, \chi)} \ll_B (\log X)^2$ (see Theorem 11.4 in \cite{montgomery_vaughan}). 
    Because $\epsilon > 0$, we know that $P^{\sigma_0} \ll_{A, B} (\log X)^{-A}$, and hence we conclude that 
    $$\sum_{P\leq p\leq 2P} \chi(p) p^{-1-it} \ll_{A, B} \frac{\log X}{1+|t|} + (\log X)^{-A}.$$
\end{proof}

\begin{lemma}{\label{lem:pnt_mobius}}
    Fix $A, B > 0$ and let $X\geq 10$. Let $q\leq (\log X)^B$  and let $\chi$ be a Dirichlet character mod $q$. For $|t| \leq X^{24}$, we have 
    $$\sum_{X \leq n\leq 2X} \frac{\chi(n)\mu(n)}{n^{1+it}} \ll_{A, B} (\log X)^{-A}.$$
\end{lemma}
The proof is essentially exactly as for the Siegel-Walfisz theorem (and consequently the constant here is ineffective).
\begin{proof}
    By Perron's formula, 
    $$\sum_{X \leq n\leq 2X}\mu(n) \chi(n) n^{-it} = \frac{1}{2\pi i}\int_{\sigma-iT}^{\sigma+iT} \frac{1}{L(s + it, \chi)} \frac{(2X)^{s} - X^s}{s}ds + O\left(\frac{X (\log X)^2}{T}\right)$$
    for $\sigma = 1+\frac{1}{\log X}$.
    We pull the contour to the edge of the Vinogradov--Korobov zero free region for $L$-functions.
    Note that $|\Im(s+it)| \ll T+X^{24}$, and choose $T = \exp\left((\log X)^{\frac{1}{2}}\right)$, say. By our bound for $q$, the zero-free region has width $\gg_B \frac{1}{(\log X)^{\frac{2}{3}}(\log\log X)^{\frac{1}{3}}}$, which is enough. Because $|s + it - \beta| \geq \frac{1}{\log q}$, we know $\abs{\frac{1}{L(s+it, \chi)}} \ll \log (q\abs{\Im(s+it)+2}) \ll \log X$ (again see Theorem 11.4 in \cite{montgomery_vaughan}). 
    If $L(s, \chi)$ has no exceptional zero or if $|t| > T$ (and hence the contour misses the potential exceptional zero) we conclude that the above is 
    $$\ll_{A,B} X(\log X)^{-A}.$$
    If $|t|\leq T$ and $L(s, \chi)$ has an exceptional zero $\beta$, then we pull the contour back to the line with real part $\beta + \frac{1}{\log X}$.
    The main contribution will come from
    $(2X)^{\beta} - X^{\beta}$.
    Recall that via Siegel's theorem (see Chapter 21 of \cite{davenport}), any real zero $\beta$ of $L(s, \chi)$ for real nonprincipal $\chi$ satisfies 
    $$\beta \leq 1-c(\epsilon) q^{-\epsilon}$$
    for any $\epsilon > 0$ (and $c(\epsilon)$ is some ineffective constant depending on $\epsilon$).   
    Because $q$ is in the admissible Siegel-Walfisz range, the $(2X)^{\beta} - X^{\beta}$ term contributes at most $\ll_{A, B} X(\log X)^{-A}$.
    As the contour is far enough away from the exceptional zero, the bound for $L(s+it, \chi)^{-1}$ still holds.
    We then obtain the desired bound on the 1-line by partial summation and the fact $\frac{1}{n}\asymp \frac{1}{X}$. 
\end{proof}

Recall that we fixed the constants $\kappa$ and $\nu$ in \eqref{eq:fixed_constants}.
\begin{lemma}{\label{lem:exponent_pairs}}
Fix $B > 0$ and let $X \geq 10$. Let $q \leq (\log X)^B$, and let $\chi$ be a Dirichlet character modulo $q$.
    For all $2 \leq |t| \leq X^{24}$, 
    $$\sum_{X\leq n\leq 2X} \chi(n) n^{-it} \ll |t|^{\kappa} X^{\nu-\kappa + o(1)} + \frac{X}{1+|t|}.$$
\end{lemma}

\begin{proof}
We can split the sum as
\begin{align}{\label{eq:dirichlet_periodicity}}
    \sum_{X\leq n\leq 2X} \chi(n) n^{-it} &= \sum_{a \pmod{q}} \chi(a)\sum_{\substack{X\leq n\leq 2X\\ n\equiv a \pmod{q}}} n^{-it} \nonumber\\
    &= \sum_{a \pmod{q}} \chi(a) \sum_{\frac{X-a}{q}\leq m \leq \frac{2X-a}{q}} (a+mq)^{-it},
\end{align}
and now, for each fixed $a$, the inner sum on the right-hand side of \eqref{eq:dirichlet_periodicity} is an exponential sum over an interval in $m$, with phase $\psi(m) = -\frac{t}{2\pi}\log(a+mq)$. This interval has length $\frac{X}{q}$, and the derivative of the phase satisfies $\abs{\psi'(m)} = \frac{1}{2\pi}\frac{|t|q}{a+mq} \asymp \frac{|t|}{\abs{\frac{X}{q}}}$. As before, depending on whether or not $\frac{q|t|}{X}\gg 1$, we consider two cases.

\textbf{Case I: } Firstly, suppose that $2 \leq |t| \leq c \frac{X}{q}$, where $c$ is taken sufficiently small such that the nearest integer to $\abs{\psi'(m)}$ is 0 rather than 1. Then, we may apply the Kusmin--Landau bound to see that the inner sum is 
$$\ll \frac{\frac{X}{q}}{1+|t|},$$
where in the denominator we have used the lower bound on $|t|$. This gives the second term in the statement of \cref{lem:exponent_pairs}. 

\textbf{Case II: } Suppose we are now in the complementary region $c\frac{X}{q} < |t| < X^{24}$. The power-saving bound is essentially an immediate consequence of the fact that $\left(\kappa, \nu\right)$ is an exponent pair. This follows from the classical exponent pair $\left(\frac{1}{2^k-2}, 1-\frac{k-1}{2^k-2}\right)$ with $k=25$, which comes from inductively applying the van der Corput $A$ and $B$ processes (again see Chapter 3 of \cite{gk91}).

For an exponent pair $(\kappa, \nu)$, provided that $\abs{\psi'(m)} \gg 1$, we get an upper bound of $\left(\frac{|t| q}{X}\right)^{\kappa} \left(\frac{X}{q}\right)^{\nu}$. Hence the inner sum on the right-hand side of \eqref{eq:dirichlet_periodicity} is bounded by
$$\ll \sum_{a \pmod{q}} \left(|t|^{\kappa} \left(\frac{X}{q}\right)^{\nu-\kappa} + \frac{\frac{X}{q}}{1+|t|}\right).$$
After summing over residue classes $a\pmod q$, we get
$$\ll q^{1-\nu+\kappa} |t|^{\kappa} X^{\nu-\kappa} + \frac{X}{1+|t|}.$$
Because $q \leq (\log X)^{B}$, this $q$ prefactor is absorbed by the $X^{o(1)}$ allowance. 
\end{proof}

\subsection{Integers with typical factorizations}{\label{sec:2.2}}
We now describe the procedure for restricting $n\in [X, 3X]$ to integers with \emph{typical factorizations}.
More or less we use the same choices as in \cite{matomaki_radziwill}, except crucially in the last range.
As in \cite{average_chowla}, we will introduce an additional parameter $X^{\frac{1}{2}} \leq X_0 \leq X$, with the idea being that sometimes we want to count integers with a typical factorization up to something slightly smaller than $3X$, but we want the number of primes in the factorization to be fixed in a way that depends only on the ambient scale $X$.
Our choice of $X_0$ will be
$$\begin{cases}
    X_0 = X, &\text{sections 3,4}\\
    X_0 = X^{\frac{1}{2}}, &\text{section 5}
\end{cases}.$$

We additionally choose $J = J(X_0)$ ranges of primes $[P_j, Q_j]$ subject to the following spacing constraints:
$$\frac{\log\log Q_j}{\log P_{j-1}-1} \leq \frac{\eta}{4j^2},$$
$$\frac{\eta}{j^2}\log P_j \geq 8\log Q_{j-1} + 16\log j.$$
Here $\eta$ is a parameter to be specified later, but should be thought of as a very small constant.
We may take $J$ to be the largest index for which $Q_j \leq \exp\left( \frac{\log X_0}{\log\log X_0}\right)$, and we retain the choices of $P_j, Q_j$ for $j\leq J$ made in \cite{matomaki_radziwill}.
Namely we choose (for $j\leq J$)
$$P_j = \exp\left(j^{4j}(\log Q_1)^{j-1}(\log P_1)\right)$$
and 
$$Q_j = \exp\left(j^{4j+2} (\log Q_1)^{j}\right).$$

\begin{definition}{\label{def:restricted_factors}}
    Let $\mathcal{S}_{P_1, Q_1, X_0, X}$ denote the set of integers up to $3X$ which contain at least one prime factor $p_j \in [P_j, Q_j]$ for each $1\leq j\leq J$.
\end{definition}

Now we have the following estimate.

\begin{proposition}{\label{prop:sieve_error}}
    Let $X^{\frac{1}{2}} \leq X_0 \leq X$, and let $P_1 < Q_1$. Let $\mathcal{S} = \mathcal{S}_{P_1, Q_1, X_0, X}$ be as above. We have
    $$\#\left\{n\leq 3X: n\notin \mathcal{S}\right\} \ll X \frac{\log P_1}{\log Q_1}. $$
\end{proposition}
\begin{proof}
    For $j=1, \dots, J$, observe that by the fundamental lemma of the sieve, we have that the number of $n \leq 3X$ such that $p|n$ implies $p\notin [P_j, Q_j]$ is 
  $$\ll X\prod_{p\in [P_j, Q_j]} \left(1-\frac{1}{p}\right) \ll X\frac{\log P_j}{\log Q_j} \ll X\frac{1}{j^2}\frac{\log P_1}{\log Q_1}.$$
  Upon taking the sum over $j$ we obtain a bound on the number of $n \leq 3X$ which do not have a prime factor in every range $[P_j, Q_j]$.
\end{proof}

The main difference in the argument presented here will be the last range. Let $R > 1$ be a parameter to be chosen later. Define 
\begin{equation}{\label{eq:PQ_choices}}
    P = \exp\left(
    \frac{1}{R}\frac{(\log X)(\log\log\log X)}{\log\log X}\right) \quad\text{ and }\quad Q=3X
\end{equation}

We will need the following lemma before counting $n\in [X, 3X]$ without prime factors in these ranges:

\begin{lemma}{\label{lem:smooth_numbers}}
    Let $x \geq 10$, and let $\Psi(x,y)$ count the number of $y$-smooth integers up to $x$.
    Let $u = \frac{\log x}{\log y}$. Fix $\epsilon > 0$. Then as long as $y\leq x^{1-\epsilon}$ and $u,y \to \infty$, we have
    $$\Psi(x, y) = xu^{-u(1+o(1))}.$$
\end{lemma}
Equivalently, the density of $x^{\frac{1}{u}}$-smooth numbers is $u^{-u + o(u)}$.
For a proof, see Corollary 1.3 of \cite{tenenbaum_smooth}.

\begin{lemma}{\label{lem:last_range_sieve_error}}
    Let $X \geq 10$ and let $P < Q$ be as specified in \eqref{eq:PQ_choices}. Then 
    $$\#\{n \leq 3X: p|n \implies p\notin [P,Q]\} \ll_R X (\log X)^{-{R}+o(1)}.$$
\end{lemma}
\begin{proof}
    If $n \leq 3X$ is not divisible by any prime in this range, $n$ is $P$-smooth. By \cref{lem:smooth_numbers}, the density of $X^{\frac{1}{u}}$-smooth numbers is $u^{-u+o(u)}$. Observe that $P = X^{\frac{1}{u}}$ with $u = R\frac{\log\log X}{\log \log \log X}$. So
$$u^{-u} = \left(\frac{1}{R} \frac{\log\log\log X}{\log\log X}\right)^{R\frac{\log\log X}{\log\log\log X}} \ll_R \frac{1}{(\log X)^{R+o(1)}}.$$
The claim follows. 
\end{proof}

Integers with no prime factor in the last range will not be excluded from $\mathcal{S}$, but rather we will manually remove these integers when appropriate in the argument. 

Later, in the proofs of \cref{thm:smooth_f_thm} and \cref{thm:general_f}, we will also want to add back the integers $n\notin\mathcal{S}$ (excluding the last interval $[P,Q]$) on the Dirichlet polynomial side in order to get improved savings from Hal\'asz theorem.

\begin{lemma}{\label{lem:add_back_sieve_error}}
    Let $X \geq h \geq 10$.
    Let $(a_n)$ be a $1$-bounded sequence of complex numbers.
    With $\mathcal{S}$ as above, if $Q_1\leq \min\left(h, \exp\left(\frac{\log X_0}{\log\log X_0}\right)\right)$ and $P_1 = (\log Q_1)^{C_1}$ for an absolute constant $C_1 > 0$, then
    $$\int_{-T}^T \abs{\sum_{\substack{X\leq n\leq 3X \\ n\notin \mathcal{S}}} a_n n^{-1-it}}^2 dt \ll \frac{T}{X} \frac{\log P_1}{\log Q_1} + \left(\frac{\log P_1}{\log Q_1}\right)^2.$$
\end{lemma}

The proof will go via the improved mean value theorem.

\begin{proof}
    Let
    $$\mathcal{A}_j = \left\{ n\in [X, 3X]: p|n \implies p\notin [P_j, Q_j]\right\}$$
    for $1\leq j\leq J$. Let 
    $$\mathcal{A} = \bigcup_{j\leq J} \mathcal{A}_j = [X, 3X] \setminus \mathcal{S}.$$
    We will apply \cref{lem:IMVT} to the Dirichlet polynomial with coefficients $\1_{\mathcal{A}}(n)$ (here we are using 1-boundedness of the $a_n$), so 
    \begin{equation}{\label{eq:apply_imvt}}
        \int_{-T}^T \abs{\sum_{\substack{X\leq n\leq 3X \\ n\notin \mathcal{S}}} a_n n^{-1-it}}^2 dt \ll \frac{T}{X^2}\sum_{\substack{X\leq n\leq 3X\\ n\in \mathcal{A}}} 1 + \frac{T}{X^2} \sum_{1\leq \ell\leq X/T} \sum_{X\leq n\leq 3X} \1_{\mathcal{A}}(n)\1_{\mathcal{A}}(n+\ell).
    \end{equation}
    The first term is admissible; we know by \cref{prop:sieve_error} that this is at most 
    $$\frac{T}{X} \frac{\log P_1}{\log Q_1}$$
    which contributes the first term to the claimed upper bound.
    The more involved computation concerns the second term. For a fixed $\ell$, note that the inner sum is exactly
    $$r(\ell) = \#\{n\in [X, 3X] : n, n+\ell \in \mathcal{A}\}.$$
    Because $\mathcal{A}$ is a union of $\mathcal{A}_j$, we can upper bound 
    $$\1_{\mathcal{A}}(n)\1_{\mathcal{A}}(n+\ell) \leq \sum_{i=1}^J\sum_{j=1}^J \1_{\mathcal{A}_i}(n) \1_{\mathcal{A}_j}(n+\ell),$$
    and so if 
    $$r_{i,j}(\ell) = \#\{n\in [X, 3X]: n\in\mathcal{A}_i, n+\ell \in \mathcal{A}_j\},$$
    we have the inequality
    $$r(\ell) \leq \sum_{i=1}^J \sum_{j=1}^J r_{i,j}(\ell).$$
    We consider the diagonal and off-diagonal contributions separately.
    Firstly suppose $i\neq j$, and recall that the prime ranges $[P_i, Q_i]$ and $[P_j, Q_j]$ are disjoint.
    For fixed $\ell$, we count $n$ such that 
    \begin{itemize}
        \item for all $p\in [P_i, Q_i]$, $p\nmid n$,
        \item for all $p\in [P_j, Q_j]$, $p\nmid n+\ell$.
    \end{itemize}
    This forbids one congruence class for each prime in this disjoint range, so by the fundamental lemma of the sieve, 
    \begin{align}{\label{eq:off_diag_imvt}}
        r_{i,j}(\ell) &\ll X \prod_{p\in [P_i, Q_i]} \left(1-\frac{1}{p}\right) \prod_{p\in [P_j, Q_j]} \left(1-\frac{1}{p}\right) \nonumber\\
        &\ll X\frac{\log P_i}{\log Q_i} \frac{\log P_j}{\log Q_j}. 
    \end{align}
    Now, suppose that $i=j$, in which case the conditions are no longer disjoint.
    For all $p\in [P_j, Q_j]$ we require that $n\not\equiv 0$, $n\not\equiv -\ell$ modulo $p$. If $p|\ell$, then this is still just one restricted congruence class, and if $p\nmid \ell$ this is two.
    Thus, 
    $$r_{j,j}(\ell) \ll X \prod_{\substack{p\in [P_j, Q_j]\\ p\nmid \ell}} \left(1-\frac{2}{p}\right) \prod_{\substack{p\in [P_j, Q_j] \\ p|\ell}}\left(1-\frac{1}{p}\right).$$
    We would like to end up with the squared density of $\mathcal{A}_j$, so we write the right-hand side as 
    \begin{equation}{\label{eq:diag_imvt}}
        X \prod_{\substack{p\in [P_j, Q_j]\\ p\nmid \ell}} \left(1-\frac{2}{p}\right) \prod_{\substack{p\in [P_j, Q_j] \\ p|\ell}}\left(1-\frac{1}{p}\right) \ll X \left(\frac{\log P_j}{ \log Q_j}\right)^2 \underbrace{\prod_{\substack{p\in [P_j, Q_j] \\ p|\ell}} \left(\frac{p-1}{p-2}\right)}_{w_j(\ell)}.
    \end{equation}
    The idea will be to show that on average over $\ell$, $w_j(\ell)$ is well-behaved.
    Before that, we will plug our bounds from \eqref{eq:off_diag_imvt} and \eqref{eq:diag_imvt} into the second term on the right-hand side of \eqref{eq:apply_imvt}:
    \begin{align*}
        \frac{T}{X^2} \sum_{1\leq \ell\leq X/T} \sum_{X\leq n\leq 3X} \1_{\mathcal{A}}(n)\1_{\mathcal{A}}(n+\ell) &= \frac{T}{X^2}\sum_{1\leq \ell\leq X/T} \left(\sum_{j=1}^J r_{j,j}(\ell) + \sum_{j=1}^J\sum_{\substack{i=1 \\ i\neq j}}^J r_{i,j}(\ell)\right)\\
        &\ll \left(\sum_{j=1}^J \left(\frac{\log P_j}{\log Q_j}\right)^2 \cdot \frac{T}{X}\sum_{1\leq \ell \leq X/T} w_j(\ell)\right) \\
        &\quad + \left(\sum_{i=1}^J\sum_{j=1}^J \frac{\log P_i}{\log Q_i}\frac{\log P_j}{\log Q_j}\right).
    \end{align*}
    If we can prove that $\frac{T}{X}\sum_{1\leq \ell\leq X/T}w_j(\ell)$ is bounded for each $j$, then we can factor this as a square and the result follows from \cref{prop:sieve_error}.
    We compute 
    \begin{align*}
        \frac{T}{X}\sum_{1\leq \ell \leq X/T} w_j(\ell) &= \frac{T}{X}\sum_{1\leq \ell \leq X/T} \prod_{\substack{p\in [P_j, Q_j] \\ p|\ell}}\left(1+\frac{1}{p-2}\right)\\
        &= \frac{T}{X}\sum_{1\leq \ell \leq X/T} \sum_{\substack{d|\ell\\ p|d \implies p\in [P_j, Q_j]}} \prod_{p|d} \frac{1}{p-2}\\
        &= \frac{T}{X}\sum_{1\leq \ell \leq h}\sum_{d|\ell}g_j(d),
    \end{align*}
    where $g_j(d)$ is the multiplicative function supported on squarefree $d$, defined by 
    $$g_j(d) = \begin{cases}
        \prod_{\substack{p|d}}\frac{1}{p-2}, &p|d\implies p\in [P_j, Q_j]\\
        0, &\text{otherwise}
    \end{cases}.$$
    Swapping sums, 
    \begin{align*}
        \frac{T}{X}\sum_{1\leq \ell \leq X/T}\sum_{d|\ell}g_j(d) &= \frac{T}{X} \sum_d g_j(d) \sum_{\substack{\ell \leq X/T \\ d | \ell}} 1 
        \leq \sum_{d} \frac{g_j(d)}{d} 
        = \prod_{p\in [P_j, Q_j]}\left(1+\frac{1}{p(p-2)}\right).
    \end{align*}
    This product is bounded uniformly and independently of $j$ because $\sum_p \frac{1}{p^2}$ converges.
    We conclude that the second term on the right-hand side of \eqref{eq:apply_imvt} is 
    \begin{equation}{\label{eq:sieve_bound}}
        \ll \left(\sum_{j=1}^J \frac{\log P_j}{\log Q_j}\right)^2.
    \end{equation}
    Now, just as in the proof of \cref{prop:sieve_error}, 
    by definition of $P_j, Q_j$, \eqref{eq:sieve_bound} is at most $\left(\frac{\log P_1}{\log Q_1}\right)^2$, which completes the proof.
\end{proof}

Later on, in the proof of \cref{thm:general_f} we will need to apply a Ramar\'e type decomposition (analogous to \cref{lem:last_range_sieve_error}) with a more restrictive choice of the parameter $Q$.  We state this result using a choice of parameters which will be explained further in \cref{sec:4.2}.
\begin{lemma}
{\label{lem:last_range_sieve_error_2}}
    Let $P$ be as in \cref{eq:PQ_choices} with $R=2$ and let $Q = \frac{X}{\exp\left((\log X)^{\beta}\right)}$. Then 
    \begin{equation}
        \#\{n \leq 3X: p|n \implies p\notin [P,Q]\} \ll X(\log X)^{-1+\beta}.
    \end{equation}
\end{lemma}
\begin{proof}
Supposing that $n \leq 3X$ is not divisible by any prime in $[P,Q]$, either $n$ is $P$-smooth, or $n$ has some prime factor in the range $[Q, 3X]$. 
 Exactly as before, the number of $P$-smooth $n \leq 3X$ is $\ll \frac{X}{(\log X)^{2+o(1)}}$.
If $n$ is not $P$-smooth, then $n$ has a prime factor in the range $[Q, 3X]$, and we may estimate the number of such $n$ using Mertens' theorem:
$$X\sum_{Q < p < 3X}\frac{1}{p} \ll X \left(\log\log X - \log\log Q\right) \ll X\left(\log\log X - \log\left(\log X \left(\frac{(\log X)^{1-\beta}}{\log X}\right)\right)\right)$$
and via Taylor expansion, this is at most 
$$\ll \frac{X}{(\log X)^{1-\beta}}.$$
This is the dominant error, and so the claim follows.
\end{proof}

We will also give a variant of \cref{lem:add_back_sieve_error} for the last range, in essentially the same way. This will enable us to handle the sieve error on the frequency side in the proof of \cref{thm:general_f}.
\begin{lemma}{\label{lem:abse_2}}
    Let $X\geq h \geq 10$ and let $T = \frac{X}{h}$. Let $(a_n)$ be a 1-bounded sequence of complex numbers. Let $\frac{1}{2}\leq \beta \leq \frac{99}{100}$. Let $P$ be as in \cref{eq:PQ_choices} with $R\geq 2$ and let $Q = X\exp\left(-(\log X)^{\beta}\right)$. Then
    $$\int_{-T}^T \abs{\sum_{\substack{X\leq n\leq 3X\\ p|n\implies p\notin [P,Q]}} a_n n^{-1-it}}^2 dt \ll \frac{1}{(\log X)^{2(1-\beta)}} + \frac{1}{h^2}.$$
\end{lemma}
The proof will mirror \cref{lem:add_back_sieve_error} and again follow from the improved mean value theorem.
\begin{proof}
    Let
    $$\mathcal{A} = \left\{n\in[X, 3X]: p|n \implies p\notin[P,Q]\right\},$$
    and define $M := \frac{3X}{Q}$, so that $\log M \asymp (\log X)^{\beta}$ and $M^2 = X^{o(1)}$. We apply \cref{lem:IMVT} to the sequence $\1_{\mathcal{A}}(n)$ to get
    \begin{equation}{\label{eq:imvt2}}
        \int_{-T}^T \abs{\sum_{X\leq n\leq 3X} \1_{\mathcal{A}}(n) n^{-1-it}}^2 dt \ll \frac{T}{X^2}\sum_{X\leq n\leq 3X} \1_{\mathcal{A}}(n) + \frac{T}{X^2}\sum_{1\leq \ell\leq \frac{X}{T}} \sum_{X\leq n\leq 3X}\1_{\mathcal{A}}(n)\1_{\mathcal{A}}(n+\ell).
    \end{equation}
    By definition of $T$ and by \cref{lem:last_range_sieve_error_2}, the diagonal term is at most
    $$\frac{1}{h} (\log X)^{-(1-\beta)} \ll \frac{1}{h^2} + (\log X)^{-2(1-\beta)},$$
    via the elementary inequality $ab \ll a^2 + b^2$. For fixed $\ell$, the inner sum in the second term of \eqref{eq:imvt2} is precisely
    $$r(\ell) = \#\left\{n\in [X, 3X]: n, n+\ell\in \mathcal{A}\right\}.$$
    If $n\in \mathcal{A}$, then either every prime dividing $n$ is at most $P$ (in which case $n$ is $P$-smooth) or $n$ is divisible by a prime $p > Q$. Define the sets
    $$\mathcal{A}_0 = \{n\in [X, 3X]: n\text{ is }P-\text{smooth}\},$$
    $$\mathcal{A}_1 = \{n\in [X, 3X]: n = pm,\ p > Q\},$$
    and observe that $\mathcal{A} \subseteq \mathcal{A}_0 \cup \mathcal{A}_1$, which implies right away that
    $$\1_{\mathcal{A}}(n)\1_{\mathcal{A}}(n+\ell) \leq \sum_{i=0}^1\sum_{j=0}^1 \1_{\mathcal{A}_i}(n)\1_{\mathcal{A}_j}(n+\ell).$$
    If we define
    $$r_{i,j}(\ell) = \#\{n\in [X, 3X]: n\in \mathcal{A}_i, n+\ell\in \mathcal{A}_j\}$$
    then we have $r(\ell) \leq \sum_{i=0}^1 \sum_{j=0}^1 r_{i,j}(\ell)$. Now, suppose that either $i=0$ or $j=0$. Just as in \cref{lem:last_range_sieve_error}, we have that 
    $$\#\mathcal{A}_0 \ll_R X (\log X)^{-R+o(1)}.$$
    For fixed $\ell$, given that either one of $n, n+\ell$ is $P$-smooth, we have that $r_{i,j}(\ell)\leq 2\#\mathcal{A}_0$ (using a crude bound) and so the total contribution to the second term of \eqref{eq:imvt2} is 
    $$\ll \frac{T}{X^2} \frac{X}{T} X (\log X)^{-R+o(1)} \ll (\log X)^{-2(1-\beta)},$$
    since $R \geq 2$. 
    
    Now, suppose that $i=j=1$. Note that if $n=pm$ with $p > Q$, because $Q > X^{\frac{1}{2}}$, it is unique, so we can decompose the integers by the residual part $m$. Write
    $$r_{1,1}(\ell) \leq \sum_{1\leq m_1, m_2\leq M} r_{m_1, m_2}(\ell),$$
    where
    $$r_{m_1, m_2}(\ell) = \#\left\{n\in [X, 3X]: n\equiv 0\pmod{m_1}, n\equiv -\ell\pmod{m_2},\ \frac{n}{m_1}, \frac{n+\ell}{m_2}\text{ are prime}\right\}.$$

    For fixed $m_1, m_2$ let $g = \gcd(m_1, m_2)$, and note that the first two congruence conditions require $g|\ell$, so we can restrict ourselves to looking at one congruence via the Chinese remainder theorem, modulo $\frac{m_1m_2}{g}$. Hence, there are at most $\frac{X g}{m_1m_2}+1$ choices for $n$. We may ignore the $+1$, as it contributes $\ll \frac{T}{X^2} \frac{X}{T} M^2 \ll X^{-1+o(1)}$, which is much smaller than the stated upper bound.
    
    Now, if $p\nmid m_1$ (resp. $p\nmid m_2$) and $p|\frac{n}{m_1}$ (resp. $p|\frac{n+\ell}{m_2}$) then we must have $p|n$ (resp. $p|n+\ell$). If $\frac{n}{m_1}$ and $\frac{n+\ell}{m_2}$ are prime, then neither has a prime factor $\leq X^{\frac{1}{100}}$, say. For all primes $p\leq X^{\frac{1}{100}}$ with $p\nmid {m_1m_2}$ we have that
    \begin{itemize}
        \item $n\not\equiv 0\pmod p$\\
        \item $n \not\equiv -\ell\pmod p$.
    \end{itemize}
    Because we only need an upper bound, we can drop the sieve conditions at $p=2$ and $p|m_1m_2$. Note that this is either one or two congruence-class restrictions, depending on whether or not $p| \ell$. By the fundamental lemma of the sieve,
    \begin{align*}
        r_{m_1, m_2}(\ell) &\ll \frac{X g}{m_1m_2} \prod_{\substack{2 < p \leq X^{1/100}\\ p\nmid \ell m_1m_2}} \left(1-\frac{2}{p}\right)\prod_{\substack{2 < p \leq X^{1/100}\\ p|\ell,\ p\nmid m_1m_2}} \left(1-\frac{1}{p}\right)\\
        &\ll \frac{Xg}{m_1m_2} \prod_{2 < p \leq X^{1/100}}\left(1-\frac{2}{p}\right) \underbrace{\prod_{\substack{2 < p \leq X^{1/100}\\ p|\ell,\ p\nmid m_1m_2}} \left(\frac{p-1}{p-2}\right)}_{w_{m_1, m_2}(\ell)} \prod_{\substack{2 < p\leq X^{1/100} \\ p|m_1m_2}}\left(1-\frac{2}{p}\right)^{-1}\\
        &\ll \frac{Xg}{m_1m_2} \cdot \frac{1}{(\log X)^{2}} \cdot w_{m_1, m_2}(\ell) u(m_1)^2 u(m_2)^2,
    \end{align*}
    with
    $$u(m) = \prod_{\substack{p > 2\\ p|m}}\frac{p-1}{p-2}.$$
    In the last line we used the fact that $\left(1-\frac{2}{p}\right)^{-1} = \left(\frac{p-1}{p-2}\right)\left(\frac{p}{p-1}\right)\leq \left(\frac{p-1}{p-2}\right)^2$, and the fact that the function $u$ is sub-multiplicative (which comes from a union bound). As before, we will prove that on average over $\ell$, the function $w_{m_1, m_2}(\ell)$ is well-behaved. In total, so far we have that
    \begin{align}{\label{eq:term2_final}}
        \frac{T}{X^2} \sum_{1\leq \ell\leq X/T} \sum_{X\leq n\leq 3X} &\1_{\mathcal{A}}(n)\1_{\mathcal{A}}(n+\ell) \ll \frac{1}{(\log X)^{2(1-\beta)}} + \frac{T}{X^2}\sum_{1\leq \ell\leq X/T} \sum_{m_1, m_2 \leq M} r_{m_1, m_2}(\ell) \nonumber\\
        &\ll \frac{1}{(\log X)^{2(1-\beta)}} + \sum_{m_1, m_2\leq M} \frac{u(m_1)^2u(m_2)^2}{m_1m_2} \frac{Tg}{X(\log X)^2}\sum_{\substack{1\leq \ell\leq X/T\\ g|\ell}}w_{m_1, m_2}(\ell).
    \end{align}

    We will prove that $$\frac{T}{X}\sum_{\substack{1\leq \ell\leq X/T\\ g|\ell}}w_{m_1, m_2}(\ell) \ll \frac{1}{g},$$
    after which we will only need to bound $\left(\sum_{m\leq M}\frac{u(m)^2}{m}\right)^2$. Write $\ell = g \ell'$ and note that because every divisor of $g$ divides $m_1m_2$, we have
    $$w_{m_1, m_2}(g\ell') \leq w(\ell') := \prod_{\substack{p > 2\\ p|\ell'}} \frac{p-1}{p-2},$$
    and so it suffices to bound $\frac{T}{X}\sum_{1\leq \ell' \leq X/Tg} w(\ell')$. Now using the Euler product structure we can write this as
    $$\frac{T}{X}\sum_{1\leq \ell' \leq X/Tg} w(\ell') = \frac{T}{X}\sum_{1\leq\ell' \leq X/Tg }\sum_{d|\ell'} g_w(d),$$
    where $g_w(d)$ is the multiplicative function given by 
    $$g_w(d) = \begin{cases}
        \prod_{p|d} \frac{1}{p-2}, &d\text{ odd and squarefree}\\
        0, &\text{ else}
    \end{cases}.$$
Just as before, by swapping the sums we have
\begin{equation}{\label{eq:w_on_avg}}
    \frac{T}{X}\sum_{1\leq \ell'\leq X/Tg} \sum_{d|\ell'}g_w(d) \leq \frac{1}{g} \sum_d \frac{g_w(d)}{d} \ll \frac{1}{g},
\end{equation}
where the convergence of the Euler product associated to $\frac{g_w(d)}{d}$ is a consequence of the fact that $\sum_{p} \frac{1}{p^2}$ converges absolutely. 

Lastly, we need to bound $\sum_{m\leq M} \frac{u(m)^2}{m}$. By the Euler product formula, we have
\begin{equation}{\label{eq:u_squared_EP}}
    \sum_{m \leq M} \frac{u(m)^{2}}{m} \leq \prod_{p \leq M}\left(1 + u(p)^{2}\sum_{k \geq 1} p^{-k}\right) = 2\prod_{2 < p \leq M}\left(1 + \frac{p-1}{(p-2)^{2}}\right) \ll \prod_{p \leq M}\left(1 + \frac{1}{p}\right) \ll \log M.
\end{equation}

So, plugging \eqref{eq:w_on_avg} and \eqref{eq:u_squared_EP} into \eqref{eq:term2_final} lets us conclude that
$$\frac{T}{X^2} \sum_{1\leq \ell\leq X/T} \sum_{X\leq n\leq 3X} \1_{\mathcal{A}}(n)\1_{\mathcal{A}}(n+\ell) \ll (\log X)^{-2(1-\beta)} + \left(\frac{\log M}{\log X}\right)^2 \ll (\log X)^{-2(1-\beta)},$$
by construction of $M$.
\end{proof}

Lastly, for $1\leq j\leq J$ define the following quantities:
$$\alpha_j  = \frac{\varepsilon}{2} - \eta\left(1+\frac{1}{2j}\right).$$
Here $\varepsilon$ is some number that we will fix later, but should be thought of as small. $\eta$ will also be fixed, and made small enough such that $\alpha_j$ is (strictly) positive for every $1\leq j\leq J$.
These are the sequence of power-savings for intermediate intervals from \cite{matomaki_radziwill}, although we have replaced $\frac{1}{2}$ by $\frac{\varepsilon}{2}$, as this will give us stronger control of quantities later on in the proof.

\subsection{Decompositions of Dirichlet polynomials}{\label{sec:2.3}}
Lastly, we need two decompositions of Dirichlet polynomials that will enable us to use standard $L^2$--$L^{\infty}$ techniques in our Dirichlet polynomial estimates. 

\begin{lemma}{\label{lem:ramare}}
    Let $X\geq 10$. Let $H\geq 1$ and let $Q \geq P \geq 1$.
    Let $a_n, b_m, c_p$ be bounded sequences such that $a_{mp} = b_m c_p$ whenever $p \nmid m$ and $p\in [P,Q]$.
    Further, let 
    $$Q_{v, H}(s) = \sum_{\substack{P\leq p \leq Q \\ e^{v/H}\leq p \leq e^{(v+1)/H}}} \frac{c_p}{p^s},$$
    $$R_{v,H}(s) = \sum_{Xe^{-v/H}\leq m\leq 2Xe^{-v/H}} \frac{b_m}{m^{s}}\frac{1}{\omega_{[P,Q]}(m)+1},$$
    and let $\mathcal{T}\subseteq [-T,T]$.
    Then 
    $$\int_{\mathcal{T}}\abs{\sum_{\substack{X\leq n\leq 3X \\ \exists p\in[P,Q]\text{ such that }p\mid n}} \frac{a_n}{n^{1+it}}}^2dt \ll H\log \left(\frac{Q}{P}\right) \sum_{v\in \mathcal{I}} \int_{\mathcal{T}} \abs{Q_{v,H}(1+it) R_{v,H}(1+it)}^2 dt + \frac{1}{H} + \frac{1}{P}.$$
    Here $\omega_{[P,Q]}(m)$ denotes the number of (distinct) primes $q\in [P,Q]$ dividing $m$, and $\mathcal{I}$ is the interval of integers in $[H\log P, H\log Q]$.
\end{lemma}
\begin{proof}
    See Lemma 12 of \cite{matomaki_radziwill}.
\end{proof}

We also give a variant of this proof where rather than split into short intervals to decouple the $m$ and $p$ variables, we use Perron's formula. This shifts the frequencies, but crucially we do not incur an error of the form $\frac{1}{H}$, and the polynomials produced are genuinely dyadic.

\begin{lemma}[Ramar\'e's identity, Perron variant]{\label{lem:perron_ramare}}
    Let $X > 10$ and define $P, Q$ as in \eqref{eq:PQ_choices}. 
    Let $a_n, b_m, c_p$ be bounded sequences such that $a_{mp} = b_m c_p$ whenever $p \nmid m$ and $p\in [P,Q]$. Let
    $$F(s) = \sum_{\substack{X\leq n\leq 3X \\ \exists p\in[P,Q]\text{ such that }p\mid n}} \frac{a_n}{n^{s}}.$$
    For each $v\in \mathcal{I} := [\log P, \log Q]$ define
    $$Q_v(s) = \sum_{\substack{P\leq p\leq Q \\ e^{v}\leq p\leq e^{(v+1)}}} c_p p^{-s},$$
    $$R_v(s) = \sum_{Xe^{-v}/4 \leq m\leq 4Xe^{-v}} \frac{b_m}{\omega_{[P,Q]}(m)+1} m^{-s}.$$
    Let $D = D(X)$ satisfy $0 < D < P$. For every $\mathcal{T} \subseteq [-T, T]$, if $T\leq X$ we have
    $$\int_{\mathcal{T}} \abs{F(1+it)}^2 dt \ll (\log D) (\log Q) \sum_{v\in \mathcal{I}} \int_{-D}^{D} \abs{\frac{3^{iu}-1}{u}} \int_{\mathcal{T}+u} \abs{Q_v(1+it)R_{v}(1+it)}^2 dt du + \frac{1}{D}.$$
\end{lemma}
Note here that we have shifted each frequency set by $u$.

\begin{proof}
    Let $S = \left\{X \leq n\leq 3X: \exists p\in [P,Q] \text{ such that }p^2|n\right\}$, which has cardinality
    $$\ll \sum_{P\leq p\leq Q} \frac{X}{p^2} \ll \frac{X}{P}.$$

    If we define the error Dirichlet polynomial $E_1(s) = \sum_{n\in S} e_n n^{-s}$, for bounded coefficients $e_n$, via the mean value theorem we have
    $$\int_{-T}^T \abs{E_1(1+it)}^2 dt \ll \frac{1}{P}$$
    for $T\leq X$. Hence, we can restrict our attention to $n=pm$ with $p\nmid m$, and write
    $$F(s) = \sum_{P\leq p\leq Q} \sum_{X\leq pm \leq 3X}\frac{c_p b_m}{\omega_{[P,Q]}(m)+1} (pm)^{-s} + E_1(s).$$
    Now we split the prime variable into $e$-adic intervals:
    $$F(s) = \sum_{v\in \mathcal{I}} \underbrace{\sum_{\substack{P\leq p\leq Q \\ e^{v}\leq p\leq e^{(v+1)}}} \sum_{X\leq pm \leq 3X}\frac{c_p b_m}{\omega_{[P,Q]}(m)+1} (pm)^{-s}}_{B_v(s)} + E_1(s),$$
    and we now know both $p, m$ lie in dyadic ranges. The idea will be to apply the truncated Perron formula at height $D$ for each $v$ to the $pm$ variable to decouple $m, p$.
    We have 
    \begin{equation}{\label{eq:B_integral}}
        B_v(s) = \frac{1}{2\pi} \int_{-D}^D Q_v(s+iu) R_v(s+iu) \frac{X^{iu}(3^{iu}-1)}{iu} du + E_v(s),
    \end{equation}
    where $E_v(s)$ has coefficients weighted by the truncation error in Perron's formula. We define $E_2(s) = \sum_v E_v(s)$. Define the error at each $n$ by
    $$\Delta_D(n):= 1_{[X,3X]}(n) -\frac{1}{2\pi}\int_{-D}^{D} \left[\left(\frac{3X}{n}\right)^{iu} -\left(\frac{X}{n}\right)^{iu}\right] \frac{du}{iu}.$$
    
     We know by the truncated Perron formula that 
     \begin{equation}{\label{eq:perron_error}}
         \abs{\Delta_D(n)} \ll \min\left(\frac{1}{D\abs{\log(n/X)}}\right) + \min\left(1, \frac{1}{D\abs{\log(3X/n)}}\right).
     \end{equation}
     If we write $E_2(s) = \sum_{X\leq n\leq 3X}e_n n^{-s}$, we have
     $$\sum_{X\leq n\leq 3X} |e_n|^2 \ll \sum_{X\leq n\leq 3X} |\Delta_D(n)|^2 \ll \frac{X}{D},$$
    by separating the $O(1)$ points at which \eqref{eq:perron_error} takes the value 1, and using the Taylor expansion for $\frac{1}{\log(n/X)}$ (resp. $\frac{1}{(\log(3X/n))}$) on the remainder of summation range.
     Hence, by the mean value theorem (\cref{lem:mvt}), we have
    $$\int_{-T}^T \abs{E_2(1+it)}^2 dt \ll \frac{T+X}{X^2} \cdot \frac{X}{D} \ll \frac{1}{D}.$$
    Returning to \eqref{eq:B_integral}, we take absolute values and apply Cauchy--Schwarz to the $u$ integral to obtain
    \begin{equation}{\label{eq:Bv_integral}}
        \abs{B_v(1+it)}^2 \ll \underbrace{\left( \int_{-D}^D \abs{\frac{3^{iu}-1}{iu}} du \right)}_{\ll \log D} \int_{-D}^D \abs{\frac{3^{iu}-1}{iu}} \abs{Q_{v}(1+i(t+u)) R_{v}(1+i(t+u))}^2 du + \frac{1}{D}
    \end{equation}
    Observe that because $F = \sum_{v}B_v + E_1 + E_2$, by Cauchy--Schwarz,
    $$\abs{F(1+it)}^2 \ll \log(Q/P) \sum_v \abs{B_{v}(1+it)}^2 + \abs{E_1(1+it)}^2 + \abs{E_2(1+it)}^2$$
    
    Now, we interchange the integral over $t\in \mathcal{T}$ with the $v$ sum and the $u$ integral, and apply the mean value theorem for the error polynomials $E_1(s)$ and $E_2(s)$ to obtain the result. 
\end{proof}

Lastly, the use of Heath-Brown's identity in order to decompose long Dirichlet polynomials supported on primes is the main engine that enables us to choose the interval $[P,Q]$ in the way we have for the proof of \cref{thm:liouville_bound}.

\begin{lemma}{\label{lem:heath_brown}}
    Let $K \geq 1$ be an integer.
    Fix $B > 0$, let $q\leq (\log X)^B$, and let $\chi$ be a Dirichlet character mod $q$.
    Let $T \geq 2$, and assume that for some $\epsilon > 0$, we have $P \geq T^{\epsilon}$.
    Let $P\leq P' \leq 2P$, and lastly let $C > 0$.
    If 
    $$P(s) = \sum_{P\leq p\leq P'}\chi(p) p^{-s},$$
    then there exist constants $A, D > 0$, a Dirichlet polynomial $E(s)$, and products
    $$G_{\ell}(s) = \prod_{j=1}^{J_{\ell}} M_{\ell,j}(s)$$
    for $1\leq \ell \leq L$ such that the following all hold:
    \begin{enumerate}
        \item For all $|t|\leq T$, 
        $$\abs{P(1+it)} \ll \sum_{\ell=1}^L \abs{G_{\ell}(1+it)} + |E(1+it)|,$$
        \item $L \leq (\log P)^D$,
        \item Define $\Delta := (\log P)^{-A}$.
        Each $M_{\ell, j}(s)$ is supported on $[M_{\ell,j}, (1+\Delta)M_{\ell, j}]$ for some $M_{\ell, j}\geq 1$, and its coefficients are one of either
        $$\chi(n),\quad \chi(n)\log(n),\quad \chi(n)\mu(n),$$
        \item Each $J_{\ell} \leq 2K$,
        \item If $M_{{\ell},j} > P^{1/K}$, then its coefficients are either $\chi(n)$ or $\chi(n)\log(n)$. 
        \item Each polynomial satisfies $M_{{\ell},j} > \exp\left(\frac{\log P}{\log\log P}\right)$, and 
        $$P\exp\left(-2K\frac{\log P}{\log\log P}\right) \leq \prod_{j=1}^{J_{\ell}}M_{\ell,j}  < 2P,$$
        \item The error $E(s)$ is small in mean-square:
        $$\int_{-T}^T \abs{E(1+it)}^2 dt \ll \left(\frac{T}{P}+1\right)(\log P)^{-C}.$$
    \end{enumerate}
    Any implied constants are allowed to depend on $B, C, K, \epsilon$. 
\end{lemma}

A similar argument can be found in \cite{almost_almost} Lemma 10, although in our application, we have a worse lower bound on the frequencies we are integrating over.
Therefore, it's important that we specify the coefficients of the polynomials $M_i(s)$ as the notion of ``prime-factored'' appearing there does not immediately transfer here.
Eventually we will have to do casework on the polynomials appearing on the right-hand side of this bound.

\begin{proof}
    Again, let $\Delta = (\log P)^{-A}$ for $A$ to be determined later.
    By partial summation it will suffice to prove the analogous bound for $P(s) = \sum_{P\leq n \leq 2P}\chi(n) \Lambda(n) n^{-s}$.
    Noting that we can write 
    $$P(s) = \underbrace{\sum_{n\leq 2P} \chi(n) \Lambda(n) n^{-s}}_{S(2P, s)} - \sum_{n\leq P} \chi(n) \Lambda(n) n^{-s},$$
    we only need to bound these partial sums $S(x,s)$ with some uniformity in $x$.
    By Heath-Brown's identity
    with $z = (2P)^{1/K}$, we have for $n\leq 2P$,
    \begin{equation}{\label{eq:heath_brown}}
    \Lambda(n) = -\sum_{k=1}^K (-1)^k\binom{K}{k}\msum_{\substack{m_1\cdots m_k n_1\cdots n_{k} = n \\ m_1, \dots, m_k\leq z}}\mu(m_1)\cdots \mu(m_k) \log(n_k).
    \end{equation}
    By complete multiplicativity of $\chi(n) n^{-s}$ we have 
    $$S(x,s) = \sum_{k=1}^K c_k H_{k}(x,s),$$
    where $c_k$ are some absolute constants (bounded in $K$).
    Here, 
    $$H_{k}(x,s)  = \msum_{\substack{m_1\cdots m_kn_1\cdots n_k \leq x \\ m_1, \dots, m_k\leq z}} \left(\prod_{i=1}^k \frac{\chi(m_i)\mu(m_i)}{m_i^s}\right) \left(\prod_{j=1}^{k-1}\frac{\chi(n_j)}{n_j^s} \right)\frac{\chi(n_k)\log(n_k)}{n_k^s}.$$
    We partition the range of each variable into $(1+\Delta)$-adic intervals.
    Let $\mathbf{N} = (M_1, \dots, M_k, N_1, \dots, N_k)$ be a tuple of scales for each variable, and define $H_k(x, s, \mathbf{N})$ to be the sub-sum, where the variables $m_i$ are contained in $[M_i, (1+\Delta)M_i]$ and similarly for the $n_j$.
    We have 
    $$H_k(x, s) = \sum_{\mathbf{N}}H_{k}(x, s, \mathbf{N}) + E_k(s),$$
    where $E_k(s)$ is the error obtained by dropping the cross condition.
    In particular, $E_k(s)$ will have coefficients supported on $x \leq n \leq (1+C_k \Delta)x$, and this is where we overshoot the $m_1\dots n_k \leq x$ condition.
    Here $C_k$ is some constant depending on $k$.
    The larger we take $A$ in the definition of $\Delta$, the sparser this error polynomial becomes.
    By the mean-value theorem, observe
    \begin{align}{\label{eq:single_error}}
        \int_{-T}^T \abs{E_k(1+it)}^2 dt &\ll \left(\frac{T}{P}+1\right)\sum_{2P \leq n\leq 2P(1+C_k\Delta)} \frac{\tau_{2k}(n)^2 (\log n)^2}{n}\\
        &\ll \left(\frac{T}{P}+1\right) \Delta (\log P)^{O_k(1)}. \nonumber
    \end{align}

    Here the coefficients of $E_k(s)$ are bounded by the $2k$-fold divisor function as well as the $\log n$ coming from the last variable.
    By taking $A$ sufficiently large in the definition of $\Delta$, we can make this $\ll \left(\frac{T}{P}+1\right)(\log X)^{-C}$, where $C$ is the constant appearing in item (7).

    Now we can properly factor each $H_k(x, s, \mathbf{N})$, and bound 
    $$\abs{H_{k}(2P, s) - H_k(P,s)} \leq \sum_{\mathbf{N}} \prod_{v=1}^{2k} \abs{M_{k, \mathbf{N}, v}(1+it)} + \abs{E_k(1+it)}.$$
    Here, this subtraction has ensured that all the polynomials are supported in the appropriate intervals.
    Now we apply the triangle inequality and sum over $k$ to obtain 
    $$\abs{P(1+it)} \leq \sum_{k=1}^K |c_k| \sum_{\mathbf{N}} \prod_{v=1}^{2k}\abs{M_{k, \mathbf{N}, v}(1+it)} + \abs{E(1+it)},$$
    where $E(1+it)$ is just the sum of $E_k(1+it)$ weighted by the $|c_k|$.
    By Cauchy--Schwarz and because we are thinking of $K$ as fixed, we save an arbitrary power of $\log P$ in the mean-square of $E(1+it)$. 
    
    Now, each $M_i$ or $N_i$ is $(1+\Delta)$-adic, so the total number of $\mathbf{N}$ is at most $\ll (\log P)^{2k}\Delta^{-2k} \ll (\log P)^{2K(A+1)}$.
    For any polynomial $M_{k, \mathbf{N}, v}(1+it)$ with length shorter than $\exp\left(\frac{\log P}{\log\log P}\right)$, we bound these trivially by $\ll 1$. 
    
    Lastly, because we limited the range of the $m_i$ variables to $z$ for $z = P^{1/K}$, any polynomial with length $> P^{1/K}$ must have corresponded to either $\chi(n_j)$ or $\chi(n_k)\log(n_k)$. 
\end{proof}

\section{Improved bounds for the Liouville function}{\label{sec:3}}
To prove \cref{thm:liouville_bound}, we will prove the following more general proposition. 

\begin{proposition}{\label{prop:explicit_1_bound}}
    Let $X \geq h \geq 10$. Let $P_1 < Q_1$.
    Let $\mathcal{S} = \mathcal{S}_{P_1, Q_1, X_0, X}$ be as in \cref{def:restricted_factors}, and set $\varepsilon = \frac{\kappa}{2}$, and let $\eta = \frac{\varepsilon}{4}$.
    Fix $z \in \{-1, +1\}$. Let $f(n)$ be the 1-bounded multiplicative function\footnote{When $z=1$ this gives the constant function 1, and when $z=-1$ this gives $\lambda(n)$.} defined by $f(n) = z^{\Omega(n)}$.
    Fix $B > 0$ and let $q\leq (\log X)^B$, and take $\chi$ to be a Dirichlet character mod $q$. Fix $A > 0$.
    In the specification of the quantities $P, Q$ given in \eqref{eq:PQ_choices}, take the quantity $R$ to be sufficiently large in terms of $A$.
    Then
    $$\frac{1}{X}\int_X^{2X} \abs{\frac{1}{h}\sum_{\substack{x\leq m \leq x+h \\ m\in \mathcal{S}}} \chi(m)f(m) - \frac{(\log X)^{3R}}{X} \sum_{\substack{x\leq m\leq x+\frac{X}{(\log X)^{3R}} \\ m\in \mathcal{S}}} \chi(m)f(m)}^2 dx \ll_{A, B}  \frac{(\log Q_1)^{\frac{1}{3}}}{P_1^{\frac{\varepsilon}{3}-\eta}} + \frac{1}{(\log X)^A}.$$
    
\end{proposition}
Stating the bound in terms of this $z$ will enable us to handle the distribution of $m\in\mathcal{S}$ in short intervals later.

In the proof of \cref{prop:explicit_1_bound}, we first need a ``Lipschitz'' estimate, in which rather than use the Lipschitz estimate in Lemma 4 of \cite{matomaki_radziwill}, we use a simple bound coming from the density of $n\in \mathcal{S}$.
We additionally need to compute the long-interval averages of the relevant sums. 
\begin{lemma}{\label{long_averages}}
    Fix $z \in \{-1, +1\}$. Let $X \geq 10$ and fix $B > 0$.
    Then for $q \leq (\log X)^B$, and $\chi$ a Dirichlet character modulo $q$, the following holds:
    \begin{enumerate}
        \item If $z=1$ and $\chi$ is principal mod $q$, then for $\alpha > 0$ arbitrarily small,
        $$\sum_{n\leq X}\chi(n) z^{\Omega(n)} = \frac{\phi(q)}{q} X + O_{\alpha, B}((\log X)^{\alpha}).$$
        \item If $z=-1$ and $\chi$ is principal mod $q$, then for $C > 0$
        $$\sum_{n\leq X} \chi(n) z^{\Omega(n)} \ll_{B, C} \frac{X}{(\log X)^C}.$$
        \item If $\chi$ is nonprincipal mod $q$, then for $C > 0$, 
        $$\sum_{n\leq X} \chi(n) z^{\Omega(n)} \ll_{B, C} \frac{X}{(\log X)^C}.$$
    \end{enumerate}
\end{lemma}
When we apply this, we think of $q$ as something fixed, hence these main terms depending on $q$ will end up canceling in the Lipschitz estimate.
What is important here is that the estimates are uniform in $q$, and they do not depend on $X$. 
\begin{proof}
    Item (1) is immediate; the error term is obtained by the divisor bound on $\tau(q)$.
    
    For item (2), observe that $z^{\Omega(n)}$ is the Liouville function, and so essentially the same proof as \cref{lem:pnt_mobius} applies here as well. The only difference is that the associated $L$-function is $\frac{L(2s, \chi^2)}{L(s,\chi)}$, and the numerator is bounded uniformly by a constant in the Vinogradov--Korobov zero-free region. 

    Lastly, for item (3), we need to consider the case in which $\chi$ is a nonprincipal character. If $z=1$, then by periodicity $\sum_{n\leq X}\chi(n) \ll q \ll (\log X)^B$, which certainly satisfies the stated bound. If $z=-1$, the same comment as for item (2) applies, and the proof is completed by the same contour integration argument as in \cref{lem:pnt_mobius}.

\end{proof}
\begin{lemma}[Lipschitz estimate]{\label{lem:lipschitz}}
    Fix $R, B > 0$. Let $\mathcal{S}$ be as in \cref{prop:sieve_error}, and suppose that $Y = \frac{X}{(\log X)^{3R}}< X$.
    Fix $z\in \{-1, +1\}$ and let $f(n) = z^{\Omega(n)}$.
    Let $q\leq (\log X)^B$, with $\chi$ a Dirichlet character modulo $q$. Let $ C > 0$. 
    Then for any $x\in [X, 2X]$
    $$\abs{ \frac{1}{Y}\sum_{\substack{x\leq n \leq x+Y \\ n\in \mathcal{S}}} \chi(n)f(n) - \frac{1}{X}\sum_{\substack{X \leq n\leq 2X \\ n\in \mathcal{S}}} \chi(n)f(n)} \ll_{z,B} \frac{\log P_1}{\log Q_1} + (\log X)^{-C}.$$
\end{lemma}

\begin{proof}
    Note crucially that $Y = X^{1-o(1)}$.
    We split into the same cases based on $z, \chi$ as in \cref{long_averages}.
    Firstly suppose that we are in either case (2) or case (3).
    We decompose
    $$\abs{ \frac{1}{Y}\sum_{\substack{x\leq n \leq x+Y \\ n\in \mathcal{S}}} \chi(n)f(n) - \frac{1}{X}\sum_{\substack{X \leq n\leq 2X \\ n\in \mathcal{S}}} \chi(n)f(n)}$$
    into
    \begin{align}{\label{eq:big_triangle_ineq}}
         \bigg| \frac{1}{Y} \sum_{x\leq n\leq x+Y} 
         \chi(n)f(n) &- \frac{1}{Y}\sum_{\substack{x\leq n\leq x+Y\\ n\notin \mathcal{S}}} \chi(n)f(n) - \frac{1}{X} \sum_{X\leq n\leq 2X}\chi(n)f(n) + \frac{1}{X}\sum_{\substack{X \leq n\leq 2X \\ n\notin \mathcal{S}}} \chi(n)f(n)\bigg|\\
        &\ll (\log X)^{-C} + \frac{X}{Y}(\log X)^{-C'} + \frac{\log P_1}{\log Q_1}  \nonumber\\
        &\ll \frac{\log P_1}{\log Q_1} + (\log X)^{-C}\nonumber
    \end{align}
    where we have used \cref{prop:sieve_error} and \cref{long_averages} for the long averages\footnote{For the moving interval $[x, x+Y]$ one can prove an analog of \cref{prop:sieve_error} again using the fundamental lemma of the sieve, precisely because $Y$ is not much smaller than $X$.}. Observe that because $\frac{X}{Y} = (\log X)^{3R}$ for fixed $R$ and $C'$ can be taken arbitrarily large, the third line follows from the second by choosing $C' > C+3R$. 

    Lastly, suppose we are in case (1).
    We again decompose the left-hand side as in \eqref{eq:big_triangle_ineq}, and via the triangle inequality, group the terms as the difference of two long averages plus the sieve error terms.
    By \cref{long_averages} and subtracting endpoints, it is clear that
    \begin{align*}
        \abs{\frac{1}{Y} \sum_{x \leq n \leq x+Y}\chi(n) f(n) - \frac{1}{X}\sum_{X \leq n\leq 2X} \chi(n)f(n)} \ll_{\alpha} \frac{(\log X)^{\alpha}}{Y}
    \end{align*}
    which is substantially smaller than both the sieve error of $\frac{\log P_1}{\log Q_1}$ and $(\log X)^{-C}$. 
\end{proof}

With this bound, we now show how \cref{prop:explicit_1_bound} can be used to deduce \cref{thm:liouville_bound}.

\begin{proof}[Proof of \cref{prop:explicit_1_bound} $\implies$ \cref{thm:liouville_bound}.]
Let $z=-1$, so $f=\lambda$. 

From here, we apply the same strategy that can be found in both \cite{matomaki_radziwill} and \cite{average_chowla} in order to transfer the result to $\chi\lambda$ \emph{without} the restriction that $n\in \mathcal{S}$.
Observe that by the triangle inequality and $|\chi(n)\lambda(n)|\leq 1$, we have
\begin{equation}{\label{eq:liouville_triangle_eq}}
    \frac{1}{X}\int_X^{2X} \abs{\frac{1}{h} \sum_{x\leq n\leq x+h} \chi(n)\lambda(n)}^2 dx \ll \frac{1}{X} \int_X^{2X}\abs{\frac{1}{h} \sum_{\substack{x\leq n\leq x+h \\ n\in\mathcal{S}}} \chi(n)\lambda(n)}^2 dx+ \frac{1}{X} \int_X^{2X}\abs{\frac{1}{h} \sum_{\substack{x\leq n\leq x+h \\ n\notin\mathcal{S}}} 1}^2 dx. 
\end{equation}
We firstly apply \cref{lem:lipschitz} to compare the long moving average and the fixed long average. Note that the long interval average is negligible, as $\lambda$ has mean value decaying to zero with error term better than the Lipschitz estimate.
We have the bound 
$$\frac{1}{X}\sum_{X\leq n\leq 2X} \chi(n)\lambda(n) \ll \exp(-(\log X)^{c})$$
for some $c > 0$ (see for example, Exercise 7 in Chapter 11.3 of \cite{montgomery_vaughan}).
For the first term, by \cref{prop:explicit_1_bound} with $z=-1$ together with the Lipschitz estimate when $f = \lambda$, we have
$$\frac{1}{X}\int_X^{2X}
\abs{\frac{1}{h}\sum_{\substack{x\leq n\leq x+h\\ n\in \mathcal{S}}}\chi(n)\lambda(n)-\frac{1}{X}\sum_{\substack{X\leq n\leq 2X}\\ n\in \mathcal{S}}\chi(n)\lambda(n)}^2 dx \ll_A  \left(\frac{\log P_1}{\log Q_1}\right)^2 + \frac{(\log Q_1)^{1/3}}{P_1^{\varepsilon/3-\eta}} +\frac{1}{(\log X)^A}.
$$
Then we observe
\begin{align*}
    \abs{\frac{1}{X}\sum_{\substack{X\leq n\leq 2X\\ n\in \mathcal{S}}}\chi(n)\lambda(n)}^2 &\ll \abs{\frac{1}{X}\sum_{X\leq n\leq 2X}\chi(n)\lambda(n)}^2 + \abs{\frac{1}{X}\sum_{\substack{X\leq n\leq 2X\\ n\notin \mathcal{S}}}\chi(n)\lambda(n)}^2\\
    &\ll \left(\frac{\log P_1}{\log Q_1}\right)^2,
\end{align*}
as by \cref{prop:sieve_error}, the sieve error is dominant. 
We rewrite the second term on the right-hand side of \eqref{eq:liouville_triangle_eq} as 
\begin{align}{\label{eq:sieve_error_bounds}}
    \abs{\frac{1}{h}\sum_{\substack{x\leq n\leq x+h \\ n\notin \mathcal{S}}} 1} &= \abs{1 + O(1/h) - \frac{1}{h}\sum_{\substack{x\leq n\leq x+h \\ n\in \mathcal{S}}} 1}\\
    &\ll \abs{\frac{1}{X}\sum_{\substack{X\leq n\leq 2X \\ n\in\mathcal{S}}}1 - \frac{1}{h}\sum_{\substack{x\leq n\leq x+h \\ n\in \mathcal{S}}}1} + \abs{\frac{1}{X}\sum_{\substack{X\leq n\leq 2X \\ n\notin \mathcal{S}}}1} + O(1/h).\nonumber
\end{align}
For the first term we apply \cref{prop:explicit_1_bound} with $z=1, q=1$ and $\chi$ the trivial character. For the second term we apply \cref{prop:sieve_error}.
When $h \leq \exp\left(\frac{\log X_0}{\log\log X_0}\right)$, then take $P_1 = (\log Q_1)^{\frac{100}{\eta}}$ and $Q_1 = h$ to conclude 
$$\frac{1}{X}\int_X^{2X} \abs{\frac{1}{h} \sum_{x\leq n\leq x+h} \chi(n)\lambda(n)}^2 dx \ll \frac{(\log\log h)^2}{(\log h)^2}.$$

If $h > \exp\left(\frac{\log X_0}{\log\log X_0}\right)$, then we take $Q_1 = \exp\left(\frac{\log X_0}{\log\log X_0}\right)$ and $P_1$ to be $(\log Q_1)^{\frac{100}{\eta}}$ to get a bound of 
$$\frac{1}{X}\int_X^{2X} \abs{\frac{1}{h} \sum_{x\leq n\leq x+h} \chi(n)\lambda(n)}^2 dx \ll \frac{(\log\log X)^4}{(\log X)^2},$$
and then we take the sum of the bounds in each regime of $h$ to give uniformity in $h$. 
\end{proof}

In practice, using \cref{lem:lipschitz} to let us work with a long moving average enables us to dispose of the low frequencies when we eventually apply Parseval's identity.
The following form of Parseval is from Section 7 of \cite{matomaki_radziwill}.

\begin{proposition}{\label{prop:parseval_1}}
    Fix $R > 1$. Assume $1\leq h_1\leq h_2 = \frac{X}{(\log X)^{3R}} \leq X$.
    Let $f:\N\to \C$ be a 1-bounded, multiplicative function.
    Consider for $X\leq x\leq 2X$
    $$S_j(x) = \sum_{\substack{x\leq m\leq x+h_j \\ m\in \mathcal{S}}} f(m)$$
    and write 
    $$F(s) = \sum_{\substack{X\leq m\leq 3X \\ m\in\mathcal{S}}} f(m)m^{-s}$$
    Then if $T_0 = (\log X)^R$
    \begin{align*}
        \frac{1}{X}\int_X^{2X} \abs{\frac{1}{h_1}S_1(x) - \frac{1}{h_2}S_2(x)}^2 dx
        &\ll \frac{1}{(\log X)^R} + \int_{T_0}^{X/h_1} |F(1+it)|^2 \, dt \\
        &\quad + \max_{T \in \left(\frac{X}{h_1}, \frac{X}{8}\right]} \frac{X / h_1}{T} \int_{T}^{2T} |F(1+it)|^2 \, dt + \frac{1}{h_1}
    \end{align*}
\end{proposition}

\begin{remark}{\label{rmk:T_range}}
In \cite{matomaki_radziwill}, the second integral is left as 
$$\max_{T > X/h_1}\frac{X/h_1}{T}\int_{T}^{2T}|F(1+it)|^2 dt.$$
We've upper bounded the range of $T$ by noting that if $T$ is larger than $\frac{X}{8}$, then the entire integral (including the prefactor) is at most $\ll \frac{1}{h_1}$ via \cref{lem:mvt},
which is an admissible error.
\end{remark}

\begin{proof}

The statement of \cref{prop:parseval_1} is essentially the same as in Lemma 14 of \cite{matomaki_radziwill}, thus we only prove the main difference. The only modification we need to make is in the first term (corresponding to low frequencies, with $|t| < T_0$) in order to get an error term of the form $\frac{1}{(\log X)^R}$. For convenience, this $R$ will be the same as the $R$ in the definition of $P$ and $Q$ in \cref{sec:2.2}.
We will not go through the proof of \cref{prop:parseval_1} and instead refer the reader to Lemma 14 of \cite{matomaki_radziwill}.
We note that in their proof, the contribution from the small frequencies is at most 
$$T_0^2 X \frac{h_2}{X^2} \ll \frac{1}{(\log X)^R}$$
given our choices of $T_0$ and $h_2$.
This is now possible because in the Lipschitz identity, we can have $Y$ further away from $X$.
Provided we choose $Y = h_2$, we can remove frequencies up to $(\log X)^R$. Lastly, \cref{rmk:T_range} justifies the truncation of the $T$-range and the bound $\ll \frac{1}{h_1}$. 
\end{proof}

We are now in a position to prove the main proposition.

\begin{proof}[Proof (\cref{prop:explicit_1_bound}).]
    
Note that by \cref{prop:parseval_1} as well as \cref{lem:lipschitz}, in order to bound the left-hand side, it suffices to bound the first integral on the right-hand side of \cref{prop:parseval_1}, and the second integral can be bounded in essentially the exact same way (we will point out where the only difference must be made).
We proceed via essentially the same method of proof as \cite{matomaki_radziwill}, although, as mentioned earlier, the power-saving parameters $\alpha_j$ will be different and the last range of primes will be treated using Heath-Brown's identity.
For each $j$, we will apply \cref{lem:ramare}, separating our long Dirichlet polynomial into two pieces, one supported on primes in $[P_j, Q_j]$, and the other a corresponding remainder term. 

When we apply Ramar\'e's identity, for each interval $[P_j, Q_j]$ for $1\leq j\leq J$ and $v_j \in [H_j\log P_j, H_j\log Q_j]$, define the Dirichlet polynomials
$$Q_{j, v, H_j}(s) = \sum_{\substack{P_j\leq p \leq Q_j \\ e^{v/H_j}\leq p \leq e^{(v+1)/H_j}}} \frac{\chi(p)f(p)}{p^s},$$
$$R_{j, v,H_j}(s) = \sum_{\substack{Xe^{-v/H_j}\leq m\leq 2Xe^{-v/H_j} \\ m\in \mathcal{S}_j}} \frac{\chi(m)f(m)}{m^{s}}\frac{1}{\omega_{[P_j,Q_j]}(m)+1}.$$
Here, the definition of the restricted set $\mathcal{S}_j$ is the same as the definition of $\mathcal{S}$, but we drop the constraint that $n$ needs a prime factor in the range $[P_j, Q_j]$. 
The Dirichlet polynomials $Q_{v, H}(s)$ and $R_{v, H}(s)$, that omit the parameter $j$, correspond to the distinguished interval $[P,Q]$ with parameters given in \cref{eq:PQ_choices} and $v\in [H\log P, H\log Q]$ for $H$ a distinguished scale to be specified later. The corresponding set of restrictions on the remainder polynomial $R_{v, H}(s)$ will be the set $\mathcal{S}$. 

We form a partition of the interval $[T_0, T]$ as follows:
    $$[T_0, T] := \bigcup_{j=1}^J \mathcal{T}_j \cup \mathcal{U}$$
where
    $$\mathcal{T}_j = \left\{t\in [T_0, T]: j \text{ is minimal such that } \abs{Q_{j,v,H_j}(1+it)}\leq e^{-\alpha_j v/H_j} \text{ for all }v\in \mathcal{I}_j\right\},$$
and $\mathcal{U} = [T_0, T] \setminus \bigcup_{j=1}^{J}\mathcal{T}_j$.
Here, $\mathcal{I}_j$ is the interval $[H_j\log P_j, H_j \log Q_j]$ as before, where we have yet to specify the value of $H_j$. We will produce the following bounds:
\begin{equation}{\label{eq:integral_T_1}}
    \int_{\mathcal{T}_1}\abs{F(1+it)}^2 dt \ll \left(\frac{T}{X/Q_1}+1\right)\frac{(\log Q_1)^{\frac{1}{3}}}{P_1^{\frac{\varepsilon}{3}-\eta}} + \frac{1}{H_1} + \frac{1}{P_1},
\end{equation}
and for $2\leq j\leq J$,
\begin{equation}{\label{eq:integral_T_j}}
    \int_{\mathcal{T}_j}\abs{F(1+it)}^2 dt \ll \left(\frac{T}{X}+1\right)\frac{1}{j^2 P_1} + \frac{1}{H_j} + \frac{1}{P_j}.
\end{equation}

Note that \eqref{eq:integral_T_1} is slightly different (corresponding to the $\frac{\varepsilon}{3}$ fraction rather than the $\frac{1}{6}$ from \cite{matomaki_radziwill}).
For completeness we sketch the proof of this bound again as the parameters have changed.

Applying \cref{lem:ramare} with the bounded multiplicative function $\chi(n)f(n)$ and the interval $[P_1, Q_1]$ yields the bound
\begin{align}{\label{eq:t1_calculation}}
    \int_{\mathcal{T}_1}\abs{F(1+it)}^2 dt &\ll H_1 \log Q_1 \sum_{v\in \mathcal{I}_1}\int_{\mathcal{T}_1}\abs{Q_{1, v, H_1}(1+it)R_{1,v,H_1}(1+it)}^2 dt + \frac{1}{H_1} + \frac{1}{P_1} \nonumber\\
    &\ll H_1\log Q_1 \sum_{v\in \mathcal{I}_1} e^{-2\alpha_1 v/H_1}\int_{\mathcal{T}_1} \abs{R_{1,v,H_1}(1+it)}^2 dt + \frac{1}{H_1} + \frac{1}{P_1} \nonumber\\
    &\leq H_1 \log Q_1 \left(\frac{T}{X/Q_1} +1\right) \sum_{v\in \mathcal{I}_1}e^{-2\alpha_1 v/H_1} + \frac{1}{H_1} + \frac{1}{P_1}
\end{align}
where the additional subscripts of 1 indicate that we are working relative to the interval $[P_1, Q_1]$, and where in the second line we have applied the definition of $\mathcal{T}_1$ and then the mean value theorem on the remaining Dirichlet polynomial.
From here we compute the geometric series and substitute in the definition of $\alpha_1$ to get that the first quantity in \eqref{eq:t1_calculation} has size 
\begin{equation}{\label{eq:unbalanced_t1_bound}}
    H_1^2 \log Q_1 P_1^{-\varepsilon +3\eta},
\end{equation}
and upon choosing $H_1 = \frac{P_1^{\frac{\varepsilon}{3}-\eta}}{(\log Q_1)^{\frac{1}{3}}}$ to balance \eqref{eq:unbalanced_t1_bound} with the $\frac{1}{H_1}$ coming from the decomposition, we obtain the stated bound for the $L^2$ norm over $\mathcal{T}_1$.
When we ultimately choose $\varepsilon$, we choose $\eta$ accordingly to ensure that this bound is admissible.

In \cite{matomaki_radziwill}, more work is required for $2\leq j\leq J$, but crucially the differences $\alpha_{j+1}-\alpha_j$ remain the same as in the original Matom\"aki--Radziwi\l{}\l{} argument, hence we can leave their estimates of these ranges completely unchanged.
All we have to do is modify the definition of $H_j$ to 
$$H_{j} = j^2 \frac{P_1^{\frac{\varepsilon}{3}-\eta}}{(\log Q_1)^{\frac{1}{3}}}.$$

\begin{remark}{\label{eq:same_integrals}}
    Note that in the proofs of equations \eqref{eq:integral_T_1} and \eqref{eq:integral_T_j}, we did not use the fact that our multiplicative functions had the form $\chi(n)$ or $\chi(n)\lambda(n)$.
    The same bounds hold for arbitrary multiplicative functions, and we will make use of the same bounds in the proofs of \cref{thm:smooth_f_thm} and \cref{thm:general_f}.
\end{remark}

Having estimated the integrals in \eqref{eq:integral_T_1} and \eqref{eq:integral_T_j}, it remains to evaluate the integral over $\mathcal{U}$.
This is our most significant departure from the Matom\"aki--Radziwi\l{}\l{} setup. First, we exclude the integers with no prime factor in $[P,Q]$:

\begin{align*}
    \int_{\mathcal{U}}\abs{\sum_{\substack{X\leq n\leq 3X \\ n\in \mathcal{S}}} \frac{\chi(n)f(n)}{n^{1+it}}}^2dt &\ll \int_{\mathcal{U}} \abs{\sum_{\substack{X\leq n\leq 3X \\ \exists p \in [P,Q]:\ p|n}} \frac{\1_{\mathcal{S}}(n) \chi(n)f(n)}{n^{1+it}}}^2 dt + \int_{\mathcal{U}} \abs{\sum_{\substack{X\leq n\leq 3X \\ \forall p \in [P,Q]:\ p\nmid n}} \frac{\1_{\mathcal{S}}(n) \chi(n)f(n)}{n^{1+it}}}^2 dt\\
    &\ll_R \int_{\mathcal{U}} \abs{\sum_{\substack{X\leq n\leq 3X \\ \exists p \in [P,Q]:\ p|n}} \frac{\1_{\mathcal{S}}(n) \chi(n)f(n)}{n^{1+it}}}^2 dt + \frac{1}{(\log X)^{R+o(1)}}, 
\end{align*}
where we extended the range of integration in the second integral from $\mathcal{U}$ to $[0,T]$ and used the classical mean value theorem (\cref{lem:mvt}) along with \cref{lem:last_range_sieve_error}. 

Now we will apply the Ramar\'e decomposition with the large prime interval $[P,Q]$, and consider the casework arising from how long the resulting polynomial $Q_{v, H}(1+it)$ ends up being. For notational convenience, let $\mathcal{S}^+$ denote the set of integers in $\mathcal{S}$ with the additional restriction that they contain a prime factor in $[P,Q]$. 

Applying \cref{lem:ramare} gives 
\begin{equation}{\label{eq:last_ramare}}
    \int_{\mathcal{U}}\abs{\sum_{\substack{X\leq n\leq 3X \\ n\in \mathcal{S}^+}} \frac{\chi(n)f(n)}{n^{1+it}}}^2dt \ll H\log \left(\frac{Q}{P}\right) \sum_{v\in \mathcal{I}} \int_{\mathcal{U}} \abs{Q_{v,H}(1+it) R_{v,H}(1+it)}^2 dt + \frac{1}{H} + \frac{1}{P},
\end{equation}
where here $H$ is a parameter to be chosen.
We will furthermore take the supremum over $v\in\mathcal{I}$ and pick up a prefactor of $H \log\left(\frac{Q}{P}\right)$ corresponding to the number of terms.

Note that because $f$ is constant on primes (recall $f(p) = z$ for all $p$), 
$$Q_{v,H}(1+it) = z \sum_{e^{v/H}\leq p\leq e^{(v+1)/H}} \frac{\chi(p)}{p^{1+it}}.$$
Note now that $Q_{v, H}$ could be as small as $P = \exp\left(\frac{1}{R}\frac{\log X \log\log\log X}{\log \log X}\right) = X^{o_R(1)}$ or as long as $Q \asymp X$.

\textbf{Case I: }Suppose that $Q_{v,H} < X^{\frac{4}{10}}$ (we only need that this exponent is strictly less than $\frac{1}{2}$). In this case we may apply \cref{lem:pnt}, and as we know that $\inf (\mathcal{U}) \geq T_0 = (\log X)^R$, and $Q_{v,H}(1+it)$ is (always) sufficiently long, we save up to $R$ powers of $\log X$. 

In particular, the right-hand side of \eqref{eq:last_ramare} is 
\begin{equation}{\label{eq:hm_case_1}}
    \ll (\log X)^{-R + 2A + 2} \int_{\mathcal{U}} \abs{R_{v,H}(1+it)}^2 dt + \frac{1}{(\log X)^A}
\end{equation}
where we have taken $H = (\log X)^A$.
We can choose $R$ so that $R > 4A$, say, so that we are certain to save enough powers of $\log X$ here.
By a standard covering argument, we may bound the integral by the sum over a discrete, well-spaced subset $\mathcal{E}\subseteq\mathcal{U}$.
Note that because $\mathcal{E}$ is a union of well-spaced subsets of the sets of large values of the polynomials $Q_{J, v,H_J}(s)$ over $v\in \mathcal{I}_J$, we have the bound
\begin{equation}{\label{eq:E_bound}}
    |\mathcal{E}| \ll T^{2\alpha_J + o(1)} \ll X^{\varepsilon}
\end{equation}
via \cref{lem:large_values} (here we use the definition of $J$ and consequent lower bound on $P_J$, as Matom\"aki-Radziwi\l{}\l{} do).  
Note that regardless of whether we consider the first or second integral on the right-hand side of \cref{prop:parseval_1}, we have $T \ll X$.

In particular, $R_{v, H}$ has length $\gg X^{\frac{6}{10}}$, so we will certainly have that $|\mathcal{E}| T^{\frac{1}{2}} \leq R_{v, H}$ and therefore upon applying \cref{lem:hm} to $\sum_{t\in\mathcal{E}}\abs{R_{v, H}(1+it)}^2$, the right-hand side of \eqref{eq:hm_case_1} can be made at most
$$\ll \frac{1}{(\log X)^A}.$$

The more challenging cases arise when the prime-supported polynomial $Q_{v,H}(s)$ is long, so that we may not simply apply the Hal\'asz--Montgomery estimate to $R_{v,H}(s)$.

\textbf{Case II: }Now suppose $Q_{v,H} > X^{\frac{4}{10}}$. 
In this case we will apply the decomposition from \cref{lem:heath_brown} to $Q_{v,H}$ with $K=3$.
This enables us more flexibility in the $L^2$-$L^{\infty}$ estimate at the cost of some constant number of logarithms (which we will regain).
We bound the right-hand side of \eqref{eq:last_ramare} by
\begin{align}{\label{eq:after_heath_brown}}
    &\ll (\log X)^{2A+2} \cdot L \sum_{\ell=1}^L \int_{\mathcal{U}} \abs{ \prod_{j=1}^{J_{\ell}}M_{\ell,j}(1+it)}^2 \abs{R_{v, H}(1+it)}^2 dt \nonumber\\
    &\hspace{6cm} + (\log X)^{2A+2} \int_{\mathcal{U}}\abs{E(1+it)R_{v, H}(1+it)}^2 dt
\end{align}
using Cauchy--Schwarz.
Each $J_{\ell}$ is at most 6.
We upper bound the integral of the error term over $\mathcal{U}$ by the integral over the entire $[-T, T]$ interval.
By the mean value theorem, this error is 
$$\ll (\log X)^{2A+2} \left(\frac{T}{X}+1\right) (\log X)^{-3A-2}.$$
Because the Dirichlet polynomial is a product, we can factor the $L^2$ norm appearing on the right-hand side.
We bound the $L^2$ norm of the coefficients of $R_{v,H}(1+it)$ trivially by $\frac{1}{R_{v, H}}$, and as in the proof of item (7) of \cref{lem:heath_brown}, we can upper bound the $L^2$ norm of the coefficients of $E(1+it)$ by $\ll \frac{1}{P}(\log X)^{-(3A+2)}$, using of course that $\log P \approx \log X$ (up to $(\log X)^{o(1)}$ terms) and the fact that we can save arbitrarily many powers of $\log P$.
Additionally note that $E R_{v, H} \asymp X$, and so the error is $\ll (\log X)^{-(3A+2)}$. 
By choosing to save $3A+2$ powers of $\log X$ here, we ensure that the error from the Heath-Brown decomposition is $\ll (\log X)^{-A}$. 
We can take the supremum over $\ell$ of the summation and multiply by another factor of $L$.
All together, the right-hand side of \eqref{eq:after_heath_brown} can be bounded by 
$$\ll_C (\log X)^C \int_{\mathcal{U}} \abs{\prod_{i=1}^{J_{\ell}} M_i(1+it)}^2 |R_{v, H}(1+it)|^2 dt,$$
where $C$ is some absolute constant (absolute as we took $K=3$) coming from item (2) of \cref{lem:heath_brown}.

\textbf{Case IIa: }Suppose that there is a polynomial $M_i$ with length in $\left[Q_{v,H}^{\frac{1}{8}}, Q_{v,H}^{\frac{1}{3}}\right]$.
Without loss of generality suppose that this is $M_{J_{\ell}}(s)$. If $M_{J_{\ell}}$ has coefficients $\chi(n)\mu(n)$, apply \cref{lem:pnt_mobius} to $\abs{M_{J_{\ell}}(1+it)}$ to save, say, $A + 100$ more powers of $(\log X)$ as we lose from the decomposition (as well as any powers of $\log X$ we might lose from choosing $H = (\log X)^A$ and the Ramar\'e prefactors).
Alternatively, suppose $M_{J_{\ell}}$ has coefficients $\chi(n)$ or $\chi(n)\log n$. We know $M_{J_{\ell}}> Q_{v, H}^{\frac{1}{8}} > X^{\frac{1}{20}}$, hence we can apply \cref{lem:exponent_pairs} in order to save $(\log X)^{-R+1}$ (where the +1 comes from the fact that $M_{J_{\ell}}$ may correspond to the sum with coefficients $\chi(n)\log(n)$ in which case we apply partial summation at the cost of another $O(\log X)$ prefactor).
The right-hand side of \eqref{eq:last_ramare} is 
\begin{equation}{\label{eq:hm_case_2}}
    \ll_C (\log X)^{-R+C} \int_{\mathcal{U}} \abs{\prod_{i\leq J_{\ell}-1} M_i(1+it)}^2 \abs{R_{v,H}(1+it)}^2 dt
\end{equation}
The remaining polynomial has length at least $Q_{v,H}^{\frac{2}{3}} R_{v,H}$, which has length in $[X^{\frac{2}{3}}, X^{\frac{13}{15}}]$ (for the lower bound, consider $Q_{v, H}\approx X$ and $R_{v, H} \approx 1$, and for the upper bound consider $R_{v, H} = X^{\frac{6}{10}}$ and $Q_{v, H}= X^{\frac{4}{10}}$, so $Q_{v,H}^{\frac{2}{3}} R_{v,H}  = X^{\frac{13}{15}}$). In particular it is long enough that we may apply \cref{lem:hm} (again noting \eqref{eq:E_bound}).
As before, in this case we may upper bound \eqref{eq:hm_case_2} again by 
$$\ll \frac{1}{(\log X)^A}$$
for $R$ chosen sufficiently large in terms of $A$. 

\textbf{Case IIb: }Suppose there is \emph{no} polynomial $M_i(s)$ with length in $[Q_{v,H}^{\frac{1}{8}}, Q_{v,H}^{\frac{1}{3}}]$.
The product of $\prod_{i=1}^6 M_i(s)$ has length $Q_{v, H}^{1-o(1)}$, and therefore not all $M_i(s)$ can have lengths $ < Q_{v, H}^{\frac{1}{8}}$, and so by the pigeonhole principle, the hypothesis that there are no Dirichlet polynomials with length in $[Q_{v,H}^{\frac{1}{8}}, Q_{v,H}^{\frac{1}{3}}]$ tells us that at least one of the Dirichlet polynomials has length (strictly) at least $Q_{v, H}^{\frac{1}{3}}$. Hence by item (5) of \cref{lem:heath_brown}, it has coefficients $\chi(n)$ or $\chi(n)\log(n)$.
It will, in fact, be enough to have one such sum, even if all we know is that it has length $> Q_{v,H}^{\frac{1}{3}}$.
Without loss of generality let this be $M_1(s)$.
If any of these sums are of the form $\sum_{N\leq n\leq 2N}\chi(n)(\log n) n^{-1-it}$, we will apply partial summation at the cost of another logarithm, and account for it in our constant $C$.
We trivially bound any of the other remaining Dirichlet polynomials by $\ll 1$ (including $|R_{v,H}(1+it)|$).
Note that $M_1(s)$ has length in $(X^{\frac{4}{30}}, X]$, and we want to estimate 
\begin{equation}{\label{eq:zeta_sum_bound}}
    (\log X)^C \int_{\mathcal{U}} \abs{M_1(1+it)}^2 dt.
\end{equation}
We apply \cref{lem:exponent_pairs}, noting $|t|\leq X$
\begin{align*}
     \abs{M_1(1+it)} &\ll X^{\kappa} M_1^{\nu-\kappa-1+o(1)} + \frac{1}{1+|t|}\\
     &\ll  X^{-\kappa+o(1)} + \frac{1}{1+|t|},
\end{align*}

as $M_1 > X^{\frac{4}{30}} > X^{\frac{1}{24}}$. Again by passing to a well-spaced subset $\mathcal{E}\subseteq \mathcal{U}$, we have the bound

$$(\log X)^C \int_{\mathcal{U}} \abs{M_1(1+it)}^2 dt \ll (\log X)^C \sum_{t\in \mathcal{E}} X^{-2\kappa+o(1)} + (\log X)^C \sum_{t\in \mathcal{E}} \frac{1}{(1+|t|)^2}.$$

For the first term, we bound it trivially by $(\log X)^C |\mathcal{E}|X^{-\kappa}$, and by taking $\varepsilon = \frac{\kappa}{2}$, we get a power saving here.
For the second term, note that the $t\in \mathcal{E}$ are well-spaced, hence if $N_0 = \min_{t\in \mathcal{E}}|t|$, we have  
$$\sum_{t\in \mathcal{E}} \frac{1}{|t|^2} \leq \sum_{n = N_0}^{\infty} \frac{1}{n^2} \ll \frac{1}{N_0} \leq (\log X)^{-R}$$
because $\mathcal{E}\subseteq \mathcal{U}$.  
With an adequate choice of $R$, we conclude \eqref{eq:zeta_sum_bound} is certainly $\ll \frac{1}{(\log X)^A}$.
We take $\eta = \frac{\varepsilon}{4} = \frac{\kappa}{8}$.

This concludes the proof of \cref{prop:explicit_1_bound}.
\end{proof}
For our applications in \cref{sec:5}, we actually want a slightly different version of \cref{thm:liouville_bound} coming from \cref{prop:explicit_1_bound}.

\begin{corollary}{\label{cor:short_interval_major_arc}}
Fix $A, B > 0$. Let $X \geq h \geq 10$, and let $\mathcal{S}$ be as in \cref{prop:explicit_1_bound} with $\varepsilon = \frac{\kappa}{2}$ and $\eta = \frac{\varepsilon}{4}$. As before, in the definition of the quantities $P, Q$ (as given in \eqref{eq:PQ_choices}), take the quantity $R$ to be sufficiently large in terms of $A$, and set $Y = \frac{X}{(\log X)^{3R}}$. Let $q \leq (\log X)^B$ and let $\chi$ be a Dirichlet character mod $q$. Then we have that
$$
\frac{1}{X}\int_X^{2X}\abs{\frac{1}{h}\sum_{\substack{x\leq m\leq x+h\\ m\in \mathcal{S}}}\chi(m)\lambda(m)}^2 dx \ll_{A,B} \frac{(\log Q_1)^{\frac{1}{3}}}{P_1^{\frac{\varepsilon}{3}-\eta}} + \frac{1}{(\log X)^A}.
$$
\end{corollary}
The main difference here is that by retaining the restriction $m\in\mathcal{S}$, we can get the much stronger first term, along with the arbitrary power of logarithm savings in $X$.
Note that this does not follow from the Lipschitz estimate, as there we would pick up an error term of the form $\frac{\log P_1}{\log Q_1}$ rather than this stronger bound with power-saving in $P_1$.
We circumvent this by directly bounding the long moving average, which we may do firstly because $Y$ is not much smaller than $X$, and secondly because we know $X\leq x+Y \leq 3X$. 
\begin{proof}
    We will proceed similarly to the given proof of the Lipschitz estimate for the Liouville function.
    Observe that 
    \begin{align*}
    \frac{1}{X}\int_X^{2X}\abs{\frac{1}{h}\sum_{\substack{x\leq m\leq x+h\\ m\in \mathcal{S}}}\chi(m)\lambda(m)}^2 dx &\ll \frac{1}{X}\int_X^{2X}\abs{\frac{1}{h}\sum_{\substack{x\leq m\leq x+h\\ m\in \mathcal{S}}}\chi(m)\lambda(m) - \frac{1}{Y}\sum_{\substack{x\leq m\leq x+Y\\ m\in \mathcal{S}}}\chi(m)\lambda(m)}^2 dx \\
    &\quad\quad + \sup_{x\in[X,2X]}\abs{\frac{1}{Y}\sum_{\substack{x\leq m\leq x+Y\\ m\in \mathcal{S}}}\chi(m)\lambda(m)}^2.
    \end{align*}
    The first term is admissible by \cref{prop:explicit_1_bound} applied with $z = -1$. 
    The idea is that we know all of the primes in $\bigcup_{j=1}^J [P_j, Q_j]$ are at most $\exp\left(\frac{\log X}{\log\log X}\right)$ by our choice of $J$ and least restrictive choice of $X_0$.
    As in \cite{matomaki_radziwill}, we observe via inclusion-exclusion that 
    \begin{equation}{\label{eq:inclusion_exclusion}}
        \sum_{\substack{x \leq n \leq x+Y \\n\in\mathcal{S}}} \chi(n)\lambda(n) = \sum_{\mathcal{J} \subseteq \{1, \dots, J\}} (-1)^{\#\mathcal{J}} \sum_{x \leq n \leq x+Y} g_{\mathcal{J}}(n) \chi(n) \lambda(n),
    \end{equation}
    where for each subset $\mathcal{J}$ we define $g_{\mathcal{J}}$ to be the multiplicative function 
    $$g_{\mathcal{J}}(p^k) = \begin{cases}
        1, &p\notin \bigcup_{j\in\mathcal{J}} [P_j, Q_j]\\
        0, &\text{otherwise}
    \end{cases}.$$
    By definition of $J$, $\log Q_J \leq \frac{\log X}{\log\log X}$. As $Q_1$ is certainly lower bounded by say, 10, unwrapping the definition of $Q_J$ gives
    $$J^{4J+2} \leq \frac{\log X}{\log\log X}$$
    and so $J\log J \ll \log\log X$. Thus, $J \ll \frac{\log\log X}{\log\log\log X}$, and so in particular, $2^J = (\log X)^{o(1)}$. Thus, we have 
    $$\sum_{\substack{x \leq n \leq x+Y \\n\in\mathcal{S}}} \chi(n)\lambda(n) \ll (\log X)^{o(1)} \sup_{\mathcal{J}\subseteq J} \abs{\sum_{x\leq n\leq x+Y} g_{\mathcal{J}}(n) \chi(n)\lambda(n)},$$
    and we can therefore reduce our task to bounding the inner sum for some fixed $g_{\mathcal{J}}(n)$.
    We will deal with this using Perron's formula and contour integration.
    The (infinite) Dirichlet series associated to the multiplicative function $g_{\mathcal{J}}(n) \chi(n)\lambda(n)$ is 
    $$\frac{L(2s, \chi^2)}{L(s, \chi)} \prod_{p\in U_{\mathcal{J}}}\left(1+\frac{\chi(p)}{p^{s}}\right),$$
    where $U_{\mathcal{J}}$ is the union of intervals $[P_j, Q_j]$ for $j\in \mathcal{J}$. Hence
    $$\sum_{x\leq n\leq x+Y}g_{\mathcal{J}}(n)\chi(n)\lambda(n) = \frac{1}{2\pi i} \int_{1+\frac{1}{\log X}-iT}^{1+\frac{1}{\log X}+iT} \underbrace{\frac{L(2s, \chi^2)}{L(s, \chi)} \prod_{p\in U_{\mathcal{J}}}\left(1+\frac{\chi(p)}{p^{s}}\right)}_{F(s)} \frac{(x+Y)^{s}-x^{s}}{s}ds + O\left(\frac{X(\log X)^2}{T}\right).$$

    Firstly, the truncation error is bounded uniformly in $x \in [X, 2X]$, as it only relies on the fact that $X\leq x+Y\leq 3X$.
    In the zero-free region, $L(2s, \chi^2)$ is unproblematic (it is bounded in modulus by a constant for $\Re(s) > \frac{3}{4}$).
    We need a bound for the finite Euler product in the zero-free region, in order for the contour integration to produce a meaningful bound.
    A sufficient bound will say that the product is at most a bounded power of $\log X$, because we can win back (say) twice as many powers of $\log X$ via the contour integration step.
    The Euler product is bounded (in absolute value) by 
    \begin{equation}{\label{eq:log_EP}}
        \exp\left(\sum_{p\in U_\mathcal{J}} \frac{1}{p^{\Re(s)}} +O(1)\right).
    \end{equation}

    We will show that this contributes at most one power of $\log X$, and we will ultimately show that we can easily compensate for this loss.
    Let $\Re(s) = 1-c$, for $c = \frac{K \log \log X}{\log X}$ with $K = K(A)$ to be chosen soon. This is certainly within the zero free-region and to the right of any exceptional zero for the specified range of moduli $q$. Let $Y' = \exp\left(\frac{\log X}{\log\log X}\right),$ and observe that
    $$\sum_{p\leq Y'} \frac{p^c - 1}{p} = \sum_{p\leq Y'}\frac{e^{c\log p}-1}{p},$$
    and because for $z\in \R_{\geq 0}$, we have $e^z - 1 \leq ze^z$, it follows that 
    $$\sum_{p\leq Y'}\frac{e^{c\log p}-1}{p} \leq \frac{K}{\log Y'} \sum_{p\leq Y'}\frac{e^{c\log p} \log p}{p} \leq \frac{Ke^K}{\log Y'} \sum_{p\leq Y'}\frac{\log p}{p} \ll_K 1,$$
    where in the second to third inequality we  have used the fact that $c\log p \leq K$ in the exponent.
    Hence \eqref{eq:log_EP} differs from $\exp(\sum_{p\leq Y'} \frac{1}{p}) \ll \log Y' $ by a multiplicative constant depending on $K$, and hence is certainly bounded by $\ll\log X$. 

    After pulling the contour back to $\Re(s) = 1-c$, we must bound the Perron kernel uniformly in $x \in [X, 2X]$. Writing $(x+Y)^s - x^s = s\int_x^{x+Y}w^{s-1}\,dw$, for $\Re(s) = 1-c$ we have
    $$\abs{\frac{(x+Y)^s - x^s}{s}} \leq Y\max_{x\leq w\leq x+Y}w^{-c} \ll YX^{-c}$$
    and trivially
    $$ \abs{\frac{(x+Y)^s - x^s}{s}} \leq \frac{(x+Y)^{1-c} + x^{1-c}}{|s|} \ll \frac{X^{1-c}}{1 + |t|}.
    $$
    By splitting the $t$ integral depending on whether or not $|t|\leq \frac{X}{Y}$ we have
    \begin{align*}
    \int_{-T}^{T}|F(1-c+it)|\abs{\frac{(x+Y)^{1-c + it} - x^{1-c + it}}{1-c + it}} dt
    &\ll (\log X)^{O(1)}\left(YX^{-c}(\log X)^{3R} + X^{1-c}\int_{(\log X)^{3R}}^{T}\frac{dt}{t}\right) \\
    &\ll X^{1-c}(\log X)^{O(1)},
    \end{align*}
    where crucially the constant power of $\log X$ is some fixed constant depending only on $B$. 
    Observe that by pulling the contour back to the line $1-c$, we save
    $$X^{-c} = (\log X)^{-K},$$
    and by taking $K$ sufficiently large in terms of $R$ (which is chosen sufficiently large in terms of $A$), $A$ and $B$, the desired bound holds.
\end{proof}

\section{Improved bounds for general multiplicative functions}{\label{sec:4}}

\subsection{Multiplicative functions with smooth support}{\label{sec:4.1}}
One can apply essentially the same strategy in order to improve bounds for multiplicative functions $f$, provided that $f$ is supported only on integers with no prime factor greater than $X^{1-\delta}$, where $\delta > 0$ is some fixed but arbitrarily small positive number.

In the last range, we will apply the Perron variant of Ramar\'e's identity to obtain some bilinear structure in the integrand with an arbitrary multiplicative function $f$.
In the general case, we cannot apply Heath-Brown's identity, as we would have no quantitative control over the Dirichlet polynomials in the factorization.
In particular, we no longer have pointwise bounds that enable us to save arbitrary powers of $\log X$ (or small powers of $X$ in the case of zeta sums), hence we need to rely on Hal\'asz--Montgomery type bounds for Dirichlet polynomials supported on primes.
These bounds cannot make use of the fact that $Q_{v, H}(s)$ is a short Dirichlet polynomial, and hence would cost us in our analysis.
In the Perron variant of Ramar\'e's identity, we no longer have polynomials supported on short intervals, so we drop the subscript $H$ from the Dirichlet polynomials it produces, and we obtain sharper control of error terms at the cost of some extra bookkeeping.
As in \cite{matomaki_radziwill}, we will need to use the fact that $Q_{v}(s)$ is supported on primes to recover two logarithms, and we will need Hal\'asz' theorem to obtain some $L^{\infty}$ savings on $R_{v}(s)$. 
If $R_{v}(s)$ is \emph{too short}, we will not be able to guarantee sufficient savings, and it's possible that $f$ might be significantly more or less pretentious at scale $R_v$ than at scale $X$.
The smooth support condition for $f$ ensures that when we apply \cref{lem:perron_ramare}, $R_{v}(s)$ is sufficiently long.
Following the proof of \cref{thm:smooth_f_thm} we discuss this a bit more.
The precise variant of Hal\'asz theorem that we will use is as follows.

\begin{lemma}[Hal\'asz theorem]{\label{lem:halasz_OG}}
    Let $f:\N\to \C$ be a multiplicative function with $\abs{f(n)}\leq 1$. Let $X\geq 3$ and $U\geq 1$. Let
    $$M(f;X,U) = \min_{|t|\leq U} \mathbb{D}(f, n\mapsto n^{it}; X)^2.$$
    Then
    $$\frac{1}{X}\abs{\sum_{n\leq X} f(n)} \ll (M(f;X, U) + 1) \exp(-M(f;X, U)) + \frac{1}{U} + \frac{\log\log X}{\log X}.$$
\end{lemma}
In our notation, we take $U=X$ and so the $U$ is usually omitted from the $M(f;X)$ notation, although the additional flexibility will be useful in the proofs presented below.
\begin{proof}
    See Corollary 1 of \cite{Granville_Soundararajan_2003}. 
\end{proof}

\begin{remark}{\label{rmk:s_plus}}
    In this setting (and in the remainder of \cref{sec:4}), we will \emph{start} with an unrestricted short average. We will remove integers with atypical factorizations on the frequency side, using \cref{lem:add_back_sieve_error} and \cref{lem:abse_2}. This changes very little, it just makes the bookkeeping at the end of the argument slightly easier.
    The choice of $R$ in \eqref{eq:PQ_choices} is less consequential as well, so we fix $R=2$. 
\end{remark}

The new ideas again relate to a sharper treatment of the last range $[P,Q]$ and the corresponding Dirichlet polynomials $Q_{v}(s)$ and $R_{v}(s)$. The proof of \cref{thm:smooth_f_thm} will require two lemmas, a variant of Hal\'asz's theorem with the Ramar\'e weights, and then a local version of the same.

We will need to apply an $L^{\infty}$ bound coming from Hal\'asz theorem to $R_{v}(s)$ in order to obtain some pointwise savings.
If we used Hal\'asz's theorem as in the proof of Lemma 3 in the Matom\"aki--Radziwi\l{}\l{} paper, we get the following bound: if
$$R(s) = \sum_{X\leq n\leq 3X} \frac{f(n)}{n^s}\frac{1}{\omega_{[P,Q]}(n)+1}$$
then for any $t\in [(\log X)^{\frac{1}{16}}, X^A]$
$$\abs{R(1+it)} \ll \frac{\log Q}{\exp(M(f;X))\log P} + \log X \exp\left(-\frac{\log X}{3\log Q}\log\frac{\log X}{\log Q}\right).$$
In \cite{matomaki_radziwill}, this theorem is applied to $R_{v, H}(s)$ with an interval $[P,Q]$ with less separation in between $P$ and $Q$ (and hence $Q$ is quite small).
There, $Q$ has size $\exp\left(\frac{\log X}{\log\log X}\right) = X^{o(1)}$. Now the parameter $Q$ is $X^{1-o(1)}$.
This would mean we end up \emph{paying} $\log X$ rather than saving $\exp(-M(f;X))$.
We modify their result in a manner that retains the main term without incurring this error.

\begin{lemma}[Ramar\'e weighted Hal\'asz inequality]{\label{lem:modified_R_halasz}}
    Let $10 < Y < \frac{X}{2}$. Let $P, Q$ be defined in terms of $X$ as in \eqref{eq:PQ_choices}, again taking $R=2$ in the definition of $P$. 
    Let $f:\N\to \C$ be a 1-bounded multiplicative function, and define
    $$R(s) = \sum_{Y\leq n \leq 2Y} \frac{f(n)}{n^s} \frac{1}{\omega_{[P, Q]}(n)+1}.$$
    Write $M(f;Y) = \min_{|\tau|\leq X}\mathbb{D}(f, n\mapsto n^{i\tau}; Y)^2$. 
    Then for any $t\in \left[0, \frac{X}{2}\right]$, we have
    $$\abs{R(1+it)} \ll (\log\log X)^4 (M(f;Y)+1) \exp(-M(f;Y)) + \frac{\log\log Y}{\log Y}.$$
\end{lemma}
Note that in the definition of $M(f;Y)$, we take the minimum over the larger window $|\tau|\leq X$, related to the ambient scale $X$ of the problem. This distance is certainly at most the minimum over the shorter scale $|t|\leq Y$. 
\begin{proof}
    Observe that $\omega_{[P,Q]}(n)+1$ is a positive integer.
    Hence, 
    $$\frac{1}{\omega_{[P,Q]}(n)+1} = \int_0^1 y^{\omega_{[P,Q]}(n)}dy.$$
    Crucially the integrand (for any fixed $y$) \emph{is} multiplicative.
    Now we exchange the order of summation and integration to see
    $$R(s) = \int_0^1 \sum_{Y\leq n \leq 2Y} \frac{f(n) y^{\omega_{[P,Q]}(n)}}{n^s} dy := \int_0^1 F_y(s) dy,$$
    where $F_y(s)$ for each $y$ is a Dirichlet polynomial corresponding to a multiplicative function.
    We now pull out a supremum over $y$ and let $s=1+it$:
    $$\abs{R(1+it)} \ll \sup_{y\in [0,1]}\abs{F_y(1+it)}.$$
    We can now freely apply \cref{lem:halasz_OG}, although now the multiplicative function will be different. Let $g_{y,t}(n) = f(n)y^{\omega_{[P,Q]}(n)}n^{-it}$. By Hal\'asz theorem and partial summation for the $\frac{1}{n}$ weights,
    \begin{align}{\label{eq:halasz_smaller}}
        &\sup_{y\in [0,1]}\abs{F_y(1+it)}\\
        &\ll \left(\min_{|\tau'|\leq Y} \mathbb{D}(g_{y,t}, n\mapsto n^{i\tau'};Y)^2 +1\right)\exp\left(-\min_{|\tau'|\leq Y} \mathbb{D}(g_{y,t}, n\mapsto n^{i\tau'};Y)^2\right) + \frac{\log\log Y}{\log Y} \nonumber\\
        &\ll \underbrace{\left(\min_{|\tau| < X} \mathbb{D}(f y^{\omega_{[P, Q]}(\cdot)}, n\mapsto n^{i\tau}; Y)^2+1\right)}_{M_y(f;Y)+1} \exp\left(- \min_{|\tau| < X} \mathbb{D}(f y^{\omega_{[P, Q]}(\cdot)}, n\mapsto n^{i\tau}; Y)^2\right) + \frac{\log\log Y}{\log Y}. \nonumber
    \end{align}
    Here, we used the fact that $0\leq t\leq \frac{X}{2}$, and so $|\tau| := |\tau' + t| \leq X$.
    We prove that the difference in the corresponding $M$ coming from pretentious distance is negligible, bounded by at most $O(\log\log\log X)$.
    The idea is that for prime $p$,
    $$y^{\omega_{[P,Q]}(p^k)} =\begin{cases}
        y, &p\in [P,Q]\\
        1, &\text{otherwise}
    \end{cases}.$$
    Hence for any multiplicative function $g(n)$ (in particular, an arbitrary Archimedean character) and any scale $T_0$, the sum corresponding to $\mathbb{D}(f, g;T_0)^2$ differs from $\mathbb{D}(f y^{\omega_{[P,Q]}},g ;T_0)^2$ only on primes in $[P,Q]$.
    The difference in numerators is bounded, hence the difference is certainly at most 
    $$\leq 2\sum_{P\leq p\leq Q}\frac{1}{p} \leq 2\log\log Q - 2\log\log P \leq 4\log\log\log X,$$
    for sufficiently large $X$. Hence $\exp(-M_y(f;Y)) \ll (\log\log X)^4 \exp(-M(f;Y))$ and the stated upper bound holds.    
\end{proof}

\begin{remark}
    To verify that \cref{lem:modified_R_halasz} is suitable for the second integral on the right-hand side of \cref{prop:parseval_1}, observe that $|\tau| \leq 2T + Y \leq X$, hence the distance is still controlled by frequencies up to $X$. 
\end{remark}

Secondly, we would like a ``local'' Hal\'asz theorem over a shorter range of frequencies. Ultimately, this will be applied when the function is suitably $n^{it}$-pretentious, hence we need only save a small power of logarithm in our bound.

\begin{lemma}{\label{lem:local_halasz}}
    Fix $0 < \delta < 1$ and let $X^{\delta} < Y < \frac{X}{2}$. Let $P, Q$ be defined in terms of $X$ as in \eqref{eq:PQ_choices}. 
    Let $f:\N\to \C$ be a 1-bounded multiplicative function, and define the Dirichlet polynomial $R(s)$ and the quantity $M(f;Y)$ as in \cref{lem:modified_R_halasz}.
    Let $t_1$ be a frequency that minimizes $M(f;Y)$ (again over the ambient scale $\abs{\tau}\leq X$).
    Then for $|t|\leq X$, whenever $|t-t_1| > \frac{5}{6}(\log X)^{\frac{1}{16}}$, we have
    $$\abs{R(1+it)} \ll_{\delta} (\log X)^{-\frac{1}{20}}.$$
\end{lemma}

\begin{proof}
Let $U = \frac{(\log X)^\frac{1}{16}}{3}$. We apply the same argument as in \cref{lem:modified_R_halasz} to deal with the Ramar\'e weight $\frac{1}{\omega_{[P,Q](n)}+1}$, the only difference is that in the application of the usual Hal\'asz theorem, we only take the minimum over frequencies $|u|\leq U$. Let $g_y(n) = f(n) y^{\omega_{[P,Q]}(n)}$. By the same application of Hal\'asz theorem as before, we have

    $$\abs{R(1+it)} \ll_{\delta} \left(\min_{|u|\leq U} \mathbb{D}(g_y, n\mapsto n^{i(t+u)};X^{\delta})^2+1\right) \exp\left(-\mathbb{D}(g_y, n\mapsto n^{i(t+u)};X^{\delta})^2\right) + \frac{1}{U} + \frac{\log\log X}{\log X}.$$

We can replace the pretentious distance up to $X^{\delta}$ by the distance up to $X$ at the cost of $O_{\delta}(1)$. Similarly, as before, we can account for the $y^{\omega_{[P,Q]}(n)}$ factor in the squared pretentious distance at the cost of $O(\log\log\log X)$. In Lemma A.4 of \cite{average_chowla}, Matom\"aki--Radziwi\l{}\l{}--Tao prove that when $|t-t_1| > (\log Y)^{\frac{1}{16}}/2$, we have 
$$\mathbb{D}(f, n\mapsto n^{it};Y)^2 \geq \left(\frac{1}{3}-\frac{2}{3\pi}-\epsilon_0\right)\log \log Y$$
for any fixed $\epsilon_0 > 0$. 
Because $X^{\delta} < Y < X$, this implies
$$\mathbb{D}(f, n\mapsto n^{it};X)^2 \geq \left(\frac{1}{3}-\frac{2}{3\pi}-\epsilon_0+o_{\delta}(1)\right)\log \log X.$$
In particular, we have that this pretentious distance is at least $\frac{1}{10}\log\log X$ for $X$ sufficiently large in terms of $\delta$. By the triangle inequality and by choice of $U$, we have $|t+u -t_1| \geq |t-t_1| - U > (\log X)^{\frac{1}{16}}/2$, and therefore upon applying \cref{lem:halasz_OG} we have
$$\abs{R(1+it)} \ll_{\delta} (\log X)^{-\frac{1}{10}+o(1)} + (\log X)^{-\frac{1}{16}} + \frac{\log\log X}{\log X}.$$

Hence, the desired bound holds.

\end{proof}

With these two pointwise bounds in hand, we can proceed with the proof.

\begin{proof}[Proof of \cref{thm:smooth_f_thm}.]
Write
$$F(s) = \sum_{\substack{X\leq n\leq 3X}} f(n) n^{-s}.$$
 We have the following Parseval bound (see the proof of Theorem A.2 in \cite{average_chowla}):
\begin{equation}{\label{eq:MRT_parseval}}
    \frac{1}{X}\int_X^{2X} \abs{\frac{1}{h} \sum_{\substack{x\leq n\leq x+h}} f(n)}^2 dx \ll \int_{0}^{X/h}\abs{F(1+it)}^2 dt + \max_{T \in \left(\frac{X}{h}, \frac{X}{8}\right]}\frac{X/h}{T}\int_{T}^{2T} \abs{F(1+it)}^2 dt + \frac{1}{h},
\end{equation}
where the same comment as in \cref{rmk:T_range} justifies the truncation of the $T$ range in the second integral.
Note here that we do not separate the contribution of small frequencies $t\leq T_0$. 

We will again focus on the first integral, as the treatment of the second integral (with respect to how we compute the first integral) is completely analogous to the proof of \cref{thm:liouville_bound}.

We will prove that, for $T = \frac{X}{h}$, we have
\begin{equation}{\label{eq:relevant_integral}}
    \int_0^T \abs{F(1+it)}^2 dt \ll_{\delta} \exp(-(2-o(1))M(f;X)) + \frac{(\log\log h)^2}{(\log h)^2} + \frac{1}{(\log X)^{2-o(1)}}.
\end{equation}

The first thing we will do is separate out the contributions of integers with atypical factorizations. Observe
$$\int_0^{X/h} \abs{F(1+it)}^2 dt \ll \int_0^{X/h} \abs{\sum_{\substack{X\leq n\leq 3X \\n\in \mathcal{S}}} f(n) n^{-1-it}}^2 dt + \int_0^{X/h} \abs{\sum_{\substack{X\leq n\leq 3X \\n\notin \mathcal{S}}} f(n) n^{-1-it}}^2 dt,$$
and by \cref{lem:add_back_sieve_error} we have that the second integral can be bounded by
$$\ll \frac{(\log \log h)^2}{(\log h)^2} + \frac{(\log\log X)^4}{(\log X)^2},$$
which is admissible.

We will now consider two cases; either $M(f;X) > \frac{1}{32}\log\log X$ (so $f$ is ``strongly non-pretentious''), or $M(f;X) \leq \frac{1}{32} \log\log X$.
In the second case, we will need to apply the method in \cite{average_chowla}, although we modify it to get better savings in the pretentious distance term.

\textbf{Case I: }Suppose $f$ is strongly non-pretentious. We will proceed with the analysis of \eqref{eq:relevant_integral} in 3 cases, depending on the length of the prime supported Dirichlet polynomial when we apply Ramar\'e's identity. In all of the following subcases, any time that we apply any $L^2$ estimate for Dirichlet polynomials, we may apply these bounds on the 1-line due to dyadic support of the polynomials and partial summation.

As before, we partition $[0,T] = \mathcal{T}_1 \cup \dots \cup \mathcal{T}_J \cup \mathcal{U}$, and (in light of \cref{eq:same_integrals}) we may bound the integrals over $\mathcal{T}_1, \dots, \mathcal{T}_J$, recovering the bounds given in \eqref{eq:integral_T_1} and \eqref{eq:integral_T_j}).

This time, when we are left with the integral over $\mathcal{U}$, we will add back the integers $n\notin \mathcal{S}$ using \cref{lem:add_back_sieve_error}.
We now separate the integers with no prime factor in the range $[P,Q]$.
We have
\begin{align*}
\int_{\mathcal{U}}\abs{\sum_{\substack{X\leq n\leq 3X}} f(n)n^{-1-it}}^2 dt &\ll \int_{\mathcal{U}}\abs{\sum_{\substack{X\leq n\leq 3X \\ \exists p\in [P,Q], p|n}} f(n)n^{-1-it}}^2 dt +  \int_{\mathcal{U}}\abs{\sum_{\substack{X\leq n\leq 3X \\ p|n\implies p\notin [P,Q]}} f(n)n^{-1-it}}^2 dt\\
    &\ll \int_{\mathcal{U}}\abs{\sum_{\substack{X\leq n\leq 3X \\ \exists p\in [P,Q], p|n}} f(n)n^{-1-it}}^2 dt + (\log X)^{-2+o(1)}.
\end{align*}
In the second line, we have used the mean value theorem along with \eqref{lem:last_range_sieve_error} with $R=2$. With $Q$ as in \cref{eq:PQ_choices}, if $p|n\implies p\notin [P,Q]$ then $n$ is $P$-smooth.
The density of such integers is $\ll (\log X)^{-2+o(1)}$ with our choice of $R=2$. 

So now we consider the first integral over $\mathcal{U}$.
We apply \cref{lem:perron_ramare} with the interval parameters $[P,Q]$ and $D = (\log X)^2$, again crudely bounding the sum over $v\in [\log P, \log Q]$ by the length of the integral times some choice of $v$ maximizing the expression:
\begin{align}
\int_{\mathcal{U}}\abs{\sum_{\substack{X\leq n\leq 3X \\ \exists p\in [P,Q], p|n}} \frac{f(n)}{n^{1+it}}}^2dt &\ll (\log\log X)(\log Q)^2 \int_{-(\log X)^2}^{(\log X)^2} \abs{\frac{3^{iu}-1}{iu}} \int_{\mathcal{U}+u} \abs{Q_{v}(1+it) R_{v}(1+it)}^2 dtdu + \frac{1}{(\log X)^2} \nonumber\\
&\ll (\log\log X) (\log Q)^2 \int_{-(\log X)^2}^{(\log X)^2} \abs{\frac{3^{iu}-1}{iu}} \sum_{t\in \mathcal{E}_u} \abs{Q_{v}(1+it)R_{v}(1+it)}^2 du + \frac{1}{(\log X)^2}, \label{eq:discrete_last_ramare}
\end{align}
where $Q_{v}(s)$ and $R_{v}(s)$ have coefficients $f(p)$ and $\frac{f(m)}{\omega_{[P,Q]}(m)+1}$ respectively, and where in the second line we have passed to a well-spaced subset $\mathcal{E}_u\subseteq\mathcal{U}+u$.

We note two important observations:
\begin{enumerate}
    \item We have added back the integers $n\notin \mathcal{S}$, so after applying Ramar\'e's identity on the last interval, the polynomial $R_{v}(s)$ is exactly the sum of a multiplicative function over an interval, and there are no extra conditions on $n$.
    \item The prime-supported Dirichlet polynomial $Q_{v}(s)$ has length in $\left[P, X^{1-\delta}\right]$ rather than $[P,Q]$ due to the imposed support condition on $f$.
    Consequently, the polynomial $R_{v}(s)$ is quite long, of length \emph{at least} $X^{\delta}$.
\end{enumerate}
Both of these facts will be important when we apply Hal\'asz theorem. 
We will consider cases based on the length of $Q_v(s)$. Note importantly that for the remainder of the argument, we will bound the sum over $t\in \mathcal{E}_u$ uniformly in $u$, and at the end integrate the $u$ kernel at the cost of another (admissible) factor of $\log\log X$.

\textbf{Case Ia: } Suppose that $P < Q_{v} < X^{\frac{4}{10}}$. We split the set $\mathcal{E}_u$ into
$$\mathcal{E}_{u, \operatorname{small}} = \left\{t\in \mathcal{E}_u: \abs{Q_{v}(1+it)} \leq (\log X)^{-100}\right\},$$
$$\mathcal{E}_{u, \operatorname{large}} = \left\{t\in \mathcal{E}_u: \abs{Q_{v}(1+it)} > (\log X)^{-100}\right\}.$$
On the small set, we can bound the right-hand side of \eqref{eq:discrete_last_ramare} by
$$\ll (\log X)^{-200+2+2+1+o(1)} \sum_{t\in\mathcal{E}_{u, \operatorname{small}}}\abs{R_{v}(1+it)}^2,$$
and because $Q_{v} < X^{\frac{4}{10}}$, we have $R_{v} \gg X^{\frac{6}{10}}$ and hence we may apply the usual Hal\'asz--Montgomery inequality (\cref{lem:hm}) to upper bound the above by $\ll (\log X)^{-93}$, which is admissible. Note here that $\abs{\mathcal{U}+u} = \abs{\mathcal{U}}$, hence we have the same upper bound on $\abs{\mathcal{E}_{u, \operatorname{small}}} \leq \abs{\mathcal{E}_u} \ll X^{\varepsilon}$.

By \cref{lem:large_values} and the lower bound $Q_{v} > P = \exp\left(\frac{(\log X)(\log\log\log X}{\log\log X}\right)$, the well-spaced exceptional set has size at most
\begin{align*}
\abs{\mathcal{E}_{u, \operatorname{large}}} &\ll X^{2\frac{\log\log X}{\log P}} (\log X)^{200}\exp\left(2\frac{\log T}{\log P}\log\log T\right)\\
&\ll \exp\left((\log\log X)^{2+o(1)}\right).
\end{align*}
So, we apply \cref{lem:modified_R_halasz} to $R_{v}(1+it)$ and apply \cref{lem:prime_hm_2} to the remaining sum over $t\in \mathcal{E}_{u, \operatorname{large}}$ to get
\begin{align*}
    (\log\log X)^2(\log Q)^2&\sum_{t\in\mathcal{E}_{u, \operatorname{large}}}\abs{Q_{v}(1+it)R_{v}(1+it)}^2 \\
    &\ll (\log\log X)^2 (\log Q)^2 \left(\frac{(\log\log X)^8 M(f;X)^2}{\exp(2M(f;X))} + \frac{(\log\log X)^2}{(\log X)^2}\right)\sum_{t\in\mathcal{E}_{u, \operatorname{large}}} \abs{Q_{v}(1+it)}^2 \\
    &\ll (\log\log X)^2\frac{(\log Q)^2}{(\log P)^2} \left(\frac{(\log\log X)^8 M(f;X)^2}{\exp(2M(f;X))} + \frac{(\log\log X)^2}{(\log X)^2}\right).
\end{align*}

Note that in going from the second to third line, the large values set is sparse enough that the off-diagonal term in \cref{lem:prime_hm_2} is smaller than the main term of $\frac{e^v}{\log P}$. We have also used that $P \leq e^{v} \leq Q$.  In total, our savings is $\frac{1}{(\log P)^2}$, where we get one factor of $\frac{1}{\log P}$ from \cref{lem:prime_hm_2}, and the other from applying the prime number theorem to the $L^2$ norm in the Hal\'asz--Montgomery bound.

By the Case I hypothesis, $f$ is strongly non-pretentious, and so any $(\log\log X)^{O(1)}$ losses can be absorbed into the $o(1)$ in the exponential of $M(f;X)$, and we can bound the right-hand side of \eqref{eq:discrete_last_ramare} by $\ll \exp(-(2-o(1)) M(f;X)) + (\log X)^{-2+o(1)}$.

\textbf{Case Ib: }Suppose that $X^{\frac{4}{10}} < Q_{v} < X^{\frac{6}{10}}$. We split $\mathcal{E}_u$ into small and large sets in the same way, the only difference will be the treatment of the small set (the treatment of the large set will again be exactly the same, as all we used was that $P < Q_{v}$ and $R_v > X^{\delta}$ for $\delta$ positive).
On the small set, we bound 
\begin{align*}
    (\log\log X)^2 (\log Q)^2\sum_{t\in\mathcal{E}_{u, \operatorname{small}}}\abs{Q_{v}(1+it)R_{v}(1+it)}^2 &\ll (\log X)^{-94} \sum_{t\in\mathcal{E}_{u, \operatorname{small}}}\abs{R_{v}(1+it)}^2 \\
    &\ll (\log X)^{-93},
\end{align*}
where in the second line we have used \cref{lem:exp_pair_hm} rather than \cref{lem:hm}. Here all we have asked is that $Q_{v} < X^{\frac{6}{10}}$, so $R_{v} \gg X^{\frac{4}{10}}$, hence ordinary Hal\'asz--Montgomery would be insufficient. But, because $\frac{1}{6}+\frac{2}{10} +\varepsilon < \frac{4}{10}$ for the choice of $\varepsilon$, the polynomial $R_{v}(s)$ is still long enough to grant us savings. Again, as the procedure for $\mathcal{E}_{u, \operatorname{large}}$ is \emph{exactly} the same as in Case Ia, we give the same bound for the right-hand side of \eqref{eq:discrete_last_ramare}.

\textbf{Case Ic: } Lastly suppose $X^{\frac{6}{10}} < Q_{v} < X^{1-\delta}$ (if this region is empty then we need not consider this case). We will apply \cref{lem:modified_R_halasz} to $R_{v}(s)$, and importantly because it has length at least $X^\delta$, we know that $\exp(-(2-o(1))M(f;R_v)) \asymp_{\delta} \exp(-(2-o(1))M(f;X))$.
To see this, observe that $M(f;X)$ and $M(f;X^{\delta})$ differ by at most $O\left(\log\left(1/\delta\right)\right)$, which (as $\delta$ was fixed) does not cost us anything when we take exponentials.
Hence, the right-hand side of \eqref{eq:discrete_last_ramare} is 
$$\ll_{\delta} (\log\log X)^2 (\log Q)^2 \left(\frac{(\log\log X)^8 M(f;X)^2}{\exp(2M(f;X))} + \frac{(\log\log X)^2}{(\log X)^2}\right)\sum_{t\in \mathcal{E}_u} \abs{Q_{v}(1+it)}^2. $$

On the right-hand side, we have the $L^2$ norm of a long Dirichlet polynomial supported on primes, and so by \cref{lem:prime_hm}, we have
\begin{align*}
    \sum_{t\in \mathcal{E}_u}\abs{Q_{v}(1+it)}^2 &\ll \left( \frac{e^{v}}{\log P} + \abs{\mathcal{E}_u} T^{\frac{1}{2}} Q_{v}^{\frac{1}{20}} \log T\right)\sum_{e^{v}\leq p\leq e^{(v+1)}}\frac{1}{p^2}\\
    &\ll \frac{1}{(\log P)^2}.
\end{align*}

Again, because $\frac{3}{100}+\frac{1}{2} + \varepsilon < \frac{6}{10}$, the off-diagonal term is unproblematic, and we get $(\log P)^2$ savings from \cref{lem:prime_hm}, and the prime number theorem.
Hence, all together we may bound 
\begin{align}{\label{eq:4.1_U_integral}}
    (\log \log X)^2 \frac{(\log Q)^2}{(\log P)^2} &\left((\log\log X)^8 M(f;X)^2\exp(-2M(f;X)) + \frac{(\log\log X)^2}{(\log X)^{2}}\right) \nonumber\\
    &\ll \frac{(\log\log X)^{14-o(1)}}{\exp(2M(f;X))} + \frac{(\log\log X)^{6-o(1)}}{(\log X)^{2}}.
\end{align}
Once again because $f$ is strongly non-pretentious, these factors of order $(\log\log X)^C$ are absorbed by $\exp(-(2-o(1)M(f;X))$ or $(\log X)^{-2+o(1)}$, and we get a bound of the desired shape.

Hence, all together, if $f$ is strongly non-pretentious, we conclude that 
\begin{align}{\label{eq:last_error_bound}}
\frac{1}{X}\int_{X}^{2X} \abs{\frac{1}{h}\sum_{\substack{x\leq n\leq x+h}} f(n)}^2 dx \ll \frac{(\log h)^{\frac{1}{3}}}{P_1^{\frac{\varepsilon}{3}-\eta}}+ \exp(-(2-o(1))M(f;X)) + \frac{(\log\log h)^2}{(\log h)^2} + \frac{1}{(\log X)^{2-o(1)}}.
\end{align}

Note that the sieve error that we incurred by adding back $n\notin \mathcal{S}$ dominates the error term of the form $\frac{(\log h)^{\frac{1}{3}}}{P_1^{\frac{\varepsilon}{3} -\eta}}$ from the integrals $\mathcal{T}_1, \dots, \mathcal{T}_J$ (for our choices of $\varepsilon, \eta, P_1, Q_1$). This concludes the analysis of strongly non-pretentious $f$. 

\textbf{Case II: } Now suppose that $f$ satisfies $M(f;X) \leq \frac{1}{32} \log\log X$. This will mean that the term $\exp(-2M(f;X))$ dominates error terms that decay faster than $(\log X)^{-\frac{1}{16}}$ that we will incur in the argument.
We follow the strategy in \cite{average_chowla}, however as stated earlier, the method of adding back the sieve error enables greater savings in the term involving $M(f;X)$. 
Let $t_1$ be a frequency satisfying $|t_1|\leq X$ minimizing  $\mathbb{D}(f, n\mapsto n^{it_1};X)^2$, and define the set of points close to $t_1$ and far from $t_1$:

$$\mathcal{C} = \left\{t\in [0, T]: |t-t_1| \leq (\log X)^{\frac{1}{16}}\right\},$$
$$\mathcal{F} = \left\{t\in [0, T]: |t-t_1| > (\log X)^{\frac{1}{16}}\right\}.$$

On the far set $\mathcal{F}$, we can run a very similar argument as in Case I. 
We will use \cref{lem:add_back_sieve_error} to separate the contribution of integers with atypical factorizations.
We will partition $\mathcal{F} = \mathcal{T}_1 \cup \dots \mathcal{T}_J \cup \mathcal{U}$, and we handle the first $J$ sets exactly as before. 

In the $\mathcal{U}$ range, we again separate integers with atypical factorizations (these are $P$-smooth integers) using the mean value theorem, and we apply \cref{lem:perron_ramare} with $D = (\log X)^{\frac{1}{16}}/100$ at the cost of a $(\log X)^{-\frac{1}{16}}$ in the truncation error term and again a $\log\log X$ prefactor. Now, when we bound the integral over $\mathcal{U}+u$ uniformly in $u$, we can again use the idea 
of considering 3 regimes for the size of $Q_{v}$, except we replace \cref{lem:modified_R_halasz} by \cref{lem:local_halasz} when bounding $\abs{R_{v}(1+it)}$ pointwise. Observe that
$$\abs{t + u - t_1} \geq |t-t_1| - |u| > \frac{99}{100} (\log X)^{\frac{1}{16}},$$
and so even after we shift again in \cref{lem:local_halasz}, all of the frequencies have distance at least $(\log X)^{\frac{1}{16}}$ from $t_1$, and hence the Matom\"aki--Radziwi\l{}\l{}--Tao pretentious distance lower bound holds. The conclusion is that the $\exp(-(2-o(1))M(f;X))$ term in the upper bound is replaced by $(\log X)^{-\frac{1}{10}-o(1)}$. Given that this is better than the Perron truncation error, we have
$$\int_{\mathcal{F}}\abs{\sum_{\substack{X\leq n\leq 3X}} \frac{f(n)}{n^{1+it}}}^2 dt \ll (\log X)^{-\frac{1}{16} + o(1)} + \frac{(\log\log h)^2}{(\log h)^2}.$$

For $t\in \mathcal{C}$, we'd like an upper bound for the following
$$\int_{\mathcal{C}}\abs{\sum_{X\leq n\leq 3X} \frac{f(n)}{n^{1+it}}}^2 dt.$$

Now, the integrand is a (complete) dyadic sum of a multiplicative function.
We apply the following re-centering lemma (Lemma 7.1 of \cite{Granville_Soundararajan_2003}) in order to bound 
$$\sum_{X\leq n\leq 3X}f(n) n^{-it} = \frac{X^{i(t_1-t)}}{1+i(t_1-t)} \sum_{X\leq n \leq 3X} f(n) n^{-it_1} + O\left(X \frac{\log\log X}{\log X} \exp\left(\sum_{p\leq X}\frac{|1-f(p)p^{-it_1}|}{p}\right)\right).$$
Because the sum over $f$ is unrestricted (that is, we don't need to separate any behavior at prime intervals, as in \cite{average_chowla}), we may use the inequality $|1-z| \leq \sqrt{2(1-\Re(z))}$ and Cauchy--Schwarz, 
\begin{align*}
    \sum_{p\leq X} \frac{|1-f(p)p^{-it_1}|}{p} &\leq \left(\sqrt{2\sum_{p\leq X}\frac{1}{p}}\right) \cdot\mathbb{D}(f, n\mapsto n^{it_1} ;X)\\
    &\leq \frac{1}{4}\log\log X. 
\end{align*}
This follows from our hypothesis $M(f;X) \leq \frac{1}{32}\log\log X$ and Mertens' theorem.
Thus, the error from this re-centering is at most $O\left(X (\log X)^{-\frac{3}{4}+o(1)}\right)$.
Now, we can apply Hal\'asz theorem to the main term on the right-hand side, and we arrive at the pointwise bound 
$$ \abs{\sum_{X\leq n\leq 3X} f(n) n^{-1-it}} \ll \frac{(M(f;X)+1) \exp(-M(f;X))}{1+|t-t_1|} + O\left(\frac{1}{(\log X)^{\frac{3}{4}-o(1)}}\right),$$
and we can bound the integral over $\mathcal{C}$ 
$$\int_{\mathcal{C}}\abs{\sum_{X\leq n\leq 3X} f(n) n^{-1-it}}^2 dt \ll M(f;X)^2 \exp(-2M(f;X)) + (\log X)^{-\frac{3}{2} + \frac{1}{16}+o(1)}.$$

All together, we conclude that 
$$\int_{0}^T \abs{F(1+it)}^2 dt \ll \exp(-(2-o(1))M(f;X)) + \frac{(\log \log h)^2}{(\log h)^2}$$
in Case II, where $f$ does not satisfy any strong non-pretentiousness hypothesis. A similar calculation recovers this bound for the second integral in \eqref{eq:MRT_parseval}, and hence \eqref{eq:last_error_bound} holds in the case where $M(f;X) \leq \frac{1}{32}\log\log X$ as well.
 
\cref{thm:smooth_f_thm} follows immediately.
\end{proof}
\begin{remark}
Observe that because we only dealt with the sieve error on the frequency side rather than the physical side, we did not need the steps in the argument that deduced \cref{thm:liouville_bound} from \cref{prop:explicit_1_bound}.
\end{remark}

\subsection{Arbitrary multiplicative functions}{\label{sec:4.2}}

While it seems substantially harder to get a bound of the form $$\exp(-(2-o(1))M(f;X)) + \frac{(\log\log h)^2}{(\log h)^2} + \frac{1}{(\log X)^{2-o(1)}},$$
for an arbitrary multiplicative function $f$, one can prove a weaker bound where the error term is $\frac{1}{(\log X)^{1-o(1)}}$ via a suitable variation of the methods presented in \cref{sec:4.1}.
The loss comes from having to choose the final prime range so that the complementary polynomial remains long enough for Hal\'asz' theorem to give useful pointwise savings.

To ensure that $R_{v}(s)$ is not too short, we will need a weaker form of \cref{lem:last_range_sieve_error}.
We let our choice of parameters be exactly the same, except rather than choose $Q$ as in \eqref{eq:PQ_choices}, we take 
$$Q = \frac{X}{\exp\left((\log X)^{\beta}\right)},$$
where $\beta \in (0,1)$ is a parameter to be chosen later.

Recall that when we apply the Ramar\'e identity, the prime-supported $Q_{v}$ can have length at most $Q$. So, $R_{v}$ has length at least $\frac{X}{Q} = \exp\left((\log X)^{\beta}\right)$. To get a bound in terms of $X$ rather than $R_{v}$, we need to compute the maximum difference in the pretentious distance when measured at these two scales. 

\begin{lemma}{\label{lem:multiscale_halasz}}
    Let $0 < \beta  < 1$. Let $X \geq 10$, and for $2\leq Y\leq X$, let $$M(f;Y) = \min_{|t|\leq X}\sum_{p\leq Y}\frac{1-\Re(f(p)p^{-it})}{p} = \min_{|t|\leq X} \mathbb{D}(f, n\mapsto n^{it};Y)^2.$$
    Then we have 
    $$M(f;X) \leq  M(f;Y) + (\log\log X - \log\log Y) + O(1).$$
    In particular if $Y \geq \exp\left((\log X)^{\beta}\right)$, then
    $$M(f;X) - M(f;Y) \leq (1-\beta + o(1))\log\log X.$$
\end{lemma}
\begin{proof}
    Define the truncated pretentious distance 
    $$\mathbb{D}(f, n\mapsto n^{it} ;Y, X)^2 := \sum_{Y \leq p \leq X} \frac{1-\Re(f(p)p^{-it})}{p}.$$
    For all $t$, we can write 
    $$\mathbb{D}(f, n\mapsto n^{it} ; X)^2 = \mathbb{D}(f, n\mapsto n^{it};Y)^2 + \mathbb{D}(f, n\mapsto n^{it} ;Y, X)^2.$$
    Let $t_Y$ be a value of $t$ that attains the minimum in $M(f ; Y)$, and let $U = \frac{1}{(\log Y)}$.
    Look at the interval 
    $$I := [t_Y, t_Y+U]$$
    (for $X$ sufficiently large), and without loss of generality, $I\subseteq [-X, X]$. 
    If not, one would take $[t_Y - U, t_Y]$ and the argument is the same.
    We will prove that there is some $t_I$ inside $I$ for which
    $$\mathbb{D}(f, n\mapsto n^{it_I};X)^2 \leq M(f ; Y) + \sum_{Y\leq p\leq X}\frac{1}{p} + O(1),$$
    and hence certainly $M(f; X)$ is upper bounded by the same.
    We do this by averaging over $I$:
    \begin{align*}
        M(f; X) &\leq \frac{1}{U} \int_{I} \mathbb{D}(f, n\mapsto n^{it} ;X)^2 dt\\
        &= \frac{1}{U} \int_I \mathbb{D}(f, n\mapsto n^{it};Y)^2 dt + \frac{1}{U}\int_I \mathbb{D}(f, n\mapsto n^{it} ; Y, X)^2 dt.
    \end{align*}
    For the first term we compare to the minimizer $t_Y$.
    Let $t = t_Y + u$ for $0\leq u\leq U$.
    Then we have
    \begin{align*}
        \abs{\mathbb{D}(f, n\mapsto n^{it};Y)^2 - \mathbb{D}(f, n\mapsto n^{it_Y};Y)^2} &\leq \sum_{p\leq Y} \frac{|p^{iu}-1|}{p}\\
        &\leq |u|\sum_{p\leq Y}\frac{\log p}{p}\\
        &\ll |u| \log Y, 
    \end{align*}
    where in the second line we used $\abs{\sin x} \leq |x|$ and in the third we used Mertens' theorem.
    Note that we can afford to be loose with constants, as
    $$U \log Y = O(1),$$
    and in the case of interest, when $Y\geq \exp\left((\log X)^{\beta}\right)$, this is much smaller than $\sum_{Y \leq p\leq X}\frac{1}{p}$.
    
    Averaging over $0\leq u\leq U$ gives 
    $$\frac{1}{U} \int_I \mathbb{D}(f, n\mapsto n^{it};Y)^2 \leq M(f;Y) + O\left(U \log Y\right) = M(f;Y) + O(1).$$
    Now for the truncated distance, again writing $t = t_Y+u$, we observe 
    $$\frac{1}{U}\int_I \mathbb{D}(f, n\mapsto n^{it}; Y, X)^2 dt = \sum_{Y < p\leq X}\frac{1}{p} - \Re \left(\sum_{Y < p\leq X} \frac{f(p)p^{-it_Y}}{p} \left(\frac{1}{U}\int_0^U p^{-iu} du \right)\right),$$
    and we can explicitly compute that the integral is $\ll \frac{1}{U\log p}$.
    We use the triangle inequality to bound the numerator by 1, and observe that the entire second term is at most
    $$\frac{1}{U} \sum_{Y\leq p\leq X}\frac{1}{p\log p} \ll \frac{1}{U\log Y} = O(1).$$
    Hence all together, 
    $$M(f;X) \leq M(f;Y) + \sum_{Y < p \leq X} \frac{1}{p} + O(1).$$
\end{proof}

\begin{proof}[Proof (\cref{thm:general_f}).]
Firstly, we may assume that $M(f; X) > \frac{1}{50}\log\log X$, as otherwise the result is a consequence of Theorem A.1 of \cite{average_chowla}.

Again take $\varepsilon = \frac{\kappa}{2}$ and $\eta = \frac{\varepsilon}{4}$. Let $P_1 = (\log Q_1)^{\frac{100}{\eta}}$ and let $Q_1 = \min\left(h, \exp\left(\frac{\log X_0}{\log\log X_0}\right)\right)$. Let $P_j, Q_j$ for $j=2, \dots, J$ be as given in \cref{sec:2.2}, and let $P$ be as in \cref{eq:PQ_choices}. 

We may apply the same Parseval bound as given in \cref{eq:MRT_parseval}, apply \cref{lem:add_back_sieve_error} to account for integers with atypical factorizations, and then bound the integrals over $\mathcal{T}_1, \dots, \mathcal{T}_J$ using the bounds given in \eqref{eq:integral_T_1} and \cref{eq:integral_T_j} (again noting \cref{eq:same_integrals}). As stated earlier, we take $Q = X / \exp\left((\log X)^{\beta}\right)$. 
When we split the Dirichlet polynomial over $\mathcal{U}$ depending on whether or not each $n$ has a prime factor in $[P,Q]$, we now need to apply \cref{lem:abse_2} rather than just the $P$-smooth integer bound given by \cref{lem:last_range_sieve_error}. This gives us a sieve error of
\begin{equation}{\label{eq:modified_sieve_error}}
    \ll (\log X)^{-2+2\beta}.
\end{equation}

For the integral over $\mathcal{U}+u$, the argument from \cref{sec:4.1},
combined with \cref{lem:modified_R_halasz} and \cref{lem:prime_hm}, gives
$$ \int_{\mathcal{U}} |F(1+it)|^2 dt \ll \frac{(\log\log X)^{O(1)}}{\exp(2M(f;R_{v}))}
+ \frac{(\log\log X)^{O(1)}}{(\log R_{v})^2},$$
where $R_{v}\geq \frac{X}{Q}=\exp((\log X)^\beta)$. Additionally, under the assumption that $M(f;X) \geq \frac{1}{50}\log\log X$, we need not consider the pretentious case in the proof of \cref{thm:smooth_f_thm}.

The only difference is that when we apply Hal\'asz' theorem to $R_{v}(s)$ and \cref{lem:prime_hm} to $Q_{v}(s)$, the integral over $\mathcal{U}+u$ will be at most 
$$\ll \frac{(\log\log X)^{O(1)}}{\exp(2M(f; R_{v}))} + \frac{(\log\log X)^{O(1)}}{(\log R_{v})^2}.$$
The estimate is derived exactly as in \eqref{eq:4.1_U_integral}, although Hal\'asz' theorem gives savings at scale $R_{v}$ rather than at scale $X^{\delta}$. By \cref{lem:multiscale_halasz}, we have $$\exp(-2M(f;R_{v})) \ll (\log X)^{2-2\beta+o(1)} \exp(-2M(f;X)).$$

Define the quantity
$$\sigma(f;X) := \min\left( \frac{M(f;X)}{\log\log X}, 1\right),$$
and observe that by hypothesis $\sigma(f;X) \geq \frac{1}{50}$. 
Our savings from Hal\'asz in terms of $\sigma(f;X)$ is at most 
\begin{equation}{\label{eq:halasz_errors}}
    \ll (\log X)^{2-2\beta - 2\sigma(f;X)+o(1)} + (\log X)^{-2\beta + o(1)} + (\log X)^{-1+o(1)},
\end{equation}
where the last error is to account for the case where $\sigma(f;X)=1$. 
Now, ignoring the $o(1)$ factors, we balance the errors arising from \eqref{eq:modified_sieve_error} and \eqref{eq:halasz_errors} by taking $\beta = 1-\frac{\sigma(f;X)}{2}$. 
Then the first Hal\'asz error and the squared sieve error are both
$$\ll \exp(-(1-o(1))M(f;X)),$$
while the second Hal\'asz error is $\ll(\log X)^{-1+o(1)}$.

Combining this with the estimates over
$\mathcal{T}_1,\dots,\mathcal{T}_J$ and restoring the integers outside
$\mathcal{S}$ gives
$$\frac{1}{X}\int_X^{2X} \abs{\frac{1}{h}\sum_{x\leq n\leq x+h} f(n)}^2 dx \ll \exp(-(1-o(1))M(f;X)) + \frac{(\log\log h)^2}{(\log h)^2} +
\frac{1}{(\log X)^{1-o(1)}}.$$
This proves \cref{thm:general_f}.

\end{proof}

\begin{remark}
To see why one should expect a bound of this shape in this case, observe that when we proved \cref{thm:smooth_f_thm}, we chose the parameter $Q$ such that in the sieve error, we essentially had $\beta = 0$, and by restricting to $f$ supported on smooth numbers, we essentially were simultaneously taking $\beta = 1$ in the Hal\'asz error \eqref{eq:halasz_errors}.
Now that we actually have to interpolate between these errors, the best possible case would have $\beta = \frac{1}{2}$, which translates to taking the square root of the sieve upper bound from $[P,Q]$ and our savings from Hal\'asz theorem.
\end{remark}

\section{Applications to the averaged Chowla conjecture}{\label{sec:5}}

One important application of the bounds for multiplicative functions in short intervals was to prove an averaged version of Chowla's conjecture.
The proof proceeds in the following way: one proves an exponential sum estimate for $\lambda(n)$ against the phase $e(\alpha n)$, and then uses a Fourier identity, Cauchy--Schwarz, and van der Corput arguments to deduce the averaged Chowla conjecture.
\cref{thm:liouville_bound} will enable us to prove bounds for both of these theorems, at the limit of what is possible using the Matom\"aki--Radziwi\l{}\l{} method.

One should think of the interval length (in the case of \cref{thm:improved_exp_sum}) or the range of admissible shifts (in the case of \cref{thm:improved_avg_chowla}) as $h \geq \exp\left( (\log X)^{\frac{1}{700}}\right)$, as for smaller values the theorems in \cite{average_chowla} are optimal\footnote{Hence, in this regime, the $h^{\frac{1}{2C_{\kappa}}}$ term in \cref{prop:key_sum} is much larger than $(\log X)^{20}$. In fact, for the theorem to have content, we must have $h\to \infty$ as $X \to\infty$, and when we eventually choose $W = (\log h)^{C}$ for $C$ some fixed constant, $W \ll H^{\frac{1}{2C_{\kappa}}}$ is guaranteed.}.

In particular, we prove the following result (this is the ``key exponential sum estimate'' in \cite{average_chowla}).
For convenience, we redefine the set $\mathcal{S}$ to count $n\leq X$ rather than $n\leq 3X$. 
Furthermore here we will specify that $X_0 = X^{\frac{1}{2}}$ in the definition of $\mathcal{S}$. In terms of quantitative bounds these changes are purely cosmetic. 

Let $C_{\kappa}$ be a constant satisfying
$$C_{\kappa}\frac{\kappa}{24} - \frac{1}{15} > \frac{5}{2}.$$ For concreteness, one can take $C_{\kappa} = 10^{10}$.

\begin{proposition}[Key exponential sum estimate]{\label{prop:key_sum}}
    Let $X \geq h \geq 10$ and let $W \geq 10$ be such that 
    $$(\log h)^5 \leq W \leq \min\{ h^{\frac{1}{2C_{\kappa}}}, (\log X)^{20}\}.$$ 
    Let
    $$h_0 = \exp\left(\frac{\log X_0}{\log\log X_0}\right),$$
    and let $\widetilde{h} = \min(h, h_0)$.
    In the definition of $\mathcal{S}$, take $P_1 = W^{C_{\kappa}}$ and $Q_1 = \widetilde{h}/W^3$. Then uniformly for $\alpha \in \mathbb{T}$
    $$\frac{1}{X}\int_{\R} \abs{\frac{1}{h}\sum_{\substack{x\leq n \leq x+h \\ n\in \mathcal{S}}} \lambda(n) e(\alpha n)} dx \ll \frac{(\log h)^{\frac{1}{4}} \log\log h}{W^{\frac{1}{4}}}.$$
\end{proposition}
In fact, we may assume $h$ sufficiently large so that $(\log h)^5 \leq h^{\frac{1}{2C_{\kappa}}}$ and the $W$ range is nonempty (else \cref{thm:improved_exp_sum} and \cref{thm:improved_avg_chowla} are trivial as stated). 

The proof applies the circle method on $\alpha$, involving casework on whether $\alpha$ is close to a rational $\frac{a}{q}$, where $q \leq W$ places $\alpha$ in a major arc and otherwise $\alpha$ is in a minor arc.
The quantitative bound here is the same as the one stated by Matom\"aki--Radziwi\l{}\l{}--Tao, but the admissible range of $W$ (the major arc threshold) is increased from $(\log X)^{\frac{1}{125}}$ to $(\log X)^{20}$.
The upshot is that we no longer have to do any casework on $h$ in the deduction of \cref{thm:improved_exp_sum}.
Any $h\leq X$ satisfies $(\log h)^5 \leq (\log X)^{20}$ and therefore when we apply this, we may always take $W = (\log h)^5$. 

Here, $\mathcal{S}$ is essentially the same, although we take $X_0 = X^{\frac{1}{2}}$, which changes nothing in the quantitative aspects.
When we ultimately use this to prove \cref{thm:improved_exp_sum} and \cref{thm:improved_avg_chowla}, we will take $W = (\log h)^5$ and $W = (\log H)^{20}$ respectively.
Hence, the structure of $P_1$ and $Q_1$ remains (basically) the same; we will still have logarithmic separation between $P_1$ and $Q_1$, and up to constants $\log P_1$ and $\log Q_1$ are basically the same as in the previous sections. 

\begin{proof}
We know via Dirichlet's theorem that there are coprime $a,q$ with $1\leq q\leq \frac{h}{W}$, such that
$$\abs{\alpha-\frac{a}{q}} \leq \frac{W}{hq}.$$

We will define major arcs and minor arcs depending on whether or not $q \leq W$. 
The structure of the proof will be \emph{exactly} the same as previous work, except at a crucial point in the major arc bounds, at which we gain extra flexibility with the choice of $W$ due to our improved short interval bounds.
In fact, we can leave the minor arc bounds completely unchanged, and the exact same argument as in Section 3 of \cite{average_chowla} works.

If $q > W$ then 
$$\frac{1}{X}\int_{\R} \theta(x) \abs{\frac{1}{h}\sum_{\substack{x\leq n\leq x+h\\ n\in \mathcal{S}}} \lambda(n) e(\alpha n)} dx \ll \frac{(\log h)^{\frac{1}{4}} \log\log h}{W^{\frac{1}{4}}},$$
where $\theta:\R\to \C$ is measurable, 1-bounded, and supported on $[0,X]$. This is precisely equation (3.1) in \cite{average_chowla}.

Hence we focus on bounding the exponential sum in \cref{prop:key_sum} in the case where $\alpha$ is on the major arcs.
The shape of the bound on the right-hand side of \cref{prop:key_sum} actually comes from the minor arcs.
In the major arc case, we will prove that
\begin{equation}{\label{eq:major_arc_exp}}
    \int_{\R} \abs{\sum_{\substack{x\leq n\leq x+h \\ n\in \mathcal{S}}} \lambda(n)e(\alpha n)} dx \ll \frac{hX}{W^{\frac{1}{4}}}.
\end{equation}
By hypothesis, $\alpha = \frac{a}{q} + \theta$ with $\theta = O\left(\frac{W}{hq}\right)$.
By partial summation, we have
\begin{align}{\label{eq:ps}}
\abs{ \sum_{\substack{x \leq n \leq x+h \\ n\in\mathcal{S}}} \lambda(n) e(\alpha n)} & \ll \abs{\sum_{\substack{x \leq n \leq x+h\\ n\in\mathcal{S}}} \lambda(n) e(an / q)} + \frac{W}{h q} \int_{0}^{h} \abs{ \sum_{\substack{x \leq n \leq x + h'\\n\in \mathcal{S}}}
\lambda(n) e(an/q)} dh'.
\end{align}

We firstly handle the cases where $h'$ is too small to apply \cref{cor:short_interval_major_arc}.
Suppose $0\leq h'\leq qQ_1$.
We may apply the trivial bound
\begin{align*}
    \int_{\R} \abs{\sum_{\substack{x\leq n\leq x+h'\\n\in \mathcal{S}}} \lambda(n) e(an/q)}dx &\ll h' X.
\end{align*}
Hence, this range of $h'$ contributes
$$\frac{W}{hq} \int_0^{qQ_1} h' X dh' \ll \frac{Wq Q_1^2}{h^2} hX \ll \frac{1}{W^4} hX,$$
where we have used $q\leq W$ and $Q_1 \leq \frac{h}{W^3}$. This bound is more than enough, and furthermore we have that for all remaining $h'$, we have $h' > qQ_1$, in which case, for any $d\leq q$ we have $\frac{h'}{d} \geq \frac{h'}{q} > Q_1$ and thus we can apply \cref{cor:short_interval_major_arc} with interval length $\frac{h'}{d}$.

Now, exactly as in \cite{average_chowla}, it will be enough to bound 
$$\int_{\R} \abs{ \sum_{\substack{x \leq n \leq x + h'\\n\in \mathcal{S}}} \lambda(n) e(a n / q)} dx,$$
again with $h' > qQ_1$. 
To do this, we split into residue classes, and for each residue class $b$ modulo $q$, we reduce modulo $\gcd(b,q)$. 
This yields
$$\int_{\R} \abs{ \sum_{\substack{x \leq n \leq x + h'\\n\in \mathcal{S}}} \lambda(n) e(an/q)} dx \leq \sum_{b \pmod{q}} \int_{\R} \abs{ \sum_{\substack{x\leq n\leq x+h' \\ n\in\mathcal{S} \\ n\equiv b \pmod{q}}} \lambda(n)} dx,$$
and let $d = \gcd(b,q)$ with $b = d b_0$, $q = d q_0$, and $n = d m$. 
This way, $n\equiv b \pmod{q}$ simplifies to $m\equiv b_0 \pmod{q_0}$, with $\gcd(b_0, q_0)=1$.
Note that $d \leq q\leq W\leq P_1$, so we can factor $d$ out of $n$ without causing problems for the typical factorization set.
The only difference will be that in what remains, the set $\mathcal{S}$ has scale $\frac{X}{d}$ rather than just $X$ (this is why we needed the additional flexibility with $X_0$).
We have 
$$\1_{\mathcal{S}}(n)\lambda(n) = \lambda(d) \1_{\mathcal{S}_{P_1, Q_1, X_0, X/d}}(m) \lambda(m),$$
as in \cite{average_chowla}.
Denote this new set of integers with typical factorization by $\widetilde{\mathcal{S}}$.
As stated earlier though, since $d$ is very small, we have not changed anything quantitative about the bounds coming from this new set $\widetilde{\mathcal{S}}$.
We now rewrite the congruence condition modulo $q_0$ in terms of Dirichlet characters, so after applying that identity and using the triangle inequality we have
\begin{align*}
    \sum_{b \pmod{q}} \int_{\R} \abs{ \sum_{\substack{x\leq n\leq x+h' \\ n\in\mathcal{S} \\ n\equiv b \pmod{q}}} \lambda(n)} dx &\leq \sum_{b \pmod{q}}  \frac{1}{\phi(q_0)} \sum_{\chi \pmod{q_0}} \int_{\R} \abs{\sum_{\substack{x/d \leq m \leq (x+h')/d \\ m\in \widetilde{\mathcal{S}}}} \overline{\chi(m)}\lambda(m)} dx \\
    &\leq \sum_{b \pmod{q}}  \frac{d}{\phi(q_0)} \sum_{\chi \pmod{q_0}} \int_{\R} \abs{\sum_{\substack{y \leq m \leq y+h'/d \\ m\in \widetilde{\mathcal{S}}}} \overline{\chi(m)}\lambda(m)} dy,
\end{align*}
with the change of variables $y = x/d$.
For the part of the integral where $y\leq \frac{X}{W^{10}}$, bound everything trivially by
$$\ll q^2 \frac{X}{W^{10}} \frac{h'}{d} \ll \frac{hX}{W^6}$$
which is more than enough.
In the remaining region $\frac{X}{W^{10}} \leq y\leq \frac{2X}{d}$, we split this into dyadic blocks of length $X'$.
The previous expression involving the sum over residues $b$ is then at most 
$$\ll \sum_{\substack{\frac{X}{W^{10}} < X' < \frac{2X}{d}\\X' \textup{ dyadic}}} \sum_{b\pmod{q}}  \frac{d}{\phi(q_0)} \sum_{\chi \pmod{q_0}} \int_{X'}^{2X'} \abs{\sum_{\substack{y \leq m \leq y+h'/d \\ m\in \widetilde{\mathcal{S}}}} \overline{\chi(m)}\lambda(m)} dy + \frac{hX}{W^6}.$$
This integral (after squaring the integrand, which will require some Cauchy--Schwarz) is precisely what \cref{cor:short_interval_major_arc} is equipped to handle.
We know that 
$$\int_{X'}^{2X'} \abs{\sum_{\substack{y \leq m \leq y+h'/d \\ m\in \widetilde{\mathcal{S}}}} \overline{\chi(m)}\lambda(m)}^2 dy \ll \left(\frac{(\log Q_1)^{\frac{1}{3}}}{P_1^{\frac{\varepsilon}{3} - \eta}} + (\log X')^{-A}\right) \frac{(h')^2X'}{d^2},$$
and since $P_1 = W^{C_{\kappa}}$, with $\varepsilon = \frac{\kappa}{2}$ and $\eta = \frac{\kappa}{8}$, the denominator is $W^{C_{\kappa}\frac{\kappa}{24}}$.
Then since $\frac{h'}{d} \leq h$, and $W \geq (\log h)^5$, we know that the numerator is at most $W^{\frac{1}{15}}$. $C_{\kappa}\frac{\kappa}{24}-\frac{1}{15} > \frac{5}{2}$, by choice of $C_{\kappa}$, so the first error term is $\ll \frac{1}{W^{\frac{5}{2}}}$.

The crucial improvement here is that we may choose $A$ to be large here.
The previous limitation that 
$$\frac{1}{(\log X)^{\frac{1}{50}}} \ll \frac{1}{W^{\frac{5}{2}}}$$ caps the major arc threshold at $(\log X)^{\frac{1}{125}}$, which propagates to these errors of $(\log X)^c$ for $c = \frac{1}{700}$ and later $c = \frac{1}{3000}$.
Here, by taking $A$ to be, say, 100, we can take $W$ all the way up to $(\log X)^{20}$.
Having applied our improved bounds, the remainder of the proof concludes exactly as in \cite{average_chowla}.
We have the $W^{-\frac{5}{2}}$ bound, so by Cauchy--Schwarz, 
$$\int_{X'}^{2X'}\abs{\sum_{\substack{y \leq m \leq y+h'/d \\ m\in \widetilde{\mathcal{S}}}} \overline{\chi(m)}\lambda(m)} dy \ll W^{-\frac{5}{4}} \frac{h' X'}{d},$$
and so when we sum over $X', b, \chi$, we see that 
$$\int_{\R} \abs{ \sum_{\substack{x \leq n \leq x + h'\\n\in \mathcal{S}}} \lambda(n) e(a n / q)} dx \ll \frac{q h' X}{W^{\frac{5}{4}}}.$$
When $h'=h$ this gives the endpoint bound in \eqref{eq:ps}.
For the remaining integral over $h'$, we plug the above bound into the integrand in the second term of \eqref{eq:ps} and applying $q\leq W$, we have 
\begin{align*}
    \frac{W}{hq}\int_0^{qQ_1} h' X dh' + \frac{W}{hq}\int_{qQ_1}^h \frac{qh'X}{W^{\frac{5}{4}}}dh' &\ll \frac{hX}{W^4} + \frac{X}{h}\frac{1}{W^{\frac{1}{4}}} \left(h^2 - q^2Q_1^2\right)\\
    &\ll \frac{hX}{W^{1/4}} + \frac{hX}{W^4}.
\end{align*}

This gives the major arc bound \eqref{eq:major_arc_exp} and hence together with the minor arc bound from \cite{average_chowla} completes the proof.
\end{proof}

\begin{proof}[Proof (of \cref{thm:improved_exp_sum}).]
Take $W=(\log h)^5$ in \cref{prop:key_sum}. This gives
$$ \sup_{\alpha\in\mathbb T} \frac{1}{X}\int_{\mathbb R} \abs{\frac{1}{h} \sum_{\substack{x\leq n\leq x+h\\ n\in \mathcal{S}}} \lambda(n)e(\alpha n)}dx \ll \frac{\log\log h}{\log h}.
$$
Firstly suppose $h \leq h_0 = \exp\left(\frac{\log X_0}{\log\log X_0}\right)$. Then by \cref{prop:sieve_error} and the choices $P_1=W^{C_{\kappa}}$ and $Q_1=h/W^3$, we have
$\#\{n\leq 2X:n\notin \mathcal{S}\} \ll X\frac{\log P_1}{\log Q_1} \ll X\frac{\log\log h}{\log h}$.
Hence, by Fubini and the triangle inequality,
$$ \frac{1}{X}\int_{\mathbb R} \abs{\frac1{h}\sum_{\substack{x\leq n\leq x+h\\ n\notin \mathcal{S}}}
 \lambda(n)e(\alpha n)}dx
 \ll \frac{\log\log h}{\log h}.
$$
As in the case of \cref{thm:liouville_bound}, in the event that $h > \exp\left(\frac{\log X_0}{\log\log X_0}\right)$, we take $Q_1 = \frac{h_0}{W^3}$, and we obtain an upper bound of $\frac{(\log\log X)^2}{\log X}$.
Again, we take the sum to state a bound that is uniform in $h$. 

Adding the estimates for the integrals where $n\in\mathcal{S}$ and $n\notin \mathcal{S}$ gives \cref{thm:improved_exp_sum}. 
\end{proof}

Having proved \cref{prop:key_sum}, the proof of \cref{thm:improved_avg_chowla} is exactly the same as in the original Matom\"aki--Radziwi\l{}\l{}--Tao paper.
The only difference is that we need not split into cases involving $H < H_0$ and $H \geq H_0$ due to some limitation on $W$. The only casework comes from \cref{cor:short_interval_major_arc} itself, in which we may require casework on $Q_1$. 
We sketch the main ideas (and refer the reader to \cite{average_chowla} for a full proof), but the key improvement is that we can take $W = (\log H)^{20}$ because for $H\leq X$, $(\log H)^{20} \leq (\log X)^{20}$.

\begin{proof}[Proof sketch (\cref{thm:improved_avg_chowla}).]
The main tool that enables one to deduce \cref{thm:improved_avg_chowla} from a bound like \cref{prop:key_sum} is the following Fourier analytic identity (see Lemma 1.4 of \cite{average_chowla}):  for $f:\Z\to \C$ finitely supported and $H > 0$, we have
$$\int_{\mathbb{T}}\left(\int_{\R}\abs{\sum_{x\leq n\leq x+H}f(n) e(\alpha n)}^2 dx \right)^2 d\alpha = \sum_{|h|\leq H}(H-|h|)^2\abs{\sum_n f(n) \overline{f(n+h)}}^2.$$
We apply this with $2H$ instead of $H$, and $f(n) = \lambda(n) \1_{\mathcal{S}}(n) \1_{[1, X+2H]}(n)$.
\cref{prop:key_sum} gives
$$\sup_{\alpha\in\mathbb T}\int_{\mathbb R} \abs{\sum_{x\leq n\leq x+2H} f(n)e(\alpha n)}^2dx \ll \frac{(\log H)^{\frac{1}{4}}\log\log H}{W^{\frac{1}{4}}}H^2X. $$
Consequently, when we drop the $H < |h| \leq 2H$ terms,
$$\sum_{|h|\leq H} \abs{\sum_n \1_{\mathcal{S}}(n)\lambda(n)\1_{\mathcal{S}}(n+h)\lambda(n+h)}^2 \ll \frac{(\log H)^{\frac{1}{4}}\log\log H}{W^{\frac{1}{4}}}HX^2.
$$
This is the two-point correlation input used in the van der Corput
argument of Matom\"aki--Radziwi\l{}\l{}--Tao. Repeating that argument with the improved admissible range for $W$ gives the
truncated averaged Chowla estimate
$$\sum_{1\leq h_2,\ldots,h_k\leq H}\abs{\sum_{1\leq n\leq X} \1_{\mathcal{S}}(n)\lambda(n) \prod_{j=2}^k \1_{\mathcal{S}}(n+h_j)\lambda(n+h_j)} \ll k\frac{H^{k-1}X}{W^{\frac{1}{20}}}.$$
Lastly, the contribution of the terms in which at least one of
$n,n+h_2,\ldots,n+h_k$ is not in $\mathcal{S}$ is, again by \cref{prop:sieve_error}, casework on whether or not $H\leq \exp\left(\frac{\log X_0}{\log\log X_0}\right)$, and the
triangle inequality,
\begin{align*}
    &\ll kH^{k-1}X\frac{\log P_1}{\log Q_1}\\
    &\ll kH^{k-1}X \left(\frac{\log\log H}{\log H} + \frac{(\log\log X)^2}{\log X}\right)
\end{align*}
Choosing $W = (\log H)^{20}$, we conclude
$$\sum_{1\leq h_2,\ldots,h_k\leq H} \abs{\sum_{1\leq n\leq X}\lambda(n)\lambda(n+h_2)\cdots\lambda(n+h_k)} \ll k H^{k-1}X\left(\frac{\log\log H}{\log H} + \frac{(\log\log X)^2}{\log X}\right).
$$
Dividing by $H^{k-1}X$ gives \cref{thm:improved_avg_chowla}.
\end{proof}

\subsection*{Acknowledgements}

The author is very grateful to his advisor, Joni Ter\"av\"ainen, for suggesting the problem and for many fruitful discussions and helpful comments on earlier drafts of this paper.  The author was supported by funding from the European Union's Horizon Europe research and innovation programme under ERC grant agreement no. 101162746. 

\bibliographystyle{alpha}
\bibliography{references}  

\end{document}